\documentclass[10pt]{amsart}

\usepackage{amsmath,amssymb,amscd,amsthm,amsxtra}

\usepackage{graphicx, mathrsfs}
\usepackage{amssymb,amsmath,amsthm}
\usepackage{mathrsfs,dsfont}

\newcommand{\ang}[1]{ \left< {#1} \right>}

\renewcommand{\mid}{{\mathrm{mid}}}

\makeatletter
\DeclareRobustCommand\widecheck[1]{{\mathpalette\@widecheck{#1}}}
\def\@widecheck#1#2{%
   \box\z@\hbox{\m@th$#1#2$}%
   \box\tw@\hbox{\m@th$#1%
      \widehat{%
         \vrule\@width\z@\@height\ht\z@
         \vrule\@height\z@\@width\wd\z@}$}%
   \dp\tw@-\ht\z@
   \@tempdima\ht\z@ \advance\@tempdima2\ht\tw@ \divide\@tempdima\thr@@
   \box\tw@\hbox{%
      \raise\@tempdima\hbox{\scalebox{1}[-1]{\lower\@tempdima\box\tw@}}}%
   {\ooalign{\box\tw@ \cr \box\z@}}}
\makeatother

\newtheorem{theorem}{Theorem} [section]

\newtheorem{lemma}[theorem]{Lemma}
\newtheorem{proposition}[theorem]{Proposition}
\newtheorem{remark}[theorem]{Remark}

\newtheorem{definition}[theorem]{Definition}
\newtheorem{corollary}[theorem]{Corollary}

\begin{document}

\title[Uniqueness of solutions to the periodic 3D GP hierarchy]{On the uniqueness of solutions to the periodic 3D Gross-Pitaevskii hierarchy}

\author[P. Gressman]{Philip Gressman}
\address{
University of Pennsylvania, Department of Mathematics, David Rittenhouse Lab, 209 South 33rd Street, Philadelphia, PA 19104-6395, USA}
\email{gressman@math.upenn.edu}
\urladdr{http://www.math.upenn.edu/~gressman/}
\thanks{P. G. was partially supported by NSF Grant DMS-1101393 and an Alfred P. Sloan Research Fellowship.}

\author[V. Sohinger]{Vedran Sohinger}
\address{
University of Pennsylvania, Department of Mathematics, David Rittenhouse Lab, 209 South 33rd Street, Philadelphia, PA 19104-6395, USA}
\email{vedranso@math.upenn.edu}
\urladdr{http://www.math.upenn.edu/~vedranso/}
\thanks{V. S. was supported by a Simons Postdoctoral Fellowship.}

\author[G. Staffilani]{Gigliola Staffilani}
\address{
Massachusetts Institute of Technology, Department of Mathematics, Building 2, 77 Massachusetts Avenue, Cambridge, MA 01239-4301, USA}
\email{gigliola@math.mit.edu}
\urladdr{http://math.mit.edu/~gigliola/}
\thanks{G.  S. was partially supported by NSF Grant DMS-1068815.}

\begin{abstract}
In this paper, we present a uniqueness result for solutions to the Gross-Pitaevskii hierarchy on the three-dimensional torus, under the assumption of an a priori spacetime bound. We show that this a priori bound is satisfied for factorized solutions to the hierarchy which come from solutions of the nonlinear Schr\"{o}dinger equation. In this way, we obtain a periodic analogue of the uniqueness result on $\mathbb{R}^3$ previously proved by Klainerman and Machedon \cite{KM}, except that, in the periodic setting, we need to assume additional regularity. In particular, we need to work in the Sobolev class $H^{\alpha}$ for $\alpha>1$. By constructing a specific counterexample, we show that, on $\mathbb{T}^3$, the existing techniques from the work of Klainerman and Machedon approach do not apply in the endpoint case $\alpha=1$. This is in contrast to the known results in the non-periodic setting, where the these techniques are known to hold for all $\alpha \geq 1$.

In our analysis, we give a detailed study of the crucial spacetime estimate associated to the free evolution operator. In this step of the proof, our methods rely on lattice point counting techniques based on the concept of the determinant of a lattice. This method allows us to obtain improved bounds on the number of lattice points which lie in the intersection of a plane and a set of radius $R$, depending on the number-theoretic properties of the normal vector to the plane. We are hence able to obtain a sharp range of admissible Sobolev exponents for which the spacetime estimate holds.
\end{abstract}
\keywords{Gross-Pitaevskii hierarchy, Nonlinear Schr\"{o}dinger equation,
BBGKY hierarchy, Bose-Einstein condensation, determinant of a lattice, $U$ and $V$ spaces, factorized solutions, multilinear estimates}
\subjclass[2010]{11P21,35Q55,70E55}
\maketitle
\section{Introduction}
\label{Introduction}
\subsection{Setup of the problem}
Let us fix $\Lambda=\mathbb{R}^d$ or $\Lambda=\mathbb{T}^d$.
We are considering the \emph{Gross-Pitaevskii hierarchy}, which is a sequence of functions $\gamma^{(k)}: \mathbb{R} \times \Lambda^k \times \Lambda^k \rightarrow \mathbb{C}$
satisfying the symmetry properties below for all $t \in \mathbb{R}$. Let $\vec{x}_k:=(x_1,\dots,x_k)$ and $\vec{x}_k':=(x_1',\dots,x_k;)$ where $x_j,x_j'\in \Lambda$ for $ j=1,\dots,k.$ We suppose that:
\begin{equation}
\label{eq:conditionsongamma}
\begin{cases}
\gamma^{(k)}(t,\vec{x}_k;\vec{x}_k')= \overline{\gamma^{(k)}(t,\vec{x}_k';\vec{x}_k)}\\
\gamma^{(k)}(t,x_{\sigma(1)},\ldots,x_{\sigma(k)};x_{\sigma(1)}',\ldots,x_{\sigma(k)}')=\gamma^{(k)}(t,x_1,\ldots,x_k;x_1',\ldots,x_k')\,\,
\end{cases}
\end{equation}
$\forall \,\, \sigma \in S_k.$ Furthermore, we suppose that  the $\gamma^{(k)}$ solve the following infinite system of linear equations:
\begin{equation}
\label{eq:GPhierarchy}
\begin{cases}
i \partial_t \gamma^{(k)} + (\Delta_{\vec{x}_k}-\Delta_{\vec{x}_k'})\gamma^{(k)}= b_0 \cdot \sum_{j=1}^{k}B_{j,k+1}(\gamma^{(k+1)})\\
\gamma^{(k)}|_{t=0}=\gamma_0^{(k)}.
\end{cases}
\end{equation}
Here, $b_0 \in \mathbb{R}$ is a coupling constant and $\Delta_{\vec{x}_k}:=\sum_{j=1}^k \Delta_{x_j},\Delta_{\vec{x}_k'}:=\sum_{j=1}^k \Delta_{x_j'}$.

In order to define the \emph{collision operator} $B_{j,k+1}$, we let $Tr_{\ell}$ denote the trace in the $x_{\ell}$ variable.
For $k \in \mathbb{N}$, $j \in \{1,\ldots,k\}$, and for $\gamma_0^{(k+1)}:\Lambda^{k+1} \times \Lambda^{k+1} \rightarrow \mathbb{C}$ we define:
\begin{equation}
\label{collisionoperator}
B_{j,k+1}(\gamma_0^{(k+1)})(\vec{x}_k;\vec{x}_k'):=Tr_{k+1}\big[\delta(x_j-x_{k+1}),\gamma_0^{(k+1)}\big](\vec{x}_k;\vec{x}_k')
\end{equation}
$$=\int_{\Lambda} dx_{k+1}\Big(\delta(x_j-x_{k+1})-\delta(x_j'-x_{k+1})\Big) \gamma_0^{(k+1)}(\vec{x}_k,x_{k+1};\vec{x}_k',x_{k+1}).
$$
In the above definition, we note that we are identifying operators on $L^2(\Lambda^{k+1})$ with their kernels.
For $j,k$, and $\gamma_0^{(k+1)}$ as above, we also define the operators $B^{+}_{j,k+1}$ and $B^{-}_{j,k+1}$:
\begin{equation}
\label{Bjk+}
B^{+}_{j,k+1}(\gamma_0^{(k+1)})(\vec{x}_k;\vec{x}_k'):=\int_{\Lambda} dx_{k+1} \, \delta(x_j-x_{k+1}) \gamma_0^{(k+1)}(\vec{x}_k,x_{k+1};\vec{x}_k',x_{k+1}).
\end{equation}
\begin{equation}
\label{Bjk-}
B^{-}_{j,k+1}(\gamma_0^{(k+1)})(\vec{x}_k;\vec{x}_k'):=\int_{\Lambda} dx_{k+1} \, \delta(x_j'-x_{k+1}) \gamma_0^{(k+1)}(\vec{x}_k,x_{k+1};\vec{x}_k',x_{k+1}).
\end{equation}
By construction:
\begin{equation}
\label{Bjk}
B_{j,k+1}=B^{+}_{j,k+1}-B^{-}_{j,k+1}.
\end{equation}
A special case of $(\ref{eq:GPhierarchy})$ which we will study is the Gross-Pitaevskii hierarchy with zero initial data:
\begin{equation}
\label{eq:homogeneousGPhierarchy}
\begin{cases}
i \partial_t \gamma^{(k)} + (\Delta_{\vec{x}_k}-\Delta_{\vec{x}_k'})\gamma^{(k)}= b_0 \cdot \sum_{j=1}^{k}B_{j,k+1}(\gamma^{(k+1)})\\
\gamma^{(k)}|_{t=0}=0.
\end{cases}
\end{equation}
In this paper, we will restrict our attention to the \emph{defocusing} Gross-Pitaevskii hierarchy, in which the coupling constant $b_0$ is positive. For simplicity of notation, we will set $b_0=1$.
\begin{remark}
The functions $\gamma_0^{(k)}: \Lambda^k \times \Lambda^k \rightarrow \mathbb{C}$ are $k$-particle density matrices or density matrices of order $k$. We will also assume, without explicit mention, that all of the $k$-particle density matrices in this paper satisfy the symmetry conditions $(\ref{eq:conditionsongamma})$.
\end{remark}
\subsection{Statement of main result}
The study of the Gross-Pitaevskii hierarchy $(\ref{eq:GPhierarchy})$ makes both mathematical and physical sense for the spatial domain $\Lambda=\mathbb{R}^d$ and $\Lambda=\mathbb{T}^d$. We will elaborate more on the physical interpretation in the following subsection. Throughout this paper, we will fix the spatial domain to be $\Lambda:=\mathbb{T}^3$.
Let us fix $\alpha \in \mathbb{R}$.
For $\gamma_0^{(k)}: \Lambda^k \times \Lambda^k \rightarrow \mathbb{C}$, we define the fractional differential operator $S^{(k,\alpha)}.$
\begin{equation}
\label{eq:Ska}
S^{(k,\alpha)} \gamma_0^{(k)}:= \prod_{j=1}^k (1-\Delta_{x_j})^{\frac{\alpha}{2}} (1-\Delta_{x'_j})^{\frac{\alpha}{2}} \gamma_0^{(k)}.
\end{equation}
In other words, $S^{(k,\alpha)}$ acts on the density matrix $\gamma_0^{(k)}$ by taking $\alpha$ fractional derivatives on each of the components.
With the above notation in mind, we suppose that we are working in the class $\mathcal{A}$ of sequences $\Gamma(t)=(\gamma^{(k)}(t))$, where each $\gamma^{(k)}: \mathbb{R} \times \Lambda^k \times \Lambda^k \rightarrow \mathbb{C}$ satisfies $(\ref{eq:conditionsongamma})$, and such that there exist continuous, positive functions $\sigma, f: \mathbb{R} \rightarrow \mathbb{R}$ with the property that for all $k \in \mathbb{N}$, $j=1,\ldots,k$, $t \in \mathbb{R}$:
\begin{equation}
\label{eq:aprioribound}
\int_{t-\sigma(t)}^{t+\sigma(t)} \|S^{(k,\alpha)}B_{j,k+1}(\gamma^{(k+1)})(s)\|_{L^2(\Lambda^k \times \Lambda^k)} ds \leq f^{k+1}(t).
\end{equation}
Our main result is:
\begin{theorem}
\label{Theorem 1}
Suppose that $\Gamma(t)=(\gamma^{(k)}(t))$ is a solution to the homogeneous Gross-Pitaevskii hierarchy $(\ref{eq:homogeneousGPhierarchy})$ which belongs to the class $\mathcal{A}$ for $\alpha>1$. Then, $\Gamma(t)$ is identically zero.
In particular, solutions to the Gross-Pitaevskii hierarchy $(\ref{eq:GPhierarchy})$ are unique in the class $\mathcal{A}$ whenever $\alpha>1$.
\end{theorem}
In order to deduce the second part of the statement of Theorem \ref{Theorem 1} from the first one, we use the fact that the Gross-Pitaevskii hierarchy is linear, and the observation that the class $\mathcal{A}$ is closed under taking differences. The latter claim follows from the triangle inequality, and the fact that $f_1^{k+1}(t)+f_2^{k+1}(t) \leq (f_1(t)+f_2(t))^{k+1}$.
The following result shows that the class $\mathcal{A}$ is non-empty, and that it contains the physically relevant \emph{factorized} solutions of the Gross-Pitaevskii hierarchy.
Suppose that $\phi$ solves the cubic nonlinear Schr\"{o}dinger equation on $\Lambda$:
\begin{equation}
\label{eq:NLS}
\begin{cases}
i \partial_t \phi + \Delta \phi = |\phi|^2\phi,\,\,\mbox{on}\,\,\mathbb{R}_t \times \Lambda\\
\phi|_{t=0}=\phi_0,\,\mbox{on}\,\Lambda.
\end{cases}
\end{equation}
Let $| \cdot \rangle \langle \cdot|$ denote the \emph{Dirac bracket}, which is defined as $|f \rangle \langle g| (x,x'):=f(x) \overline{g(x')}$.
We define:
\begin{equation}
\label{eq:factorizedsolution}
\gamma^{(k)}(t,\vec{x}_k;\vec{x}_k'):=\prod_{j=1}^k \phi(t,x_j) \overline{\phi(t,x_j')}=|\phi \rangle \langle \phi|^{\otimes k}(t,\vec{x}_k;\vec{x}_k').
\end{equation}
Then $(\gamma^{(k)}(t))$ solves the Gross-Pitaevskii hierarchy with initial data $(\gamma_0^{(k)})=(|\phi_0 \rangle \langle \phi_0|^{\otimes k})$. These are defined to be the \emph{factorized solutions}. It can be shown that if $(|\phi \rangle \langle \phi|^{\otimes k}(t))$ solves the Gross-Pitaevskii hierarchy with the initial data $(|\phi_0 \rangle \langle \phi_0|^{\otimes k})$, then $\phi$ necessarily solves the NLS equation. Hence, it makes sense to define $(\gamma^{(k)}(t))$ to be the \emph{factorized solution evolving from $(|\phi_0 \rangle \langle \phi_0|^{\otimes k})$}. We can prove the following result:
\begin{theorem}
\label{Theorem 2} Suppose that $\alpha \geq 1$ and that $\phi_0 \in H^{\alpha}(\Lambda)$. Then, the factorized solution $\Gamma(t)= (|\phi \rangle \langle \phi|^{\otimes k}(t))$ evolving from initial data $(|\phi_0 \rangle \langle \phi_0|^{\otimes k})$ belongs to the class $\mathcal{A}$.
\end{theorem}
From Theorem \ref{Theorem 2}, it follows that, for $\alpha \geq 1$, the class $\mathcal{A}$ is nonempty, and that it contains the the solutions generated by the phenomenological mean-field equation, namely the NLS \eqref{eq:NLS}.
\begin{remark}
The key estimate of Theorem \ref{Theorem 1} fails at the endpoint $\alpha=1$. This is in sharp contrast to the work by Klainerman and Machedon on $\mathbb{R}^3$ \cite{KM}, in which the authors show that uniqueness holds for $\alpha \geq 1$. The case $\alpha=1$ is important since it corresponds to finite $k$-particle kinetic energy. In Subsection \ref{alpha1discussion}, we give a detailed explanation why none of the existing techniques in the Klainerman-Machedon argument apply to the problem of uniqueness in the endpoint case when the spatial domain is $\mathbb{T}^3$. In particular, in Proposition \ref{Alpha1unbounded}, we obtain that the key spacetime estimate for the free evolution does not hold when $\alpha=1$. Heuristically, the fact that the presented method holds for a smaller range of exponents in the case of $\mathbb{T}^3$ is a manifestation of the fact that dispersion is weaker on periodic domains. It is an interesting open problem to obtain uniqueness when $\alpha=1$ on $\mathbb{T}^3$. In fact if one obtains a $\mathbb{T}^3$ uniqueness result to the hierarchy \eqref{eq:homogeneousGPhierarchy} when $\alpha=1$, then together with \cite{EESY, ESY1}, one would have derived the $\mathbb{T}^3$ cubic NLS from quantum many-body dynamics.
\end{remark}

\begin{remark}
The a priori bound $(\ref{eq:aprioribound})$ corresponds to the bound on $\mathbb{R}^3$ used in the work of Klainerman and Machedon \cite{KM}. In \cite{KM}, the upper bound is proved for factorized solutions evolving from $(|\phi_0 \rangle \langle \phi_0|^{\otimes k})$, with $\phi_0 \in H^{\alpha}(\mathbb{R}^3)$, whenever $\alpha=1$. Moreover, in the case $\alpha=1$, the function $f$ can be chosen to be independent of $t$. The latter fact uses conservation of energy for the nonlinear Schr\"{o}dinger equation, which gives an a priori bound on the $H^1$ norm of the solution. There is no such known conservation law which would give a uniform bound on the $H^{\alpha}$ norm of the solution when $\alpha=1+$. According to Theorem \ref{Theorem 2}, in the periodic case, we can still consider factorized solutions when $\alpha=1$, but the uniqueness statement in Theorem \ref{Theorem 1} requires us to consider the regularity $\alpha>1$.
\end{remark}

\begin{remark}
\label{mathcalB}
A different non-empty class of local-in-time solutions to the Gross-Pitaevskii hierarchy $(\ref{eq:GPhierarchy})$ can be constructed by using the methods of T. Chen and Pavlovi\'{c} \cite{CP4}. The authors study the Gross-Pitaevskii hierarchy on $\mathbb{R}^d$, but their analysis
can be applied on $\mathbb{T}^d$ without any changes, assuming that one can prove a spacetime estimate on the free evolution as is given by Proposition \ref{Spacetime Bound}. In the authors' notation, given $\alpha>0, \xi \in (0,1)$, and $\Sigma=(\sigma^{(k)})$, a sequence of $k$-particle density matrices as above, one defines the norm:
$$\big\|\Sigma\big\|_{\mathcal{H}^{\alpha}_{\xi}}:=\sum_{k=1}^{+\infty} \xi^k \cdot \|\sigma^{(k)}\|_{H^{\alpha}_k(\Lambda^k \times \Lambda^k)}$$
where:
$$\|\sigma^{(k)}\|_{H^{\alpha}_k}:=\|S^{(k,\alpha)}\sigma^{(k)}\|_{L^2(\Lambda^k \times \Lambda^k)}.$$
Let us assume that $\Gamma_0 \in \mathcal{H}^{\alpha}_{\xi}$ for $\alpha>1$.
By using the methods of \cite{CP4}, it can be shown that there exists $\eta>0$ sufficiently small, such that for $\xi' \in (0,\eta \xi)$, there exists $T \sim (\xi')^2$ and a solution $\Gamma \in L^{\infty}_{[0,T]} \mathcal{H}^{\alpha}_{\xi'}$ of $(\ref{eq:GPhierarchy})$ on the time interval $[0,T]$. Furthermore, this solution satisfies the a priori estimate:
\begin{equation}
\label{classmathcalB}
\|S^{(k,\alpha)}B_{j,k+1} \gamma^{(k+1)}\|_{L^2({[0,T]} \times \Lambda^k \times \Lambda^k)} \leq C^k
\end{equation}
for some $C>0$, and for all $k \in \mathbb{N}$, and $j \in \{1,2,\ldots,k\}$, which is exactly the bound from \cite{KM}.
One can now use the proof of our Theorem \ref{Theorem 1} to deduce the uniqueness of solutions to the Gross-Pitaevskii hierarchy on $[0,T]$ in the class $\mathcal{B}$ given by the condition \eqref{classmathcalB}. This class is non-empty by the above arguments. The only condition one needs is that $\alpha>1$. In this way, one can use Proposition \ref{Spacetime Bound}, which is crucial both in the proof of the existence and uniqueness of local-in-time solutions in $\mathcal{B}$. We omit the details.
\end{remark}

\subsection{Motivation for the problem and previously known results}
\subsubsection{Link between the Gross-Pitaevskii hierarchy and the BBGKY hierarchy}\label{1.3.1}
The Gross-Pitaevskii hierarchy is related to the \emph{BBGKY hierarchy}\footnote{BBGKY stands for Bogolyubov-Born-Green-Kirkwood-Yvon.}:
\begin{equation}
\label{BBGKY}
i \partial_t \gamma^{(k)}_{N} + (\Delta_{\vec{x}_k}-\Delta_{\vec{x}_k'})\gamma^{(k)}_N=
\end{equation}
$$\frac{1}{N} \sum_{i<j}^{k} \big[N^{d\beta}V(N^{\beta}(x_i-x_j)),\gamma^{(k)}_{N}\big]+\frac{N-k}{N}\sum_{j=1}^k Tr_{k+1}\big[N^{d\beta}V(N^{\beta}(x_j-x_{k+1})),\gamma^{(k+1)}_N\big].$$
Here, $N \geq k$ is a positive integer, $\beta \geq 0$ is a scaling factor, $V:\Lambda \rightarrow \mathbb{R}$ is a potential, and $\gamma^{(k)}_N: \mathbb{R} \times \Lambda^k \times \Lambda^k \rightarrow \mathbb{C}$, satisfies the symmetry assumptions in $(\ref{eq:conditionsongamma})$. Moreover, the sequence $(\gamma^{(k)}_N)$ satisfies the consistency condition $\gamma^{(k)}_N=Tr_{k+1}\, \gamma^{(k+1)}_N$.

If $\beta>0$, and if we formally let $N \rightarrow \infty$, we would expect the limiting object $\gamma^{(k)}_{\infty}$ to solve the Gross-Pitaevskii hierarchy $(\ref{eq:GPhierarchy})$ with coupling constant depending on $V$, since $N^{d\beta}V(N^{\beta}(\cdot))$ is a rescaling of the delta function. In the \emph{mean field case} $\beta=0$, there is no such rescaling present, and we formally expect the limiting object $\gamma^{(k)}_{\infty}$ to solve the Hartree hierarchy:
\begin{equation}
\label{eq:Hartreehierarchy}
\begin{cases}
i \partial_t \gamma^{(k)} + (\Delta_{\vec{x}_k}-\Delta_{\vec{x}_k'})\gamma^{(k)}= \sum_{j=1}^{k}\tilde{B}_{j,k+1}(\gamma^{(k+1)})\\
\gamma^{(k)}|_{t=0}=\gamma_0^{(k)}.
\end{cases}
\end{equation}
where the collision operator is now defined by:
$$\tilde{B}_{j,k+1}(\gamma_0^{(k+1)})(\vec{x}_k;\vec{x}_k'):=\sum_{j=1}^{k} Tr_{k+1} \left[V(x_j-x_{k+1}), \gamma_0^{(k+1)}\right](\vec{x}_k;\vec{x}_k'),$$
whenever $\gamma_0^{(k+1)}$ $(k+1)$-particle density matrix.

The BBGKY hierarchy comes up naturally when one wants to describe the $N$-boson system on $L^2_s(\Lambda^N)$, the subspace of the Hilbert space $L^2(\Lambda^N)$ consisting of permutation symmetric functions, on which one defines the Hamiltonian:
$$H_N=\sum_{j=1}^N (-\Delta_{x_j})+\frac{1}{N} \sum_{i<j}^N N^{d\beta} V(N^{\beta}(x_i-x_j)),$$
for some fixed $\beta \geq 0$.
The $N$-body Schr\"{o}dinger equation corresponding to $H_N$ is:
$$i \partial_t \psi_{N,t} = H_N \psi_{N,t}.$$
Then, the associated $k$-particle density matrices $$\gamma^{(k)}_{N,t}:=\int_{\Lambda^{N-k}}d\vec{y}_{N-k}\psi_{N,t}(\vec{x}_k,\vec{y}_{N-k}) \overline{\psi_{N,t}}(\vec{x}_k',\vec{y}_{N-k}) = Tr_{k+1,\ldots,N} |\psi_{N,t} \rangle \langle \psi_{N,t}| (\vec{x}_k, \vec{x}_k')$$ are non-negative trace class operators which solve the BBGKY hierarchy.

This setup provides us with a method to rigorously derive NLS-type equations from the $N$-body Schr\"{o}dinger equation in the appropriate limit as $N \rightarrow \infty$. More precisely, if one starts with initial data $\psi_{N,0} \approx \phi_0^{\otimes N}$ for some $\phi_0$ in a way which will be made precise later, the goal is to show that:
\begin{equation}
\label{eq:convergence}
Tr\Big|\gamma^{(k)}_{N,t}-|\phi \rangle \langle \phi|^{\otimes k} \Big| \rightarrow 0,\,\,\mbox{as}\,\,N \rightarrow \infty
\end{equation}
where $\phi$ is the solution of $(\ref{eq:NLS})$ when $\beta>0$ and where $\phi$ solves the Hartree equation:
\begin{equation}
\label{eq:Hartree}
\begin{cases}
i \partial_t \phi + \Delta \phi = (V*|\phi|^2)\,\phi,\,\,\mbox{on}\,\,\mathbb{R}_t \times \Lambda\\
\phi|_{t=0}=\phi_0,\,\mbox{on}\,\Lambda.
\end{cases}
\end{equation}
when $\beta=0$.
Here $Tr |A|:=\sum_{\lambda \in Sp(A)}|\lambda|$ is the trace norm.

The strategy for this derivation consists of two parts:
\begin{enumerate}
\item $lim_{N \rightarrow \infty} \gamma^{(k)}_{N,t}$ converges to a solution of the appropriate infinite hierarchy with initial data $|\phi_0 \rangle \langle \phi_0|^{\otimes k}$.
\item Solutions of this infinite hierarchy are unique in a class which contains the limiting objects and the desired limit.
\end{enumerate}
\subsubsection{Physical interpretation and previously known results:}
Both the Gross-Pitaevskii hierarchy and the Hartree hierarchy arise in the study of Bose-Einstein condensation. This is a state of matter of dilute bosonic gases confined by an external potential at a temperature near absolute zero.
At such a low temperature, the gas particles tend to occupy a single one-particle state, even after an external trap is turned off. In our previous framework, this state corresponds to the solution of the nonlinear Schr\"{o}dinger equation in $(\ref{eq:NLS})$ when $\beta>0$, and to the solution of the Hartree equation $(\ref{eq:Hartree})$ when $\beta=0$. In this context, the cubic nonlinear Schr\"{o}dinger equation is also called the \emph{Gross-Pitaevskii equation} after the work of Gross \cite{Gross} and Pitaevskii \cite{Pitaevskii}. This phenomenon was predicted by Bose and Einstein in 1924-1925 \cite{Bose,Einstein}.
The theory was experimentally verified in the work of the groups of Cornell, Wieman \cite{CW}, and Ketterle \cite{Ket}, who were given the Nobel Prize in Physics in 2001 for their discovery.

The problem of rigorously deriving NLS-type equations from $N$-body Schr\"{o}dinger dynamics had been studied since the 1970s, the first results being those of Hepp \cite{Hepp}, and of Ginibre and Velo \cite{GV1,GV2}. The techniques used in these papers are different than the ones outlined above and are based on the
Fock space version of the
$N$-body Schr\"{o}dinger equation, as well as on the use of coherent states as initial data. The Fock space representation turns out to be necessary since coherent states do not need to have a fixed number of particles. The Fock space strategy has been applied in the work of Ammari and Nier \cite{AN1,AN2}, Benedikter, de Oliveira, Schlein \cite{BdOS}, X. Chen \cite{XC1,XC2}, Fr\"{o}hlich, Graffi, and Schwarz \cite{FGS}, Fr\"{o}hlich, Knowles, and Pizzo \cite{FKP}, Grillakis and Machedon \cite{GM}, and Grillakis, Machedon, and Margetis \cite{GMM1,GMM2}.

The first application of the strategy outlined in Subsection \ref{1.3.1} was due to Spohn \cite{S}. In this work, the author gives a rigorous derivation of the Hartree equation: $i\partial_t u + \Delta u = (V*|u|^2)\,u$ on $\mathbb{R}^d$ with $V \in L^{\infty}(\mathbb{R}^d)$, see also \cite{AN1, AN2, MichelangeliSchlein}.
In further work, Erd\H{o}s and Yau \cite{EY} extend Spohn's result to the case of the Coulomb potential $V(x)=\pm \frac{1}{|x|}$ on $\mathbb{R}^3$, by the use of operator inequalities. Some partial results in this direction were also proved by Bardos, Golse, and Mauser \cite{BGM}. An alternative proof of the result with Coulomb potential was given in the work of Fr\"{o}hlich, Knowles, and Schwarz \cite{FKS}. Furthermore, in \cite{ES}, Elgart and Schlein consider the case of a Coulomb potential, but with relativistic dispersion $\sum_{j=1}^{N}(1-\Delta_j)^{\frac{1}{2}}$ and they derive a relativistic nonlinear Hartree equation. All of the aforementioned results deal with the mean-field case $\beta=0$ and with factorized initial data $\psi_N=\phi^{\otimes N}$. Subsequent results on this problem were also obtained in the work of Ammari and Nier \cite{AN1,AN2}.

In the case $\beta>0$, due to the stronger interaction between the particles, one needs to replace the assumption of pure factorization with the less stringent assumption of \emph{complete Bose-Einstein condensation}. Complete Bose-Einstein condensation states that there exists $\phi \in L^2(\Lambda)$ such that:
$$Tr \big|\gamma^{(1)}_N - |\phi \rangle \langle \phi| \big| \rightarrow 0$$
as $N \rightarrow \infty.$
The assumption of complete Bose-Einstein condensation is in general difficult to verify.
This is not a purely a condition which simplifies the analysis, but it occurs in the physical study of the problem.
When $\Lambda=\mathbb{R}^3$ and when $\beta=1$,
Lieb, Seiringer, and Yngvason \cite{LSY,LSY2} add an external confining potential $V_{ext}$ in $H_N$ and study the ground state energy of $H_N$ divided by the number of particles $N$. Here, $V_{ext}$ is assumed to be measurable, locally bounded, and it is assumed to tend to infinity as $|x| \rightarrow \infty$, in the sense that $\inf_{|x|\geq R} V(x) \rightarrow \infty$ as $R \rightarrow \infty$. Under these assumptions, it is shown that the ground state energy converges to the minimum over $\|\phi\|_{L^2}=1$ of the Gross-Pitaevskii energy:
$\mathcal{E}_{GP}(\phi):=\int \big(|\nabla \phi(x)|^2 + V_{ext}(x)|\phi(x)|^2+4\pi a_0 |\phi(x)|^4 \big) \, dx$, where $a_0$ equals the scattering length of the interaction potential $V$. Here, $V$ is assumed to be positive, radial, and rapidly decreasing.
In later work, Lieb and Seiringer \cite{LS} use these results to rigorously verify the condition of Bose-Einstein condensation in the stationary problem. A systematic expository summary of these results can be found in the textbook by Lieb, Seiringer, Solovej and Yngvason \cite{LSSY}. Related results were also proven by Lee and Yin \cite{LeeYin}, Yin \cite{Yin}, and Yau and Yin \cite{YauYin}.

In a series of monumental works, Erd\H{o}s, Schlein, and Yau \cite{ESY2,ESY3,ESY4,ESY5} give a rigorous derivation of the defocusing cubic NLS on $\mathbb{R}^3$. For the initial data, they need to assume complete Bose-Einstein condensation, as well as finite energy per particle, i.e.
$$\langle \psi_N, H_N \psi_N \rangle \leq CN.$$
The interaction potential is again assumed to be sufficiently regular.
The coupling constant is given by $b_0=\int dx V$ when $0<\beta<1$.
Due to the short scale correlation structure, in the case of the \emph{Bose-Einstein scaling} $\beta=1$, the coupling constant is determined by the scattering length of $V$.
In the uniqueness step of the strategy, the authors use a Feynman graph expansion in order to control the terms which come from the Duhamel iteration.
The same problem was also studied in the one-dimensional case in the work of Adami, Bardos, Golse, and Teta \cite{ABGT}, as well as Adami, Golse, and Teta \cite{AGT}. We would also like to recall the connection with certain optical lattice models that have been studied in the work of Aizenman, Lieb, Seiringer, Solovej and Yngvason \cite{ALSSY1,ALSSY2}.
For an expository account of all of these results, we refer the reader to the lecture notes of Schlein \cite{Schlein}.

A different approach to the question of uniqueness was given in the work of Klainerman and Machedon \cite{KM}. The authors use an alternative combinatorial argument which shows uniqueness under the assumption of an a priori spacetime bound. In the proof, they also have to use Strichartz-type estimates reminiscent of their earlier work \cite{KM2}. In subsequent work, Kirkpatrick, Schlein, and the third author of this paper \cite{KSS} verify that the above spacetime bound holds for the limiting objects when $\Lambda=\mathbb{R}^2$ and when $\Lambda=\mathbb{T}^2$ using a trace theorem. 
For a different treatment of the derivation of the cubic NLS on $\mathbb{R}^2$ coming from 3D quantum dynamics, we also refer the reader to the work of X. Chen and Holmer \cite{ChenHolmer1}.
The work \cite{KSS} is the first complete treatment of a periodic problem. A previous analysis of the periodic problem in three dimensions, without the study of uniqueness, was done in the work of Elgart, Erd\H{o}s, Schlein and Yau \cite{EESY}, and in the work of Erd\H{o}s, Schlein, and Yau \cite{ESY1}. Elgart, Erd\H{o}s, Schlein, and Yau \cite{EESY1} previously also considered the mean-field case $\beta=0$, without the uniqueness analysis.

In the work of T. Chen and Pavlovi\'{c} \cite{CP}, the spacetime bound has also been verified in the case of $\Lambda=\mathbb{R}^3$ under a slightly different limiting procedure when the scaling parameter $\beta \in (0,1/4)$. More precisely, the authors assume that the initial data is slightly more regular than in $H^1$. By an appropriate truncation in the number of particles, they are able to show that the limiting object solves the Gross-Pitaevskii hierarchy and satisfies the wanted spacetime bound. The additional regularity assumption has been removed in further work of T. Chen and Taliaferro \cite{CT}. 
In subsequent work, X. Chen \cite{XC4} gives a derivation of the NLS with a quadratic potential when $\beta \in (0,\frac{2}{7}]$ and proves the Klainerman-Machedon spacetime bound under the limiting procedure from \cite{ESY2,ESY3,ESY4} with initial data of regularity $H^1$.
Moreover, in \cite{ChenHolmer2}, X. Chen and Holmer have verified the spacetime bound of Klainerman and Machedon for the NLS on $\mathbb{R}^3$
for $\beta\in (0,2/3)$.

The work \cite{CP} is based on techniques developed by T. Chen and Pavlovi\'{c} to study the Cauchy problem associated to the GP-hierarchy on $\mathbb{R}^d, d \geq 1$ \cite{CP1,CP2,CP3,CP1B,CP2B,CP3B,CP4} . In addition, in joint work with Tzirakis \cite{CPT1,CPT2}, they have shown blow-up results for the Cauchy problem, as well as the existence of multilinear Morawetz identities. A different approach to the study of the Cauchy problem of the Gross-Pitaevskii hierarchy in several special cases was given in the work of Z. Chen and C. Liu \cite{CL}. We further note that, in this context, a rigorous derivation of the nonlinear Schr\"{o}dinger equation with anisotropic switchable quadratic traps has been shown by X. Chen \cite{XC3,XC4}. Furthermore, in recent work Xie \cite{Xie} has given a derivation of a nonlinear Schr\"{o}dinger equation with general power-type nonlinearity on $\mathbb{R}$ and on $\mathbb{R}^2$. In \cite{ChenHolmer3}, X. Chen and Holmer established the rigorous derivation 
of the focusing cubic NLS on $\mathbb{R}$ from quantum many-body dynamics with attractive interactions. 


It is a natural question to ask what is the rate of convergence in $(\ref{eq:convergence})$. This question was first addressed by Rodnianski and Schlein \cite{RodnianskiSchlein}. The authors show that, in the case of the Coulomb interaction potential in $\mathbb{R}^{3+1}$, one has the bound:
$$Tr\big|\gamma^{(k)}_{N,t}-|\phi \rangle \langle \phi|^{\otimes k} \big| \leq C \frac{e^{Kt}}{N^{\frac{1}{2}}}.$$
This result was extended in the work of Knowles and Pickl \cite{KP} by using different methods. 
Subsequent results on this aspect of the problem have been proved in \cite{Anapolitanos,BdOS,CLS, CP,CT,XC1,XC2,ChenHolmer1,ChenHolmer2,ChenHolmer3,ES,FGS,FL,GM,GMM1,Lee,Luhrmann,MichelangeliSchlein,Pickl1,Pickl2,Xie}.

After the appearance of our preprint, an alternative approach for showing uniqueness in Gross-Pitaevskii hierarchies on $\mathbb{R}^3$ was given by T. Chen, Hainzl, Pavlovi\'{c}, and Seiringer \cite{CHPS}. This uniqueness result is unconditional and it holds for regularity $\alpha \geq 1$. Its proof is based on the quantum de Finetti theorem, which is a quantum analogue of the Hewitt-Savage theorem from probability theory \cite{HewittSavage}. In the version applied in \cite{CHPS}, this theorem was first stated in the work of Lewin, Nam, and Rougerie \cite{LewinNamRougerie}. The proof of the uniqueness result in \cite{CHPS} uses Strichartz estimates on $\mathbb{R}^3$, so this approach does not directly apply to the domain $\mathbb{T}^3$. The same group of authors used these ideas to prove a scattering result for the GP hierarchy \cite{ChHaPavSei2}.

In the analysis of the Gross-Pitaevskii hierarchy based on the technique of  Klainerman and Machedon, it is important to obtain a dispersive spacetime estimate such as  $(\ref{eq:outlinespacetime})$ below. This estimate was proved on $\mathbb{R}^3$ by Klainerman and Machedon \cite{KM} when $\alpha \geq 1$.
In \cite{KSS}, Kirkpatrick, Schlein, and the third author of this 
paper prove the estimate on $\mathbb{R}^2$ and $\mathbb{T}^2$ for a class of regularity exponents which are below the energy regularity. The authors do not address the issue of optimality of the regularity exponent, as the aim of this work was to obtain an estimate in the energy space. Similar estimates for the three-body collisions on $\mathbb{R}^d$ were proved in the work of T. Chen and Pavlovi\'{c} \cite{CP1}. In subsequent work \cite{XC3}, X. Chen proved the spacetime estimate on $\mathbb{R}^d$ when $d \geq 2$ for $\alpha \geq \frac{d-1}{2}$, which was noted to be optimal in the non-periodic setting by scaling arguments. In addition, some variable-coefficient versions of the estimate were obtained. This work builds on the previous results by X. Chen \cite{XC} and Grillakis and Margetis \cite{GrillakisMargetis}. In particular, when $d=3$, the optimal regularity estimate corresponds to the estimate used by Klainerman and Machedon \cite{KM}. In \cite{Beckner}, Beckner gives two higher dimensional generalizations of the spacetime bound. The estimates proved in \cite{Beckner} have a slightly different motivation than the study of the Gross-Pitaevskii hierarchy. In particular, they are viewed as a step in the larger scale dual program for understanding how smoothness controls restriction to a non-linear sub-variety \cite{Beckner2}. Subsequently, in \cite{ChenHolmer3}, X. Chen and Holmer prove the optimal range of regularity exponents on $\mathbb{R}$.
Let us remark that in the analysis of the estimate $(\ref{eq:outlinespacetime})$ in the non-periodic setting mentioned above, the arguments used rely on scaling properties of the delta function. Such arguments can not be applied in the periodic setting.

\subsection{Main ideas of the proof}
In order to prove Theorem \ref{Theorem 1}, we first write an iterated Duhamel expansion. We then apply the combinatorial \emph{boardgame} approach from \cite{KM}. The combinatorial reductions in the non-periodic setting apply in the periodic setting without any changes.

The main issue is to check two properties. The first is the spacetime estimate, given by Proposition
\ref{Spacetime Bound}, which allows us to contract the terms appearing in the Duhamel expansion. The second property one has to check is that, under appropriate assumptions, the factorized solutions to the Gross-Pitaevskii hierarchy, given by $(\ref{eq:factorizedsolution})$, lie in the admissible class $\mathcal{A}$, defined by the a priori bound $(\ref{eq:aprioribound})$. Hence, the class $\mathcal{A}$ is non-empty and it contains the physically relevant solutions which are related to the nonlinear Schr\"{o}dinger equation. The precise statement is given in Theorem \ref{Theorem 2}. As we will see, the fact that we are in the periodic setting will cause additional difficulties in checking both properties.

The spacetime bound we prove is:
\begin{equation}
\label{eq:outlinespacetime}
\|S^{(k,\alpha)}B_{j,k+1} \, \mathcal{U}^{(k+1)}(t)\gamma_0^{(k+1)}\|_{L^2([0,2\pi] \times \Lambda^k \times \Lambda^k)} \leq C_1 \|S^{(k+1,\alpha)} \gamma_0^{(k+1)}\|_{L^2(\Lambda^{k+1} \times \Lambda^{k+1})}.
\end{equation}
whenever $\alpha>1$.
Here,\, $\mathcal{U}^{(k)}(t)\gamma_0^{(k)}:=e^{it \sum_{j=1}^{k} \Delta_{x_j}} \gamma_0^{(k)} e^{-it \sum_{j=1}^{k} \Delta_{x_j'}}$ denotes the free evolution operator. 
This spacetime bound is proved in Proposition \ref{Spacetime Bound}.

In the proof of the estimate $(\ref{eq:outlinespacetime})$, we argue as in \cite{KM}, and we observe that the left-hand side of $(\ref{eq:outlinespacetime})$ is
given in terms of a convolution. The key is then to check uniform pointwise bounds on the multiplier:
\begin{equation}
\label{eq:1expressionI}
I=I(\tau,p):=\sum_{n,m \in \mathbb{Z}^3} \frac{\delta(\tau+|p-n-m|^2+|n|^2-|m|^2) \langle p \rangle^{2\alpha}}
{\langle p-n-m \rangle^{2\alpha} \langle n \rangle^{2\alpha} \langle m \rangle^{2\alpha}}.
\end{equation}

One way to obtain the pointwise boundedness of the multiplier given in $(\ref{eq:1expressionI})$ is based on lattice counting arguments previously used in work on nonlinear dispersive equations on periodic domains \cite{B,DPST,KSS}. In this way, one can show that the multiplier given in $(\ref{eq:1expressionI})$ is bounded when $\alpha>\frac{5}{4}$. The idea is that, in order to count the number of lattice points in a set $S \subseteq \mathbb{R}^3$, we fix $m_0 \in S$, and we then count the number of $k \in \mathbb{Z}^3$ such that $m_0+k \in S$. The point is that, if $S$ has a specific structure, we can obtain additional geometric insight about $k$ and so this counting problem is easier than the original one. For instance, if $k$ has to lie in the intersection of a plane and a ball of radius $R$, then there are at most a uniform constant times $R^2$ possible values of $k$. The fact that there is a delta function in $(\ref{eq:1expressionI})$ allows us to reduce the dimension by one and to look at a five-dimensional sum, instead of a six-dimensional sum. This idea was first applied in \cite{KSS}, in which the two-dimensional analogue of the above sum was shown to be bounded 
below the energy threshold $\alpha=1$. In the two-dimensional case, the authors had at their disposal sharper lattice point counting techniques, such as the Gauss Lemma \cite{BombieriPila,Huxley,IrelandRosen}, which was first used in the PDE context in work of Bourgain \cite{B}, and later in \cite{DPST}. In three-dimensions, however, this approach only gives the bound when $\alpha>\frac{5}{4}$. We will omit the details of this approach and we will refer the interested reader to \cite{KSS} for the discussion in the two-dimensional setting.

In Proposition \ref{Spacetime Bound}, we will use a different method to prove that the expression in $(\ref{eq:1expressionI})$ is bounded, which allows us to obtain the range $\alpha>1$. Our idea is to rewrite the sum $(\ref{eq:1expressionI})$ as:
\begin{equation}
\label{1thesum2}
I(\tau,p)=\sum_{m,n \in \mathbb{Z}^3} \frac{\delta(\tau + |p|^2 - 2 \langle n, m \rangle) \ang{p}^{2\alpha}}{\ang{m-p}^{2\alpha} \ang{n-p}^{2\alpha} \ang{p-n-m}^{2\alpha}}.
\end{equation}
In the non-periodic case \cite{KM}, it was an important part of the proof to use scaling properties of the delta function, in order to obtain a decay factor. In the periodic setting, we cannot apply such scaling arguments. In fact, we observe that there is never any decay along $I(|p|^2,p)$, if we take $(m,n) = (p,p)$. However, having rewritten the sum as $(\ref{eq:1expressionI})$, we can use the concept of the \emph{determinant of a lattice}, which allows us to sharpen the counting arguments and obtain the bound for $\alpha>1$. Here, we recall that the determinant of the lattice is defined to be the (positive) volume of the smallest parallelepiped spanned by the elements of the lattice.

The key point is that we can use more detailed number theoretic arguments to estimate the number of lattice points which lie in the intersection of a ball and a plane which is normal to a vector in $\mathbb{Z}^3$. The number theoretic properties of the normal vector will be important for our argument. More precisely, in Lemma \ref{MainCountingLemma}, we prove that for fixed $m \in \mathbb{Z}^3$, and $c \in \mathbb{Z}$, if ${\mathcal E} \subset \mathbb{Z}^3$ is a set of diameter $R \geq 1$ then
\begin{equation}
\label{eq:1counting1}
\# \left( \mathcal{E} \cap \{ x \in \mathbb{Z}^3, \langle m,x \rangle = c\} \right) \lesssim R + \frac{R^2}{|m_*|}.
\end{equation}
Here $m_*$ is defined to be $m$ divided by the greatest common divisor of its coordinates, with the convention that $m_* := 0$ when $m=0$. This quantity equals the determinant of a certain lattice. In particular, in the work of P. McMullen \cite{mcmullen1984}, with later generalizations by Schnell \cite{schnell1992}, it is shown that, for fixed $m \in \mathbb{Z}^3$, the determinant of the lattice $$\Lambda_m := \{x \in \mathbb{Z}^3, \langle m,x \rangle = 0 \}$$ equals $|m_*|$.
We note that the bound in $(\ref{eq:1counting1})$ gives an improvement over the immediate bound of $O(R^2)$ which holds for general planes. Its proof is based on the geometric interpretation of the determinant of a lattice.
The bound $(\ref{eq:1counting1})$ is useful since we are also able to give a bound on the sum of $\frac{1}{|m_*|}$. More precisely, in Lemma \ref{sum1mstar}, we prove that whenever $\mathcal{E} \subset \mathbb{Z}^3 \setminus \{0\}$ is a set of diameter $R \geq 1$, then:
\begin{equation}
\label{eq:1counting2}
\sum_{m \in {\mathcal E}} \frac{1}{|m_*|} \lesssim R^2.
\end{equation}
We can then estimate $I(\tau,p)$ by using $(\ref{eq:1counting1})$ in the inner sum and $(\ref{eq:1counting2})$ in the outer sum. We believe that this counting approach is of independent interest.

Let us recall that in \cite{KM}, when one is working on $\mathbb{R}^3$, the above mentioned scaling properties of the $\delta$ function lead one to study convolution estimates for generalizations of Riesz potentials:
\begin{equation}
\label{eq:Rieszpotential}
\int_{\Sigma} \frac{1}{|\xi-\eta|^a} \frac{1}{|\eta|^b} dS(\eta)
\end{equation}
where $\Sigma$ is a smooth submanifold of $\mathbb{R}^3$, and $\xi \in \mathbb{R}^3$. In this concrete case, $\Sigma$ is taken to be either a plane or a sphere. More precisely, in equation (16) of \cite{KM}, the authors use homogeneity properties of the delta function and the coarea formula to deduce that:
$$\delta(\tau+|\xi_1-\xi_2-\xi_2'|^2+|\xi_2|^2-|\xi'_2|^2) \, d\xi'_2 =\frac{dS(\xi'_2)}{2|\xi_1-\xi_2|}$$
where $dS$ is the surface measure of a plane in $\mathbb{R}^3$.

Similarly, in equation (18) of \cite{KM}, it is shown that:
$$\delta(\tau+|\xi_1-\xi_2-\xi'_2|^2+|\xi_2|^2-|\xi'_2|^2)\,d\xi_2 = \frac{dS(\xi_2)}{4|\xi_2-\frac{\xi_1-\xi'_2}{2}|}$$
where now $dS$ is the surface measure of a sphere in $\mathbb{R}^3$. The aim is to obtain a bound on the integral in $(\ref{eq:Rieszpotential})$ in terms of a negative power of $|\xi|$, for a suitable choice of $a$ and $b$.

In Subsection \ref{alpha1discussion}, we give a precise study of the endpoint case $\alpha=1$. In Proposition \ref{Alpha1log}, we prove that, in the case of frequency localized data, we obtain a logarithmic loss of derivative in the analogue of $(\ref{eq:outlinespacetime})$. Furthermore, in Proposition \ref{Alpha1unbounded}, we prove that the estimate $(\ref{eq:outlinespacetime})$ does not hold in the case $\alpha=1$. This shows that the number theoretic method based on the determinant of a lattice gives a sharp result for the range of $\alpha$.
The proof requires a precise decomposition of the operator $B_{j,k+1}$ and a specific choice of data for which one part in this decomposition is bounded and the other one grows logarithmically in the size of the largest frequency.

Such a result is in drastic contrast with the setting of $\mathbb{R}^3$ \cite{KM} and in the setting of $\mathbb{T}^2$ and $\mathbb{R}^2$ \cite{KSS}, where the analogue of $(\ref{eq:outlinespacetime})$ holds for $\alpha=1$.
Let us remark that in Proposition \ref{factorized}, we prove that the above spacetime estimate holds when $\alpha=1$ in the case of \emph{factorized objects}, i.e. for $\gamma^{(k)}_{0}:=|\phi_0 \rangle \langle \phi_0|^{\otimes k}$ for $\phi_0 \in H^1$. However, it is important to note that the factorization property is not preserved under the action of the collision operators $B_{j,k+1}$ in the Duhamel iteration for general factorized density matrices. As a result, the result of Proposition \ref{factorized} cannot be used to prove a uniqueness result at the level of regularity $\alpha=1$.

Having these facts in mind, we note that the factorized objects are very special in the class $\mathcal{H}^1$ of density matrices containing one derivative in each variable in $L^2$. Namely, they satisfy $(\ref{eq:outlinespacetime})$ for $\alpha=1$, whereas general objects do not
necessarily have to satisfy this estimate. It would be interesting to see if there is any other non-trivial class of elements in $\mathcal{H}^1$ which also satisfy $(\ref{eq:outlinespacetime})$ when $\alpha=1$. In the forthcoming work \cite{SoSt}, an affirmative answer is given to this question in the probabilistic sense. Here, the range of the regularity parameter is $\alpha>\frac{3}{4}$.
We want to prove that the factorized solutions to the Gross-Pitaevskii hierarchy lie in the class $\mathcal{A}$ if the initial data $\phi_0$ for the NLS $(\ref{eq:NLS})$ is sufficiently regular. This fact tells us that the class $\mathcal{A}$ is non-empty and physically relevant\footnote{For another non-empty class $\mathcal{B}$ of \emph{local-in-time} solutions, see Remark \ref{mathcalB}.}. The analogous fact in the non-periodic setting relies, given in Remark 1.2 of \cite{KM}, relies on the use of Strichartz estimates. If one wanted to take this approach in the periodic setting, one would need to use the Strichartz estimate on $\mathbb{T}^3$ due to Bourgain \cite{B}:
\begin{theorem}
\label{StrichartzT3}
(Strichartz estimate on $\mathbb{T}^3$; Bourgain \cite{B})
Let $p>4,N>1$. Let $P_{\leq N}$ denote the projection onto frequencies $|n| \leq N$. Then:
$$\|P_{\leq N}e^{it\Delta}\phi\|_{L^p([0,1] \times \mathbb{T}^3)} \lesssim N^{\frac{3}{2}-\frac{5}{p}}\|P_{\leq N}\phi\|_{L^2(\mathbb{T}^3)}.$$
\end{theorem}
However, the loss of $\frac{3}{2}-\frac{5}{p}$ derivatives turns out to be too much to allow for the argument from \cite{KM} to carry through, so we have to use a refinement of Bourgain's result.
We can circumvent the mentioned difficulty by using multilinear estimates in the atomic spaces $U$ and $V$. We note that the $U$ spaces were first defined in the work of Koch and Tataru \cite{KT}. See also \cite{KT2,KT3}. A variant of the $V$ spaces is originally found in the work of Wiener \cite{Wiener}. The first application of both spaces in the context of nonlinear dispersive equations appears in \cite{KT}. Variants of these spaces have been used in a class of critical problems in the work of Hadac, Herr, and Koch \cite{HHK}, Herr, Tataru, and Tzvetkov \cite{HTT}, Ionescu and Pausader \cite{IP}, and Dodson \cite{Dodson1,Dodson2}, as well as in the context of low regularity solutions for the Korteweg-de Vries equation by B. Liu \cite{Liu}. More recently, these techniques have been applied in the study of low regularity almost-sure local well-posedness for the NLS in the work of Nahmod and the third author of this paper \cite{NS}, as well as in the study of low regularity solutions for the NLS equation on the irrational torus by Guo, Oh, and Wang \cite{GuoOhWang}. For a detailed account of $U$ and $V$ spaces, we refer the reader to Section 2 of \cite{HHK}.

In particular, we use variants of these spaces adapted to the Schr\"{o}dinger equation which were used in the context of the energy critical NLS on $\mathbb{T}^3$ \cite{HTT}, and later in \cite{IP}. The key fact that one has to prove is the estimate, given in Corollary \ref{productestimate} (see also $(\ref{eq:trilinearbound})$) :
\begin{equation}
\label{eq:3linearestimate}
\|D^{\alpha}_x(|\phi|^2\phi)\|_{L^2(I \times \Lambda)} \lesssim \|\phi\|_{Y^{\alpha}(I)} \|\phi\|_{Y^1(I)}^2 \lesssim \|\phi\|_{X^{\alpha}(I)}^3.
\end{equation}
Here, $I$ is a time interval and the function spaces are defined below. In these spaces, the superscript denotes the order of differentiation. The bound $(\ref{eq:3linearestimate})$ holds when $\alpha \geq 1$, and the implied constant depends on the length of $I$. The point is that, in the above estimate, there is no longer a derivative loss which would be present if one were to directly apply the Strichartz estimates.

We note that the factorized solutions lie in the class $\mathcal{A}$ whenever $\alpha \geq 1$. Hence, the restriction to $\alpha>1$ in Theorem \ref{Theorem 1} comes from the fact that we need to have
$\alpha>1$ in order to prove the spacetime bound $(\ref{eq:outlinespacetime})$. In particular, if $\alpha=1$, the bound $(\ref{eq:outlinespacetime})$ does not  hold. For details, we refer the reader to Subsection \ref{alpha1discussion}.
\subsection{Organization of the paper}
In Section \ref{NotationHarmonicAnalysis} we recall some notation and some known facts from harmonic analysis. In particular, we give the definition of the atomic spaces $U$ and $V$. In Section \ref{Spacetime Bound}, we prove the spacetime estimate which was mentioned in $(\ref{eq:outlinespacetime})$ when $\alpha>1$. This is the content of Proposition \ref{Spacetime Bound}. In Subsection \ref{NumberTheory}, we introduce some notions from number theory, which allow us to prove strong lattice point counting results in Lemma \ref{sum1mstar} and in Lemma \ref{MainCountingLemma}. In Subsection \ref{ConclusionProofProposition}, we put these number-theoretic results together to give a proof of Proposition \ref{Spacetime Bound}, thus obtaining the key spacetime estimate for the free evolution operator when $\alpha>1$. In Subsection \ref{alpha1discussion} we discuss what happens when $\alpha=1$. In particular, we prove Proposition \ref{Alpha1log}, which states that we obtain a logarithmic loss of derivative and Proposition \ref{Alpha1unbounded}, which states that the spacetime estimate in $(\ref{eq:outlinespacetime})$ does not hold when $\alpha=1$. The main result of the paper, Theorem \ref{Theorem 1}, is proved in Section \ref{MainResult}. We note that the Duhamel expansion is explicitly given in $(\ref{eq:Duhamel})$ of this section. Theorem \ref{Theorem 2} concerning the factorized solutions is proved in Section \ref{Factorized}. We note that in this section, we need to use the theory of $U$ and $V$ spaces. In Appendix A, we present the proof of the local-in-time result given by Proposition \ref{Proposition 4}, which is needed in Section \ref{factorized} in order to prove Theorem \ref{Theorem 2}.

\vspace{5mm}

\textbf{Acknowledgements:}
The authors would like to thank Herbert Koch for a helpful explanation of $U$ and $V$ spaces and for suggesting the reference \cite{HHK}. They are also grateful to Boris Ettinger, Sebastian Herr and Baoping Liu for useful discussions about these spaces. Moreover, the authors would like to thank Benjamin Schlein, Antti Knowles, Thomas Chen and Nata\v{s}a Pavlovi\'{c} for helpful comments and discussions about the Gross-Pitaevskii hierarchy. They are grateful to Xuwen Chen for comments on an earlier version of the manuscript. Finally, they would like to thank the referee for a detailed reading and for a lot of helpful comments.
P. G. was partially supported by NSF Grant DMS-1101393 and an Alfred P. Sloan Research Fellowship. V.S. was supported by a Simons Postdoctoral Fellowship. G.S. was partially supported by NSF Grant DMS-1068815.
\section{Some notation and tools from harmonic analysis}
\label{NotationHarmonicAnalysis}
In our paper, let us denote by $A \lesssim B$ an estimate of the form $A \leq CB$, for some constant $C>0$. If $C$ depends on a parameter $p$, we write $A \lesssim_p B$, which we also write as $C=C(p)$.
Throughout the paper, we fix $\Lambda=\mathbb{T}^3$.
\subsection{The Fourier transform and differentiation}
We define the Fourier transform on $\Lambda=\mathbb{T}^3$ to be:
$$\widehat{f}(n):=\int_{\Lambda} f(x) e^{- i \langle x, n \rangle} dx.$$
Here, $n \in \mathbb{Z}^3$ and $\langle \cdot, \cdot \rangle$ denotes the inner product.
On $[0,2\pi] \times \Lambda$, we define the spacetime Fourier transform by:
$$\widetilde{u}(\tau,n):=\int_{[0,2\pi]} \int_{\mathbb{T}^3} u(t,x) e^{- i t \tau-i \langle x, n \rangle} dx dt.$$
Here, $\tau \in \mathbb{Z}$.

When we are considering functions $\gamma_0^{(k)}: \Lambda^k \times \Lambda^k \rightarrow \mathbb{C}$, for fixed $k \in \mathbb{N}$, we define the Fourier transform to be:
$$(\gamma_0^{(k)})\,\,\widehat{}\,\,(\vec{n}_k;\vec{n}'_k): = \int_{\Lambda^k \times \Lambda^k} \gamma_0^{(k)}(\vec{x}_k;\vec{x}'_k)
e^{-i \cdot \sum_{j=1}^{k} \langle x_j, n_j \rangle + i \cdot \sum_{j=1}^{k} \langle x'_j, n'_j \rangle} d\vec{x}_k \, d\vec{x}'_k.$$
Here $\vec{n}_k=(n_1,\ldots,n_k),\vec{n}'_k=(n'_1,\ldots,n'_k) \in (\mathbb{Z}^3)^k$.
The reason why we take the complex conjugate in the last $k$ variables is in order to be consistent with the definition of the factorized solutions to the Gross-Pitaevskii hierarchy given by $(\ref{eq:factorizedsolution})$.

We also define the spacetime Fourier transform of functions $\gamma^{(k)}: [0,2\pi] \times \Lambda^k \times \Lambda^k \rightarrow \mathbb{C}$ as:
$$(\gamma^{(k)})\,\,\widetilde{}\,\,(\tau,\vec{n}_k;\vec{n}'_k): = \int_{[0,2\pi] \times \Lambda^k \times \Lambda^k} \gamma^{(k)}(t,\vec{x}_k;\vec{x}'_k)
e^{-it \tau-i \cdot \sum_{j=1}^{k} \langle x_j, n_j \rangle + i \cdot \sum_{j=1}^{k} \langle x'_j, n'_j \rangle} d\vec{x}_k d\vec{x}'_k dt.$$
The fractional differentiation operator $S^{(k,\alpha)}$ was defined on $k$-particle density matrices in $(\ref{eq:Ska})$. For a sequence $\Gamma_0=(\gamma_0^{(k)})$ of $k$-particle density matrices as above, we say that
\begin{equation}
\label{mathcalHalpha}
\Gamma_0 \in \mathcal{H}^{\alpha}
\end{equation}
if for all $k$, the quantity $\|S^{(k,\alpha)} \gamma_0^{(k)}\|_{L^2(\Lambda^k \times
\Lambda^k)}$ is finite.
For fractional differentiation in a single variable $x$, we also sometimes write:
$$D^{\alpha}_x:=(1-\Delta_x)^{\frac{\alpha}{2}}.$$
Finally, we define the \emph{free evolution} operator $\mathcal{U}^{(k)}(t)$ by:
\begin{equation}
\label{eq:FreeEvolution}
\mathcal{U}^{(k)}(t)\gamma_0^{(k)}:=e^{it \sum_{j=1}^{k} \Delta_{x_j}} \gamma_0^{(k)} e^{-it \sum_{j=1}^{k} \Delta_{x_j'}}.
\end{equation}
We note that then:
$$\Big(i \partial_t + (\Delta_{\vec{x}_k}-\Delta_{\vec{x}_k'})\Big)\,\mathcal{U}^{(k)}(t)\gamma_0^{(k)}=0.$$
\subsection{Function spaces}
We take the following convention for the Japanese bracket $\langle \cdot \rangle$ :
$$\langle x \rangle: =\sqrt{1+|x|^2}.$$
Let us recall that we are working in Sobolev Spaces $H^s(\Lambda)$ on the the three-dimensional torus whose norms are defined for $s \in \mathbb{R}$ by:
$$\|f\|_{H^s(\Lambda)}:=\big(\sum_{n \in \mathbb{Z}^3}|\widehat{f}(n)|^2 \langle n \rangle^{2s}\big)^{\frac{1}{2}}. $$
Given a time interval $I$, we denote by $C_u(I;H^{\alpha}(\Lambda))$ the space of all functions $u=u(t,x)$, such that the function $t \mapsto H^{\alpha}(\Lambda)$ is uniformly continuous on $I$. With the norm $\|\cdot\|_{L^{\infty}(I;H^{\alpha}(\Lambda))}$, $C_u(I;H^{\alpha}(\Lambda))$ becomes a Banach space.

We now define the atomic function spaces $U$ and $V$ following \cite{KT,KT2}.
We fix $H$ to be a separable Hilbert space over $\mathbb{C}$. Let $\mathcal{Z}$ denote the set of all finite partitions of the real line of the form $-\infty<t_0<t_1<\cdots<t_K \leq +\infty$. If $t_K$ equals $+\infty$, we use the convention that all of the functions $v:\mathbb{R} \rightarrow H$ take value zero at $t_K$. Given a set $I \subseteq \mathbb{R}$, we let $\chi_I$ denote its characteristic function.
We can now define the $U$-space:
\begin{definition} ($U$-space)
\label{Uspace}
Given $1 \leq p <\infty$, and $H$ as above, we call $a:\mathbb{R} \rightarrow H$ an $U^p$-atom if $a=\sum_{k=1}^{M} \chi_{[t_{k-1},t_k)}\phi_{k-1}$ for some $\phi_k \in H$ satisfying $\sum_{k=0}^{M-1} \|\phi_{k}\|_{H}^p=1$.
Then, $U^p=U^p(\mathbb{R},H)\subseteq L^{\infty}(\mathbb{R},H)
$ is defined to be the space given by the norm:
$$\|u\|_{U^p}:=\inf \Big\{\sum_{j=1}^{+\infty} |\lambda_j|, u=\sum_{j=1}^{+\infty} \lambda_j a_j, a_j-U^p\, \mbox{atom}\Big\}.$$
\end{definition}
Let us observe that the procedure to define the $U$-space above is analogous to atomic decomposition in Harmonic analysis as was studied by Fefferman and Stein in \cite{FeffermanStein}.

We now define the $V$ space:
\begin{definition}($V$-space)
\label{Vspace}
Given $1 \leq p < \infty$, we define $V^p(\mathbb{R};H)$ to be the space of all functions $v:\mathbb{R} \rightarrow H$ such that:
$$\|v\|_{V^p}:=\sup_{(t_k)_{k=0}^{K} \in \mathcal{Z}} \Big(\sum_{k=1}^{K}
\|v(t_k)-v(t_{k-1})\|_{H}^p\Big)^{\frac{1}{p}} < \infty.$$
In the definition, one also takes the convention that $v(+\infty):=0$.
\end{definition}
It is shown \cite{KT,KT2} that the following inclusions hold:
\begin{equation}
\label{eq:basicinclusion}
U^p(\mathbb{R};H) \hookrightarrow
V^p(\mathbb{R};H) \hookrightarrow
L^{\infty}(\mathbb{R};H),\,\,\mbox{for all}\,\, 1 \leq p<\infty.
\end{equation}
\begin{equation}
\label{eq:basicinclusion2}
U^p(\mathbb{R};H) \subseteq U^q(\mathbb{R};H) \subseteq L^{\infty}(\mathbb{R};H),\,\,\mbox{for all}\,\, 1 \leq p<q <\infty.
\end{equation}
A key link between the $U$ and $V$ spaces is the following duality statement:
\begin{equation}
\label{eq:UVduality}
(U^p)^{*}=V^{p'},\,\,\mbox{whenever}\,\,1<p<\infty.
\end{equation}
More precisely, there exists a bilinear pairing
$B: U^p \times V^{p'} \rightarrow \mathbb{C}$ such that the map $T: V^{p'} \rightarrow (U^p)^{*}$ given by $T(v):=B(\cdot,v)$ is an isometric isomorphism. This statement is reminiscent of the duality between $H^1$ and $BMO$ from \cite{FeffermanStein}. For a precise form of the bilinear form $B$, we refer the reader to Section 2 of \cite{HHK}.

As in \cite{KT2,KT3}, it is convenient to define:
\begin{equation}
\label{eq:U2Delta}
\|u\|_{U^2_{\Delta}H^{\alpha}}:=\|e^{-it\Delta}u\|_{U^2(\mathbb{R};H^{\alpha}(\Lambda))}.
\end{equation}
\begin{equation}
\label{eq:V2Delta}
\|u\|_{V^2_{\Delta}H^{\alpha}}:=\|e^{-it\Delta}u\|_{V^2(\mathbb{R};H^{\alpha}(\Lambda))}.
\end{equation}

In \cite{HTT,IP}, as well as in \cite{GuoOhWang,NS}, one also considers the following spaces:
\begin{equation}
\label{eq:spaceX}
\|u\|_{X^{\alpha}}:= \big( \sum_{\xi \in \mathbb{Z}^3} \langle \xi \rangle^{2\alpha} \|e^{it|\xi|^2} \widehat{u(t)}(\xi)\|_{U^2(\mathbb{R}_t;\mathbb{C})}^2 \big)^{\frac{1}{2}}.
\end{equation}
\begin{equation}
\label{eq:spaceY}
\|u\|_{Y^{\alpha}}:=\big( \sum_{\xi \in \mathbb{Z}^3} \langle \xi \rangle^{2\alpha} \|e^{it|\xi|^2}
\widehat{u(t)}(\xi)\|_{V^2(\mathbb{R}_t;\mathbb{C})}^2 \big)^{\frac{1}{2}}.
\end{equation}
We note that the above spaces are reminiscent of the $X^{s,b}$ spaces which are commonly used in the theory of dispersive equations \cite{B3,Tao}.

In \cite{HTT}, it is noted that the following bound holds:
\begin{equation}
\label{eq:linearestimate}
\|e^{it\Delta}f\|_{X^{\alpha}} \leq \|f\|_{H^{\alpha}}
\end{equation}
as well as the following inclusion of spaces:
\begin{equation}
\label{eq:spaceinclusion}
U^2_{\Delta}H^{\alpha} \hookrightarrow X^{\alpha} \hookrightarrow Y^{\alpha} \hookrightarrow V^2_{\Delta}H^{\alpha}.
\end{equation}
Let us also note that:
\begin{equation}
\label{eq:LinftyY}
X^{\alpha} \hookrightarrow Y^{\alpha} \hookrightarrow L^{\infty}(\mathbb{R}; H^{\alpha}(\Lambda)).
\end{equation}
To deduce the last inclusion, we note that, by unitarity of $e^{-it\Delta}$ on $H^{\alpha}(\Lambda)$, one has:
$$\|u\|_{L^{\infty}(\mathbb{R};H^{\alpha}_x(\Lambda))} = \|e^{-it\Delta}u\|_{L^{\infty}(\mathbb{R};H^{\alpha}_x(\Lambda))}$$
which by $(\ref{eq:basicinclusion})$, the definition of $V^2_{\Delta}H^{\alpha}$ and $(\ref{eq:spaceinclusion})$ is:
$$\lesssim \|e^{-it\Delta}u\|_{V^2(\mathbb{R};H^{\alpha})}=\|u\|_{V^2_{\Delta}H^{\alpha}} \lesssim \|u\|_{Y^{\alpha}}.$$

Similarly to $(\ref{eq:UVduality})$, the following duality relation holds \cite{HTT}:
\begin{equation}
\label{eq:duality}
(X^{\alpha}(I))^{*}=Y^{-\alpha}(I).
\end{equation}
Here, $I$ is a time interval and $X^{\alpha}(I), Y^{-\alpha}(I)$ denote the corresponding restriction spaces.

Moreover, Proposition 2.11 in \cite{HTT} tells us that for $f \in L^1([0,T];H^{\alpha}(\Lambda))$, one has:
$$\int_0^t e^{i(t-\tau)\Delta}f(\tau) d\tau \in X^{\alpha}([0,T])$$
and:
\begin{equation}
\label{eq:Duhamelbound}
\left\|\int_0^t e^{i(t-\tau)\Delta} f(\tau) d\tau\right\|_{X^{\alpha}([0,T])} \lesssim \|f\|_{L^1([0,T];H^{\alpha}(\Lambda))}.
\end{equation}
This is a useful estimate on the Duhamel terms.

Given a dyadic integer $N$, let $P_N$ denote the Littlewood-Paley projection to frequencies of order $N$. In other words, we start with a non-negative, even function $\psi \in C_0^{\infty}(-2,2)$, which equals 1 on $[-1,1]$. For $p \in \mathbb{Z}^3$, we let $\psi_N(p)$ equal:
\begin{equation}
\notag
\begin{cases}
\psi \big(\frac{|p|}{N}\big)-\psi \big(\frac{2|p|}{N}\big),\,\mbox{for}\,\,N \geq 2.\\
\psi(|p|),\,\,\mbox{for}\,\,N=1.
\end{cases}
\end{equation}
Then, $P_N$ is defined on $L^2(\Lambda)$ as:
$$(P_N f)\,\,\widehat{}\,\,(p):=\psi_N(p) \widehat{f}(p).$$
In our paper, we will use the following estimate, which was first proved in \cite{ }.
\begin{proposition}
\label{PropositionStar}(Proposition 3.5 in \cite{HTT})
There exists $\delta>0$ such that for any $N_1 \geq N_2 \geq N_3 \geq 1$, and any finite interval $I \subseteq \mathbb{R}$, one has:
\begin{equation}
\label{eq:Star}
\|\prod_{j=1}^{3} P_{N_j}u_j \|_{L^2(I \times \Lambda)} \lesssim N_2 N_3 \max\,
\Big\{\frac{N_3}{N_1},\frac{1}{N_2}\Big\}^{\delta} \prod_{j=1}^{3} \|P_{N_j}u_j\|_{Y^0(I)}.
\end{equation}
The implied constant depends only on the size of the interval $I$.
\end{proposition}
We remark that, in \cite{HTT}, this result is stated with the $Y^0$ norm instead of the $Y^0(I)$ norm on the right-hand side. However, the result with the $Y^0(I)$ norm immediately follows by the definition of the restriction norm.
\section{Spacetime Bound}
\label{Spacetime Bound}
An important step in our proof will be the spacetime bound for the free evolution
$\mathcal{U}^{(k)}(t)$ defined in \eqref{eq:FreeEvolution}. We note that the bound resembles spacetime estimates for the free evolution used in the theory of dispersive equations. This is a manifestation of the connection between the Gross-Pitaevski hierarchy and the Schr\"{o}dinger equation. The bound that we prove is:
\begin{proposition}
\label{Spacetime Bound}
For $\alpha>1$, there exists $C_1=C_1(\alpha)>0$ such that, for all sequences $(\gamma_0^{(n)}) \in \mathcal{H}^{\alpha}$, for all $k \in \mathbb{N}$, and for all $j \in \{1,\ldots,k\}$, one has:
\begin{equation}
\label{eq:spacetime}
\|S^{(k,\alpha)}B_{j,k+1} \, \mathcal{U}^{(k+1)}(t)\gamma_0^{(k+1)}\|_{L^2([0,2\pi]\times \Lambda^k \times \Lambda^k)} \leq C_1 \|S^{(k+1,\alpha)} \gamma_0^{(k+1)}\|_{L^2(\Lambda^{k+1} \times \Lambda^{k+1})}.
\end{equation}
\end{proposition}
We note that, by using a Sobolev embedding argument as in \cite{CL}, one can directly prove the bound $(\ref{eq:spacetime})$ when $\alpha>\frac{3}{2}$. Furthermore, an application of the number theoretic techniques used in \cite{KSS} adapted to the three-dimensional setting gives the bound $(\ref{eq:spacetime})$ when $\alpha>\frac{5}{4}$.
Moreover, we show that the condition $\alpha>1$ is sharp by proving the following two claims:
The first claim states that, in the endpoint case $\alpha=1$, one obtains an upper bound involving a logarithmic loss of derivative.
\begin{proposition}
\label{Alpha1log}
Suppose that $N \geq 1$ and suppose that $(\gamma_{0,N}^{(k)}) \in \mathcal{H}^1$ is a sequence such that for all $k \in \mathbb{N}$, one has $(\gamma_{0,N}^{(k)})\,\,\widehat{}\,\,(n_1,\ldots,n_k; n'_1, \ldots, n'_k)=0$ if some $|n_j|>N$ or if some $|n'_j|>N$. Then, there exists a constant $C_1>0$ such that for all $k \in \mathbb{N}$ and for all $j \in \{1,\ldots,k\}$, one has:
\begin{equation}
\label{logAlpha1}
\|S^{(k,1)} B_{j,k+1} \mathcal{U}^{(k+1)}(t) \gamma_{0,N}^{(k+1)}\|_{L^2([0,2\pi] \times \Lambda^k \times \Lambda^k)} \leq C_1 \ln N \, \|S^{(k+1,1)} \gamma_{0,N}^{(k+1)}\|_{L^2(\Lambda^{k+1} \times \Lambda^{(k+1)})}.
\end{equation}
\end{proposition}
The second claim states that, when $\alpha=1$, one cannot choose a uniform constant $C_1>0$ such that $(\ref{eq:spacetime})$ holds.
\begin{proposition}
\label{Alpha1unbounded}
There does not exist $C_1>0$ such that for all $(\gamma_0^{(n)}) \in \mathcal{H}^1$, $k \in \mathbb{N}$ and $j \in \{1,\ldots,k\}$ one has:
\begin{equation}
\label{Alpha1}
\|S^{(k,1)}B_{j,k+1} \, \mathcal{U}^{(k+1)}(t)\gamma_0^{(k+1)}\|_{L^2([0,2\pi] \times \Lambda^k \times \Lambda^k)} \leq C_1 \|S^{(k+1,1)} \gamma_0^{(k+1)}\|_{L^2(\Lambda^{k+1} \times \Lambda^{k+1})}.
\end{equation}
\end{proposition}
\begin{remark}
In the subsequent work \cite{SoSt} of the second two authors, it will be shown that the spacetime bound holds in a probabilistic sense whenever $\alpha>\frac{3}{4}$. More precisely, the collision operator is replaced by a randomized collision operator. A study of the local-in-time theory of the hierarchies obtained from the randomized collision operators in the work of the second author \cite{VS}.
\end{remark}
\medskip
Before we give the details of the proofs of the above results, let us make some preliminary observations.
\medskip
With $B^{\pm}_{j,k+1}$ defined as in \eqref{Bjk+} and \eqref{Bjk-}, we will prove the bound for the
term given by: $$S^{(k,\alpha)}B^{+}_{j,k+1} \,\mathcal{U}^{(k+1)}(t)\gamma_0^{(k+1)}$$ instead of for: $$S^{(k,\alpha)}B_{j,k+1} \,\mathcal{U}^{(k+1)}(t)\gamma_0^{(k+1)}.$$ The bound for the expression with $B^{-}_{j,k+1}$ is proved in an analogous way (one just exchanges the $n$'s and $m$'s), and the general claim then follows since by \eqref{Bjk}, we know that $B_{j,k+1}=B^{+}_{j,k+1}-B^{-}_{j,k+1}$. For simplicity of notation, let us consider without loss of generality the case when $j=1$. The other cases follow by symmetry.

We note that:
\begin{eqnarray}\label{FTSpace}
&&\big(S^{(k,\alpha)}B^{+}_{1,k+1} \, \mathcal{U}^{(k+1)}(t)\gamma_0^{(k+1)}\big)\,\,\widehat{}\,\,(t,\vec{n}_k;\vec{n}'_k)=\\\notag
&&\sum_{n_{k+1},n'_{k+1} \in \mathbb{Z}^3} e^{it(- |n_1-n_{k+1}+n'_{k+1}|^2-|\vec{n}_{k+1}|^2 + |n_1|^2 + |\vec{n}'_{k+1}|^2)}\\ \notag
&\times&
\prod_{j=1}^{k} \langle n_j \rangle^{\alpha} \prod_{j=1}^{k} \langle n'_j \rangle^{\alpha}\,
\cdot (\gamma_0^{(k+1)}) \,\,\widehat{}\,\,(n_1-n_{k+1}+n'_{k+1},n_2,\ldots,n_{k+1}; n'_1,\ldots, n'_{k+1}).\notag
\end{eqnarray}
Hence, the above expression is periodic in time with period $2\pi$. Thus, we need to restrict the time integral to the interval $[0,2\pi]$.
For fixed $\tau \in \mathbb{Z}$ and $p \in \mathbb{Z}^3$, we define:
\begin{equation}
\label{eq:expressionI}
I=I(\tau,p):=\sum_{n,m \in \mathbb{Z}^3} \frac{\delta(\tau+|p-n-m|^2+|n|^2-|m|^2) \langle p \rangle^{2\alpha}}
{\langle p-n-m \rangle^{2\alpha} \langle n \rangle^{2\alpha} \langle m \rangle^{2\alpha}}.
\end{equation}
where $\delta$ denotes the Kronecker delta function on $\mathbb{Z}$.
Arguing as in the proof of Theorem 1.3 in \cite{KM} or as in the Propositions 7.2 and 7.5 in \cite{KSS}, the estimate:
\begin{equation}
\label{eq:spacetimeFT+}
\|S^{(k,\alpha)}B^{+}_{j,k+1} \, \mathcal{U}^{(k+1)}(t)\gamma_0^{(k+1)}\|_{L^2([0,2\pi] \times \Lambda^k \times \Lambda^k)} \leq C_1 \|S^{(k+1,\alpha)} \gamma_0^{(k+1)}\|_{L^2(\Lambda^{k+1} \times \Lambda^{k+1})}
\end{equation}
follows from:
\begin{equation}
\label{eq:Proposition1bound}
I \leq C_1,\mbox{uniformly in}\, \tau\,\mbox{and}\, p.
\end{equation}

Let us recall this reduction in more detail. By taking the Fourier transform on $[0,2\pi] \times \Lambda^k \times \Lambda^k$ of the function on the left-hand side, and by taking the Fourier transform on $\Lambda^{k+1} \times \Lambda^{k+1}$ of the function on the right-hand side, Plancherel's theorem implies that $(\ref{eq:spacetimeFT+})$ is equivalent to:
$$\|\big(S^{(k,\alpha)}B^{+}_{1,k+1} \, \mathcal{U}^{(k+1)}(t)\gamma_0^{(k+1)}\big)\,\,\widetilde{}\,\,\|_{\ell^2(\mathbb{Z}_{\tau}
\times \mathbb{Z}^{3k} \times \mathbb{Z}^{3k})} \leq C_1 \|\big(S^{(k+1,\alpha)}\,\gamma_0^{(k+1)}\big)\,\,\widehat{}\,\, \|_{l^2(\mathbb{Z}^{3(k+1)} \times \mathbb{Z}^{3(k+1)})}
$$
It can be shown that:
$$\big(S^{(k,\alpha)}B^{+}_{1,k+1} \, \mathcal{U}^{(k+1)}(t)\gamma_0^{(k+1)}\big)\,\,\widetilde{}\,\,(\tau,\vec{n}_k;\vec{n}'_k)=$$
$$\sum_{n_{k+1},n'_{k+1} \in \mathbb{Z}^3} \delta(\tau + |n_1-n_{k+1}+n'_{k+1}|^2 + |\vec{n}_{k+1}|^2 - |n_1|^2 - |\vec{n}'_{k+1}|^2) \cdot$$
\begin{equation}
\label{FTSpaceTime}
\prod_{j=1}^{k} \langle n_j \rangle^{\alpha} \prod_{j=1}^{k} \langle n'_j \rangle^{\alpha}
\,\cdot \widehat{\gamma_0}^{(k+1)}(n_1-n_{k+1}+n'_{k+1},n_2,\ldots,n_{k+1}; n'_1,\ldots, n'_{k+1}).
\end{equation}
By applying the Cauchy-Schwarz inequality in $n_{k+1}$ and $n'_{k+1}$, we observe that, for fixed $\tau \in \mathbb{Z}$:
\begin{eqnarray*}
&&|\big(S^{(k,\alpha)}B^{+}_{1,k+1} \, \mathcal{U}^{(k+1)}(t)\gamma_0^{(k+1)}\big)\,\,\widetilde{}\,\,(\tau,\vec{n}_k;\vec{n}'_k)|^2\\
&\leq &\Big[ \sum_{n_{k+1},n'_{k+1} \in \mathbb{Z}^3} \frac{\delta(\tau + |n_1-n_{k+1}+n'_{k+1}|^2 + |\vec{n}_{k+1}|^2 - |n_1|^2 - |\vec{n}'_{k+1}|^2) \langle n_1 \rangle^{2\alpha}}
{\langle n_1-n_{k+1}+n'_{k+1}\rangle^{2\alpha} \langle n_{k+1} \rangle^{2\alpha} \langle n'_{k+1} \rangle^{2\alpha}}\Big]\\
&\times&\Big[\sum_{n_{k+1},n'_{k+1} \in \mathbb{Z}^3} \delta(\tau + |n_1-n_{k+1}+n'_{k+1}|^2 + |\vec{n}_{k+1}|^2 - |n_1|^2 - |\vec{n}'_{k+1}|^2) \Big.\\
&&\Big.|(S^{(k+1,\alpha)} \gamma_0^{(k+1)})\,\,\widehat{}\,\,(n_1-n_{k+1}+n'_{k+1},n_2,\ldots,n_{k+1};\vec{n}'_{k+1})|^2\Big].
\end{eqnarray*}
Hence, in order to prove Proposition \ref{Spacetime Bound}, we need to obtain pointwise bounds on:
$$\tilde{I}(\tau,\vec{n}_k;\vec{n}'_k):=$$
\begin{equation}
\label{eq:Jtaup}
\sum_{n_{k+1},n'_{k+1} \in \mathbb{Z}^3} \frac{\delta(\tau + |n_1-n_{k+1}+n'_{k+1}|^2 + |\vec{n}_{k+1}|^2 - |n_1|^2 - |\vec{n}'_{k+1}|^2) \langle n_1 \rangle^{2\alpha}}
{\langle n_1-n_{k+1}+n'_{k+1}\rangle^{2\alpha} \langle n_{k+1} \rangle^{2\alpha} \langle n'_{k+1} \rangle^{2\alpha}}
\end{equation}
We remark that we are originally looking at:
$$\tau + |n_1-n_{k+1}+n'_{k+1}|^2+|\vec{n}_{k+1}|^2-|n_1|^2-|\vec{n}'_{k+1}|^2=$$
$$= \tau + |n_1-n_{k+1}+n'_{k+1}|^2+|n_1|^2 +$$
$$+\sum_{\ell=2}^{k} |n_{\ell}|^2+|n_{k+1}|^2-|n_1|^2 - \sum_{\ell=1}^{k}|n'_{\ell}|^2-|n'_{k+1}|^2=$$
$$=(\tau + \sum_{\ell=2}^{k}|n_{\ell}|^2 - \sum_{\ell=1}^{k}|n'_{\ell}|^2)
+|n_1-n_{k+1}+n'_{k+1}|^2+|n_{k+1}|^2-|n'_{k+1}|^2$$
Hence, by making the change of variable $\tau \mapsto \tau + \sum_{\ell=2}^{k} |n_{\ell}|^2- \sum_{\ell=1}^{k}|n'_{\ell}|^2$, and by taking $p:=n_1, n:=n_{k+1}$, and $m:=-n_{k+1}'$, it follows that we indeed need to show $(\ref{eq:Proposition1bound})$. The rest of this section is devoted to the proof of $(\ref{eq:Proposition1bound})$.
\subsection{Rewriting the sum in $(\ref{eq:Proposition1bound})$}
\label{RewriteSum}
The first step is to rewrite the sum $I(\tau,p)$ from $(\ref{eq:Proposition1bound})$. By using the identity
\begin{align*}
\tau + |p-n-m|^2 + |n|^2 - |m|^2 & = \tau + |n|^2 + |p-n|^2 - 2 \langle p-n, m \rangle \\
& = \tau + |p|^2 - 2 \langle p-n, n+m \rangle
\end{align*}
we can rewrite the sum as follows:
\begin{align}
I(\tau,p) & = \sum_{m,n \in \mathbb{Z}^3} \frac{\delta(\tau + |p|^2 - 2 \langle p-n, n+m \rangle) \ang{p}^{2\alpha}}{\ang{p-n-m}^{2\alpha} \ang{n}^{2\alpha} \ang{m}^{2\alpha}} \nonumber \\ & = \sum_{m,n \in \mathbb{Z}^3} \frac{\delta(\tau + |p|^2 + 2 \langle n-p,m\rangle) \ang{p}^{2\alpha}}{\ang{p-m}^{2\alpha} \ang{n}^{2\alpha} \ang{m-n}^{2\alpha}} \nonumber \\
& = \sum_{m,n \in \mathbb{Z}^3} \frac{\delta(\tau + |p|^2 - 2 \langle n, m \rangle) \ang{p}^{2\alpha}}{\ang{m-p}^{2\alpha} \ang{n-p}^{2\alpha} \ang{p-n-m}^{2\alpha}}. \label{thesum2}
\end{align}

We will use the form \eqref{thesum2} in order to prove that $I(\tau,p)$ is uniformly bounded for any $\alpha > 1$. The point is that, for fixed $m \in \mathbb{Z}^3$, all the points $n \in \mathbb{Z}^3$ satisfying the condition $\tau+|p|^2-2 \langle n, m \rangle=0$ lie in a plane with normal vector $m$. By symmetry, the same property holds if we reverse the roles of $m$ and $n$. Having rewritten the sum this way, we can use the number theoretic properties of the normal vector.
\begin{remark}
We observe that there is never any decay along $I(|p|^2,p)$, just by evaluating at $(m,n) = (p,p)$.
\end{remark}
\subsection{Some number theory}
\label{NumberTheory}
In order to show that the expression in \eqref{thesum2} is bounded for $\alpha>1$, we need to observe some facts from number theory. As was mentioned in the introduction, we want to estimate the number of lattice points lying in the intersection of a ball with a plane which is normal to a vector in $\mathbb{Z}^3$. The first result which we prove is the following property about sets whose points are sufficiently separated:
\begin{lemma}
Let $X \subset \mathbb{R}^d$ be any set, and let $X_{s}$ be the set of points in $\mathbb{R}^d$ which are a distance less than $s$ to the set $X$. Let $\Lambda \subset \mathbb{R}^d$ be any $r$-separated set (meaning that $|x-x'| \geq r$ for $x,x'$ distinct points in $\Lambda$). Then for any $r' > 0$, \label{seplemma}
\[ \# ( \Lambda \cap X_{r'}) \lesssim (\min \{r,r'\} )^{-d} |X_{r'}|, \]
i.e., the number of points in $\Lambda$ at distance less than $r'$ to $X$ is at most a universal constant (depending only on dimension) times $(\min \{r,r'\} )^{-d}$ and the Lebesgue measure of $X_{r'}$.
\end{lemma}
\begin{proof}
Suppose $x \in X$ and $|x-y| < r'$. Let $\theta := \min \{ \frac{r}{4 r'}, 1\}$. By the triangle inequality, every point in the ball $B := \{z \in \mathbb{R}^d, \left|z - ( \theta x + (1-\theta) y ) \right| < \frac{1}{4} \min \{r,r'\}\}$
will have
\begin{align*}
|z - y| & \leq \left|z - (\theta x + (1-\theta) y) \right| + \theta |x-y| < \frac{r}{4} + \frac{r}{4} = \frac{r}{2}, \\
|z - x| & \leq \left|z - (\theta x + (1-\theta) y) \right| + (1 - \theta) |x-y| \\ & < \frac{1}{4} \min \{r,r'\} + \max \left\{ r' - \frac{r}{4}, 0 \right\} \leq r'.
\end{align*}
Consequently, for each $y \in \Lambda$ of distance less than $r'$ to the set $X$, we may find a ball $B_y$ of radius $\frac{1}{4} \min \{r,r'\}$ which is contained both in the ball of radius $\frac{r}{2}$ centered at $y$ and in $X_{r'}$. Since $\Lambda$ is $r$-separated, these balls must be disjoint. Thus we have the trivial estimate that the measure of $X_{r'}$ must be bounded below by the number of points in $\Lambda$ of distance less than $r'$ to the set $X$ times the volume of the ball of radius $\frac{1}{4} \min\{r,r'\}$. The constant clearly depends only on the dimension.
\end{proof}
Given $m \in \mathbb{Z}^d$, let $m_*$ equal $m$ divided by the greatest common divisor of its coordinates (and define $m_*:= 0$ when $m=0$).
The following lemma gives us a useful way of relating the above defined quantity $m_*$ to the diameter of a set:
\begin{lemma}
\label{sum1mstar}
Let $\mathcal{E} \subset \mathbb{Z}^3 \setminus \{0\}$ be any set of diameter $R \geq 1$.
Then
\begin{equation} \sum_{m \in {\mathcal E}} \frac{1}{|m_*|} \lesssim R^2 \label{numth1} \end{equation}
for some universal implied constant.
\end{lemma}
\begin{proof}
For each $q \in \mathbb{Z}^3$, let $\varphi(q)$ equal the number of points $m \in {\mathcal E}$ such that $m_* = q$. We have the equality
\[ \sum_{q \in \mathbb{Z}^3} \frac{\varphi(q)}{|q|} = \sum_{m \in {\mathcal E}} \frac{1}{|m_*|}. \]
Several observations are in order. The first is that the sum over all $q$ of $\varphi(q)$ is exactly the cardinality of ${\mathcal E}$. Because ${\mathcal E}$ has diameter $R$, it will be contained in some ball of radius $R$, and since the integer lattice is $1$-separated, the number of points in ${\mathcal E}$ will (by Lemma \ref{seplemma}) be at most a constant times the volume of the ball of radius $(R+1)$, which is (by assumption on $R$) itself bounded by a constant times $R^3$. Next we make the observation that
\begin{equation}
\label{eq:varphibound}
\varphi(q) \leq 1 + R |q|^{-1}.
\end{equation}
This follows since $\varphi(q) = N > 1$ implies that there are positive integers $n_1$ and $n_2$ with $|n_1 - n_2| \geq N-1$ such that $n_1 q \in \mathcal{E}$ and $n_2 q \in \mathcal{E}$. On the other hand, by assumption, points in $\mathcal{E}$ are separated by distance at most $R$). Hence,
\begin{align*}
|q| \cdot (N-1) \leq |q| \cdot |n_1-n_2| = |qn_1-qn_2| \leq R,
\end{align*}
which indeed implies $(\ref{eq:varphibound})$.
We conclude by estimating:
\begin{align*}
\mathop{\sum_{q \in \mathbb{Z}^3}}_{2^{j-1} < |q| < 2^j} \frac{\varphi(q)}{|q|} \leq 2^{-j+1} ( 1 + R 2^{-j+1}) 2^{3j}
\end{align*}
when $j \geq 0$ (we have simply bounded $|q|^{-1}$, $\varphi(q)$ and the number of $q$ in the annulus. Fixing any $N \in \mathbb{N}$, we may sum these estimates over $0 \leq j \leq N$. For the remaining piece, we know that $|q|^{-1} < 2^{-N}$ and that the sum over $\varphi$ on the remainder term is at most a constant times $R^3$. Thus
\[ \sum_{m \in {\mathcal E}} \frac{1}{|m_*|} \lesssim 2^{2 N} + R 2^{N} + 2^{-N} R^3. \]
Since $R \geq 1$, we can find $N \in \mathbb{N}$ such that $2^N \sim R$. The estimate \eqref{numth1} follows for this choice of $N$.
\end{proof}
Let us recall the definition of a lattice and of its determinant \cite{Lekkerkerker}
\begin{definition}
A $k$-dimensional lattice in $\mathbb{R}^d$ is a discrete subgroup $L$ of $\mathbb{R}^d$ of rank $k$. Given such an $L$, we define its determinant to be the smallest positive volume of a $k$-parallelepiped whose vertices belong to $L$.
\end{definition}
A fundamental number theoretic fact is:
\begin{lemma}
For fixed $m \in \mathbb{Z}^3 \setminus \{0\}$, the determinant of the lattice $$\Lambda_m := \{x \in \mathbb{Z}^3, \langle m,x \rangle = 0 \}$$ is $|m_*|$ (i.e., the Euclidean norm of the vector $m$ divided by the greatest common divisor of the entries of $m$).
\end{lemma}
\begin{proof}
This lemma is an immediate corollary of a theorem of P. McMullen \cite{mcmullen1984}. McMullen's theorem states that if $E$ and $E'$ are orthogonal rational subspaces of Euclidean space, then the determinants of the lattices $\mathbb{Z}^d \cap E$ and $\mathbb{Z}^d \cap E'$ are equal.
This theorem was later generalized by Schnell \cite{schnell1992}. In our case, $\Lambda_{m}$ has a dual lattice $\Lambda_{m}'$ which consists of all integer points which are multiples of $m$ itself. This lattice is one-dimensional, and its determinant is merely equal to the Euclidean norm of the smallest nonzero element, which must equal $|m_*|$.
\end{proof}
The concept of the determinant of a lattice allows us to prove the main counting lemma of this subsection. This is the result which will allow us to estimate the number of lattice points in the intersection of a ball and a plane. We note that it gives us a substantial improvement over the a priori bound $O(R^2)$, which we would obtain for generic values of the normal vector.
\begin{lemma}
\label{MainCountingLemma}
For fixed $m \in \mathbb{Z}^3 \setminus \{0\}$ and $c \in \mathbb{Z}$, if ${\mathcal E} \subset \mathbb{Z}^3$ is any set of diameter $R \geq 1$ then
\begin{equation} \# \left( \mathcal{E} \cap \{ x \in \mathbb{Z}^3, \langle m,x \rangle = c\} \right) \lesssim R + \frac{R^2}{|m_*|}. \label{countingineq}
\end{equation}
\end{lemma}
\begin{proof}
Translating ${\mathcal E}$ as necessary, it suffices to assume $c=0$ (since if $q$ is any point belonging to the intersection, then the cardinality of the intersection will equal the cardinality of the intersection $({\mathcal E} - q) \cap \Lambda_m$.
Let $a \in \Lambda_m$ be any nonzero element of minimal norm. On any line in $\Lambda_m$ with direction $a_1:=\frac{a}{|a|^2}$, the points of $\Lambda_m$ will be separated by a distance of at least $|a|$ (if this were not the case, a translation of that line back to the origin would show that $|a|$ cannot be minimal). On the other hand, since the determinant of the lattice is $|m_*|$, this means that $\Lambda_m$ may be written as a countable union of equally spaced lines in the $a_1$-direction which are separated by a distance of at least $|m_*| |a|^{-1}$ (since any parallelogram will have area at least $|m_*|$). Now it's clear, however, that if ${\mathcal E}$ has diameter $R$, then at most $1 + |a| |m_*|^{-1} R$ of these lines may intersect ${\mathcal E}$, and the number of points on each such line which belong to $\Lambda_m \cap {\mathcal E}$ can be at most $1 + |a|^{-1} R$. Thus
$$\# \left( \mathcal{E} \cap \{ x \in \mathbb{Z}^3, \langle m, x \rangle = c\} \right) \leq 1 + (|a|^{-1} + |a| |m_*|^{-1}) R + |m_*|^{-1} R^2. $$
Because $|a| \geq 1$ (due to the $1$-separation of $\mathbb{Z}^3$), $|a| \leq R$ (since $\mathcal{E}$ has diameter $R$), and $R \geq 1$, the conclusion \eqref{countingineq} immediately follows.
\end{proof}
\subsection{The proof of Proposition \ref{Spacetime Bound}}
\label{ConclusionProofProposition}
We can now put together all of the previous estimates and estimate the sum in Proposition \ref{Spacetime Bound}. More precisely, let us prove the bound $(\ref{eq:Proposition1bound})$, which, in turn, implies Proposition \ref{Spacetime Bound} from the earlier discussion. In the proof, it will be necessary to dyadically localize all of the factors. We will then argue by cases and sum up all the estimates to obtain the bound.
\begin{proof} (of Proposition \ref{Spacetime Bound})
For any $j := (j_1,j_2,j_3) \in \mathbb{N}_0^3$, let $E_{\tau,p}(j)$ be the subset of $\mathbb{Z}^3 \setminus \{0\} \times \mathbb{Z}^3 \setminus \{0\}$ on which
\[ |m-p| < 2^{j_1}, |n-p| < 2^{j_2}, \ |m+n-p| < 2^{j_3}, \]
and $0 = \tau + |p|^2 - 2 \langle n, m \rangle$. If $j_1 > 0$, we further assume that $|m-p| \geq 2^{j_1-1}$ and likewise for $j_2$ (assume $|n-p| \geq 2^{j_2-1}$) and $j_3$ (assume $|m+n-p| \geq 2^{j_3-1}$).
Each pair $(m,n) \in \mathbb{Z}^3 \setminus \{0\} \times \mathbb{Z}^3 \setminus \{0\}$ satisfying the relation $0 = \tau + |p|^2 - 2 \langle n, m \rangle$ will belong to $E_{\tau,p}(j)$ for exactly one index $j$; for that specific index we will have
\[ \frac{1}{\ang{m-p}^{2 \alpha} \ang{n-p}^{2 \alpha} \ang{m+n-p}^{2 \alpha}} \leq 2^{-2\alpha (j_1 + j_2 + j_3 - 3)}. \]
The above inequality is immediate when $j_1,j_2,j_3>0$. If any of the indices are non-positive, we use the fact that $\ang{\cdot} \geq 1$.
Now consider the sum $I(\tau,p)$. From \eqref{thesum2}, we note that:
\begin{equation}
\begin{split}
I(\tau,p) \lesssim \delta( \tau + |p|^2) & \sum_{m} \frac{1}{ \ang{m-p}^{4 \alpha} } \\
& + \ang{p}^{2 \alpha} \sum_{j_1,j_2,j_3 = 0}^{\infty} 2^{- 2\alpha (j_1 + j_2 + j_3)} \# E_{\tau,p}(j).
\end{split} \label{mainest}
\end{equation}
The first sum on the right-hand side arises from those terms in the sum for $I(\tau,p)$ at which either $m = 0$ or $n = 0$ (note these cases are symmetric) and is finite whenever $\alpha>\frac{3}{4}$.
The second sum on the right-hand side of \eqref{mainest} is substantially more difficult to control. Let $j_{\min}, j_{\mid}$, and $j_{\max}$ equal the entries of $j$ arranged in increasing order. By the triangle inequality, $|p| \leq |m-p| + |n-p| + |m+n-p| \leq 2^{j_1} + 2^{j_2} + 2^{j_3}$. Furthermore, since each of $j_1,j_2,j_3$ is nonnegative, we also have $1 \leq \frac{7}{9} (2^{j_1} + 2^{j_2} + 2^{j_3})^2$, and so
\begin{align*}
1+|p|^2 \leq 1 + (2^{j_1}+2^{j_2}+2^{j_3})^2 \leq \frac{16}{9}(2^{j_1}+2^{j_2}+2^{j_3})^2 \leq 16 \cdot 2^{2 j_{\max}}.
\end{align*}
Hence, $\ang{p} \leq 2^{j_{\max}+2}$ whenever $E_{\tau,p}(j) \neq \emptyset$.
We can rewrite this inequality as:
\begin{equation}
\label{eq:jmaxlogp}
j_{\max} \geq \log_2 \ang{p} -2.
\end{equation}
If we can establish the inequality
\begin{equation} \# E_{\tau,p}(j) \lesssim 2^{2 j_{\min} + 2 j_{\mid}} \label{sizee} \end{equation}
then we will have
\begin{align*}
\ang{p}^{2 \alpha} & \sum_{j_1,j_2,j_3 = 0}^{\infty} 2^{- 2\alpha (j_1 + j_2 + j_3)} \# E_{\tau,p}(j) \\ & \lesssim \ang{p}^{2 \alpha} \mathop{\sum_{j_1,j_2,j_3 = 0}^\infty}_{j_{\max} \geq \log_2 \ang{p} - 2} 2^{- 2\alpha (j_{\min} + j_\mid + j_{\max} ) + 2 j_{\min} + 2 j_{\mid}}
\end{align*}
\begin{equation}
\label{alphale1}
\lesssim \ang{p}^{2 \alpha} \left( \sum_{j_{\max} \geq \log_2 \ang{p} - 2} 2^{- 2\alpha j_{\max}} \right) \left( \sum_{j=0}^\infty 2^{2 ( 1 - \alpha) j} \right)^2
\end{equation}
which is uniformly bounded in $\tau$ and $p$ so long as $\alpha > 1$. Here, we are using the fact that the sum of the tail of a convergent geometric series is comparable to the largest term with a constant depending only on the ratio of adjacent terms. This, in turn, implies that the sum over $j_{\max}$ will cancel the factor of $\ang{p}^{2 \alpha}$. To establish \eqref{sizee} we must consider several cases.

\vspace{5mm}

\textbf{Case 1:} $j_1 = \min \{j_1,j_2,j_3\}$.
For each fixed $m$, the pairs $(m,n) \in E_{\tau,p}(j)$ satisfy the inequalities $|m-p| \leq 2^{j_1}, |n-p| \leq 2^{j_2}, |m+n-p| \leq 2^{j_3}$. In particular, for $m$ fixed, since $p$ is also fixed, the $n$ coordinates must lie in a ball of radius $2^{\min \{j_2,j_3\}}$. We use Lemma \ref{MainCountingLemma} in order to deduce:
\begin{align*}
\# E_{\tau,p}(j) & \leq \mathop{\sum_{|m-p| \leq 2^{j_1}}}_{m \neq 0} \mathop{\sum_{|n-p| \leq 2^{j_2}}}_{|n+m-p| \leq 2^{j_3}} \delta( \tau + |p|^2 - 2 \langle n, m \rangle) \\
& \lesssim \mathop{\sum_{|m-p| \leq 2^{j_1}}}_{m \neq 0} \left[ 2^{\min\{j_2,j_3\}} + \frac{2^{2 \min \{j_2,j_3\}}}{|m_*|} \right] \end{align*}
Furthermore, by Lemma \ref{sum1mstar} it follows that:
\[ \# E_{\tau,p}(j) \lesssim 2^{3 j_1 + \min \{j_2,j_3\}} + 2^{2 j_1 + 2 \min \{j_2,j_3\}} \lesssim 2^{2 j_1 + 2 \min \{ j_2,j_3\}}. \]

\vspace{5mm}

\textbf{Case 2:} $j_2 = \min \{j_1,j_2,j_3\}$.
By symmetry between $n$ and $m$, we get
\[ \# E_{\tau,p}(j) \lesssim 2^{2 j_2 + 2 \min \{j_1,j_3\}} \]
in exactly the same manner just described.

\vspace{5mm}

\textbf{Case 3:} $j_3 = \min \{j_1,j_2,j_3\}$. Without loss of generality, we may assume $j_1$ and $j_2$ are both positive (otherwise one of the previous cases would also hold). This means we may assume $|m-p| \geq 2^{j_1 - 1}$ and $|n-p| \geq 2^{j_2 - 1}$ if necessary. For any $k := (k_1,k_2,k_3) \in \mathbb{Z}^3$, let
$B_k := [2^{j_3} k_1, 2^{j_3} k_1 + 2^{j_3}-1] \times [2^{j_3} k_2, 2^{j_3} k_2 + 2^{j_3}-1] \times [ 2^{j_3} k_3 + 2^{j_3} - 1]$.
For any pair $(k,k') \in \mathbb{Z}^3 \times \mathbb{Z}^3$, we have $E_{\tau,p}(j) \cap (B_k \times B_{k'})$ has diameter at most comparable to $2^{j_3}$, and so just as in the previous cases, we will have an estimate
\[ \# ( E_{\tau,p}(j) \cap (B_k \times B_{k'})) \lesssim 2^{4 j_3}. \]
Namely, for $m$ fixed, the $n$ coordinates lie in a ball of radius $\sim 2^{j_3}$. Moreover, the $m$ coordinates lie in a ball of radius $\sim 2^{j_3}$.
In order to establish the estimate
\[ \# E_{\tau,p}(j) \lesssim 2^{2 j_3 + 2 \min \{j_1,j_2\}}, \]
it suffices to show that the number of pairs $(k,k')$ for which $E_{\tau,p}(j) \cap (B_k \times B_{k'})$ is nonempty can be at most a uniform constant times $2^{2 \min\{j_1,j_2\} - 2 j_3}$. To establish this estimate we appeal to Lemma \ref{seplemma}.
For any pair $(k,k') \in \mathbb{Z}^3 \times \mathbb{Z}^3$, we will have $(B_k \times B_{k'}) \cap E_{\tau,p}(j) \neq \emptyset$ only when the center point of $B_{k} \times B_{k'}$ is at Euclidean distance less than $2^{j_3+1}$ from the set $E_{\tau,p}(j)$. Furthermore, the centers of such boxes form a $2^{j_3}$-separated set. We conclude by Lemma \ref{seplemma} that
$$\# \{ (k,k') \in \mathbb{Z}^3 \times \mathbb{Z}^3, E_{\tau,p}(j) \cap ( B_k \times B_{k'}) \neq \emptyset\} \lesssim 2^{-6 j_3} | X_{2^{j_3+1}}|,$$
where $X_{2^{j_3+1}}$ refers to the $2^{j_3+1}$-neighborhood of $E_{\tau,p}(j)$. Here, we note that the dimension $d=6$.
It follows that we need only show that the measure of this neighborhood of $E_{\tau,p}(j)$ is at most a uniform constant times $2^{2 \min \{j_1,j_2\} + 4 j_3}$. It will be convenient to rewrite the relation $\tau + |p|^2 - 2 \langle n, m \rangle = 0$ as
\begin{equation}
\label{relation}
2 \tau + |p|^2 + |n-m|^2 = \langle m+n-p,m+n+p \rangle
\end{equation}
and split into two subcases.

\vspace{5mm}

\textbf{Case 3(a):} $|n-m| \geq 2^{\min \{j_1,j_2\}+4}$ for every $(m,n) \in E_{\tau,p}(j)$.
By virtue of the inequalities
\begin{equation}
\label{dyadiclocalizationmnp}
2 \min \{ |m-p|, |n-p| \} \geq 2^{\min\{j_1,j_2\}} \geq 2^{j_3} \geq |m+n-p|
\end{equation}
and
\begin{equation}
\label{triangleineq}
|m-n| + \min \{ |m-p|, |n-p| \} \geq \max \{|m-p|,|n-p|\},
\end{equation}
we must have that:
\begin{equation} 2|n-m| + 10 \min \{ |m-p|, |n-p|\} \geq |m+n+p|, \label{dichot} \end{equation}
for every pair $(m,n) \in E_{\tau,p}(j)$.
In order to prove \eqref{dichot}, we can argue by symmetry and assume, without loss of generality, that $|m-p| \geq |n-p|$. Hence, it follows that:
$$2|m-p| \geq 2|n-p| \geq |m+n-p|.$$
In order to obtain the last inequality, we used \eqref{dyadiclocalizationmnp}.
We also note:
$$|m-p|+|n-p|+|m+n-p| \geq |p|.$$
Furthermore, we use \eqref{triangleineq} to deduce:
\begin{eqnarray*}
&&|n-m|+4|n-p|=|n-m|+|n-p|+|n-p|+2|n-p| \geq \\
&&|m-p|+|n-p|+|m+n-p| \geq |p|.
\end{eqnarray*}
Consequently,
\begin{align*}
2|n-m|+10|n-p| \geq 2|p| + 2|n-p| \geq 2|p|+|m+n-p| \geq |m+n+p|,
\end{align*}
which proves \eqref{dichot}. Here, we again applied \eqref{dyadiclocalizationmnp}.
From \eqref{dichot}, we can deduce that $|n-m| \geq \frac{|m+n+p|}{4}$.
Namely, we know that $\min\{|m-p|,|n-p|\} \leq 2^{\min\{j_1,j_2\}}\leq \frac{1}{16}|m-n|$ and so:
$$2|n-m|+\frac{5}{8}|n-m|=\frac{21}{8} |n-m| \geq |m+n+p|$$
which implies that $|n-m| \geq \frac{8|m+n+p|}{21} \geq \frac{|m+n+p|}{4}.$
Furthermore, by hypothesis $|n-m| \geq 2^{\min \{j_1,j_2\}+4}$. From these two bounds, combined with \eqref{relation}, the fact that $|m+n-p|\leq 2^{j_3} \leq 2^{\min \{j_1,j_2\}}$, and the Cauchy-Schwarz inequality, we may conclude that:
\begin{equation}
\label{Case3aestimate1}
2 \tau + |p|^2 \leq - |n-m|^2 + 2^{j_3} |m+n+p| \leq - 2^{2 \min \{j_1,j_2\} + 6}
\end{equation}
\begin{equation}
\label{Case3aestimate2}
\left| |m-n| - \left| 2 \tau + |p|^2 \right|^{\frac{1}{2}} \right| \leq 2^{j_3+2}.
\end{equation}
In order to deduce \eqref{Case3aestimate2}, we use observe that:
\begin{align*}
\left| |2\tau+p^2|^{\frac{1}{2}} - |n-m| \right| \cdot \left| |2\tau+p^2|^{\frac{1}{2}} + |n-m| \right|=
\left| |2\tau+p^2| - |n-m|^2 \right| \leq \left| 2\tau+p^2 + |n-m|^2 \right| \\
= \left| \langle m+n-p,m+n+p \rangle \right| \leq |m+n-p| |m+n+p| \leq 2^{j_3} \cdot 4|n-m|
\end{align*}
Here, we deduced the fact that
$\left| |2\tau+p^2| - |n-m|^2 \right| \leq \left| 2\tau+p^2 + |n-m|^2 \right|$
by the triangle inequality.
Since $\left| |2\tau+p^2|^{\frac{1}{2}} + |n-m| \right| \geq |n-m|$, the estimate \eqref{Case3aestimate2} follows.
Consequently, we have shown that when $(m,n) \in E_{\tau,p}(j)$, the triple $m-n$ must lie between two spheres of radius at least comparable to $2^{\min\{j_1,j_2\}}$ whose separation is at most comparable to $2^{j_3}$. The first fact follows from
\eqref{Case3aestimate1}, whereas the second fact follows from \eqref{Case3aestimate2}.
We note that all pairs $(m,n) \in E_{\tau,p}(j)$ have to satisfy $|m+n-p| \leq 2^{j_3}$ and $\min\{|m-n-p|,|m-n+p|\} \leq |m+n-p| + 2 \min \{ |m-p|,|n-p| \} \leq 2^{\min \{j_1,j_2\} + 2}$, since $j_3 \leq j_1,j_2$.
In particular, we know by the triangle inequality that
$|m-n-p| \leq |m+n-p| +2|m-p|$ and $|m-n+p| \leq |m+n-p| + 2|n-p|$.
This means that, aside from lying between two spheres, $m-n$ must also lie in some ball of radius comparable to $2^{\min \{j_1,j_2\}}$. Thus, the $2^{j_3+1}$-neighborhood of $E_{\tau,p}(j)$ will have Lebesgue measure in $\mathbb{R}^3 \times \mathbb{R}^3$ at most a uniform constant times $2^{2 \min \{j_1,j_2\} + 4 j_3}$.
We obtain this bound by first fixing $m+n$ which by assumption can lie in a set of diameter at most a uniform constant times $2^{j_3}$ (even after taking into account the $2^{j_3+1}$ neighborhood). Once we fix such an $m+n$, we then note that $m-n$ will lie in a set with thickness at most comparable to $2^{j_3}$ transverse to the spherical shells and diameter at most comparable to $2^{\min \{j_1,j_2\}}$ in the remaining two directions.

\vspace{5mm}

\textbf{Case 3(b):} There is some pair $(m,n) \in E_{\tau,p}(j)$ for which $|n-m| < 2^{\min \{j_1,j_2\} + 4}$.
Since \eqref{dichot} still holds, we must conclude that $2^{\min\{j_1,j_2\} + 6} \geq |m+n+p|$ in this case. Pairs $(m,n) \in E_{\tau,p}$ still have $|m+n-p| \leq 2^{j_3}$. By using these two bounds, \eqref{relation}, and Cauchy-Schwarz, the best that can be said for the annular region in $n-m$ is that
\[ \left| 2 \tau + |p|^2 + |n-m|^2 \right| \leq 2^{j_3 + \min \{j_1,j_2\} + 6}. \]
Evaluating at the specific pair $(m,n) \in E_{\tau,p}(j)$ for which $|n-m| < 2^{\min \{j_1,j_2\} + 4}$, we conclude
\[ |2 \tau + |p|^2| \leq 2^{2 \min \{j_1,j_2\} + 9}. \]
We note the calculus fact that, for fixed $A>0$ and $\delta>0$, the volume of the $\delta$-neighborhood of the annular region
\begin{equation}
\label{xAlambdaregion}
|\lambda + |x|^2| \leq A
\end{equation}
is a decreasing function of $\lambda$.
In order to prove this claim, we rewrite the annular region as:
\begin{align*}
-A-\lambda \leq |x|^2 \leq A-\lambda.
\end{align*}
We now argue by cases.
\\
\\
\textit{Case 1:} $-A-\lambda \leq 0$.
\\
The assumption can be rewritten as $\lambda \in [-A,+\infty)$.
In this case, the condition for the original set is equivalent to $|x|^2 \leq A-\lambda$, so the $\delta-$neigborhood is given by $|x| \leq \sqrt{A-\lambda}+\delta$, and it has volume $C (\sqrt{A-\lambda}+\delta)^3$, which is a decreasing function of $\lambda$.
\\
\\
\textit{Case 2:} $-A-\lambda>0$
\\
In this case, the condition for the original set is equivalent to $\sqrt{-A-\lambda} \leq |x| \leq \sqrt{A-\lambda}$, so the $\delta-$neighborhood is given by
\begin{align*}
\sqrt{-A-\lambda}-\delta_1 \leq |x| \leq \sqrt{A-\lambda}+\delta.
\end{align*}
where $\delta_1:=\min\{\delta,\sqrt{-A-\lambda}\}$.
We need to consider the two possibilities separately.
\\
\\
\textit{Case 2a:} $\delta_1=\sqrt{-A-\lambda}$.
\\
The assumption is equivalent to $\lambda \in [-A-\delta^2,-A)$.
The original set is then the same as in Case 1, so the volume is a decreasing function of
$\lambda$.
\\
\\
\textit{Case 2b:} $\delta_1=\delta$.
\\
Now, we note that $\lambda \in (-\infty, -A-\delta^2]$.
The $\delta-$neighborhood is then given by:
\begin{align*}
0<\sqrt{-A-\lambda}-\delta \leq |x| \leq \sqrt{A-\lambda}+\delta.
\end{align*}
whose volume equals:
\begin{align*}
C \cdot \left[ (\sqrt{A-\lambda}+\delta)^3-(\sqrt{-A-\lambda}-\delta)^3 \right].
\end{align*}
We want to show that the function:
$$f(\lambda):=(\sqrt{A-\lambda}+\delta)^3-(\sqrt{-A-\lambda}-\delta)^3$$
is decreasing on $(-\infty,-A-\delta^2]$.

We compute:
$$f'(\lambda)=3(\sqrt{A-\lambda}+\delta)^2\cdot \frac{-1}{2\sqrt{A-\lambda}}+3(\sqrt{-A-\lambda}-\delta)^2 \cdot \frac{1}{2 \sqrt{-A-\lambda}}.$$
We want to show that $f'(\lambda) \leq 0$ for $
\lambda \leq -A-\delta^2$.
This is equivalent to showing that, for such $\lambda$:
$$(\sqrt{-A-\lambda}-\delta)^2 \cdot \frac{1}{\sqrt{-A-\lambda}} \leq (\sqrt{A-\lambda}+\delta)^2 \cdot \frac{1}{\sqrt{A-\lambda}},$$
which is equivalent to:
$$(-A-\lambda-2\delta\sqrt{-A-\lambda}+\delta^2) \cdot \sqrt{A-\lambda} \leq (A-\lambda + 2\delta\sqrt{A-\lambda}+\delta^2) \cdot \sqrt{-A-\lambda}$$
Letting $r_1:=\sqrt{-A-\lambda}$ and $r_2:=\sqrt{A-\lambda}$, the above inequality can be rewritten as:
$$(r_1^2-2\delta r_1 +\delta^2) \cdot r_2 \leq (r_2^2+2\delta r_2 +\delta^2) \cdot r_1.$$
This is equivalent to:
\begin{equation}
\label{r1r2}
(r_2-r_1) \cdot (r_1r_2-\delta^2) +4\delta r_1 r_2 \geq 0.
\end{equation}
Since by assumption $\lambda \leq -A-\delta^2$, it follows that $r_1=\sqrt{-A-\lambda} \geq \delta$. We also know that $r_2=\sqrt{A-\lambda}>r_1$ since $A>0$. In particular, it follows that $r_1r_2>\delta^2$, and so \eqref{r1r2} indeed holds.

\vspace{5mm}

Hence, the volume of the $\delta-$neighborhood of the annular region \eqref{xAlambdaregion} is indeed a decreasing function of $\lambda$. Consequently, the volume of the $2^{j_3+1}$ neighborhood of $E_{\tau,p}(j)$ in $\mathbb{R}^3 \times \mathbb{R}^3$ will be bounded above by the volume of the $2^{j_3+1}$ neighborhood of the set where $|m+n-p| \leq 2^{j_3}$ and
\[ \left| |n-m|^2 - 2^{2 \min \{j_1,j_2\} + 9} \right| \leq 2^{j_3 + \min \{j_1,j_2\} + 6}. \]
We note that the above constraint implies:
\begin{align*}
2^{\min\{j_1,j_2\}+4.5}-2^{j_3+0.5} \leq |n-m| \leq 2^{\min \{j_1,j_2\}+4.5}+2^{j_3+0.5}.
\end{align*}
Hence, in order to compute the volume of the set of all possible $m-n$, we again consider the volume in $\mathbb{R}^3$ of a $2^{j_3+1}$-neighborhood of a sphere with radius comparable to $2^{\min\{j_1,j_2\}}$. Arguing as in Case 3a, we first fix $m+n$, which is allowed to vary in a set of diameter at most comparable to $2^{j_3}$ and we then consider the possible values for $m-n$. The end result is that the $2^{j_3+1}$-neighborhood of $E_{\tau,p}(j)$ still has volume at most comparable to $2^{4 j_3 + 2 \min \{j_1,j_2\}}$, which proves the claim.
\end{proof}
\subsection{Discussion when $\alpha=1$ and the proofs of Proposition \ref{Alpha1log} and Proposition \ref{Alpha1unbounded}}
\label{alpha1discussion}
We first prove Proposition \ref{Alpha1log} in order to obtain the logarithmic upper bound in the endpoint case $\alpha=1$.
\begin{proof}(of Proposition \ref{Alpha1log})
We recall the bound on the quantity $I(\tau,p)$ given by \eqref{mainest}, by which:
\begin{equation}
\notag
I(\tau,p) \lesssim \delta( \tau + |p|^2)\sum_{m} \frac{1}{ \ang{m-p}^{4} } + \ang{p}^{2} \sum_{j_1,j_2,j_3 = 0}^{\infty} 2^{- 2\alpha (j_1 + j_2 + j_3)} E_{\tau,p}(j)
\end{equation}
with notation as before.
The first term is uniformly bounded in $(\tau,p)$, whereas by using the frequency localization, and arguing as in \eqref{alphale1}, we can bound the second term by:
\begin{equation}
\notag
\langle p \rangle^2 \, \big(\sum_{j_{max} \geq \log_2 \langle p \rangle -2} 2^{-2j_{max}}\big) \big(\sum_{j=0}^{\log_2 N + 1} 1\big)^2 \lesssim (\ln N)^2.
\end{equation}
The claim now follows from the definition of $I(\tau,p)$.
\end{proof}
The proof of Proposition \ref{Alpha1unbounded} requires several preliminary steps.
Let us note that the our proof of the uniform boundedness of $I(\tau,p)$ crucially used the assumption that $\alpha>1$ in the estimate \eqref{alphale1}. We now demonstrate that this is not a feature of the method, but that the uniform boundedness indeed does not hold when $\alpha=1$. More precisely, we now prove:
\begin{lemma}
\label{alpha1unbounded1}
The expression $I(\tau,p)$ defined in $(\ref{eq:expressionI})$ is not uniformly bounded when $\alpha=1$.
\end{lemma}
\begin{proof}
Let us recall that, when $\alpha=1$:
\begin{align*}
I(\tau,p)=\sum_{m, n \in \mathbb{Z}^3} \frac{\delta\left(\tau+|p|^2-2 \langle n, m \rangle\right) \langle p \rangle^2}{\langle m-p \rangle^2 \langle n-p \rangle^2 \langle p-n-m \rangle^2}.
\end{align*}
Let us fix $\kappa \gg 1$.
We then take $p=(\kappa,0,0),\tau=-|p|^2$ and we look at the part of the sum where $n=p$. In other words, we just sum in $m$ variable. The condition that
\begin{align*}
\tau+|p|^2-2 \langle n, m \rangle=0
\end{align*}
is equivalent to
\begin{align*}
\langle n, m \rangle=0.
\end{align*}
If we write $m \in \mathbb{Z}^3$ as $m=(m_1,m_2,m_3)$, we can deduce that $m_1=0$. We therefore need to estimate:
\begin{align*}
\sum_{m_2,m_3 \in \mathbb{Z}} \frac{\kappa^2}{\big(1+\kappa^2+m_2^2+m_3^2\big) \cdot \big(1+m_2^2+m_3^2\big)}
\sim \int_{\mathbb{R}^2} \frac{\kappa^2}{\big(1+\kappa^2+|x|^2\big) \cdot \big(1+|x|^2 \big)} \,dx=
\end{align*}
\begin{equation}
\label{logbound1}
=\int_{\mathbb{R}^2} \Big(\frac{1}{1+|x|^2}-\frac{1}{1+\kappa^2+|x|^2}\Big) \,dx.
\end{equation}
Let us note that whenever $1+|x|^2 \leq \kappa^2$, it is the case that $\frac{1}{1+|x|^2} \geq \frac{2}{1+\kappa^2+|x|^2}$. Hence, the expression in \eqref{logbound1} is:
\begin{align*}
\geq \frac{1}{2} \int_{1+|x|^2 \leq \kappa^2} \frac{1}{1+|x|^2} \, dx
\end{align*}
which by passing to polar coordinates is
\begin{align*}
\sim \int_0^{\sqrt{\kappa^2-1}} \frac{r}{1+r^2} \,dr = \frac{1}{2} \ln(1+r^2)\Big|_{r=0}^{r=\sqrt{\kappa^2-1}}=\ln \kappa.
\end{align*}
In other words,
\begin{align*}
I(\tau,p) \gtrsim \ln \kappa.
\end{align*}
The proposition now follows.
\end{proof}

We now argue by duality to see how we can relate the unboundedness of $I(\tau,p)$ when $\alpha=1$ to the original spacetime estimate. Before we prove Proposition \ref{Alpha1unbounded}, let us first prove the following related result for the individual operators $B^{+}_{j,k+1}$ and $B^{-}_{j,k+1}$ given by \eqref{Bjk+} and \eqref{Bjk-}, since the analysis of these operators was crucial in the proof of Proposition \ref{Spacetime Bound}.
\begin{proposition}
\label{alpha1unbounded2}
There does not exist $C_1>0$ such that for all $(\gamma_0^{(n)}) \in \mathcal{H}^1$, $k \in \mathbb{N}$ and $j \in \{1,\ldots,k\}$ one has:
\begin{equation}
\label{alpha1unbounded}
\|S^{(k,1)}B^{+}_{j,k+1} \, \mathcal{U}^{(k+1)}(t)\gamma_0^{(k+1)}\|_{L^2([0,2\pi] \times \Lambda^k \times \Lambda^k)} \leq C_1 \|S^{(k+1,1)} \gamma_0^{(k+1)}\|_{L^2(\Lambda^{k+1} \times \Lambda^{k+1})}.
\end{equation}
The analogous statement holds when $B^{+}_{j,k+1}$ is replaced by $B^{-}_{j,k+1}$.
\end{proposition}
In other words, Proposition \ref{alpha1unbounded2} shows that the estimates for $B^{+}_{j,k+1}$ and $B^{-}_{j,k+1}$ which we needed to use when $\alpha>1$ do not hold when $\alpha=1$. We now give the proof of Proposition \ref{alpha1unbounded2}. The main idea will be to choose a sequence of initial data $(\gamma^{(k)}_0)_k$ for which the constant $C_1$ in the estimate
\eqref{alpha1unbounded} becomes unbounded.
\begin{proof}
It suffices to show that \eqref{alpha1unbounded} does not hold when $j=k=1$.
By using the summation index we used in order to deduce \eqref{thesum2}, we obtain:
$$\big(S^{(1,1)} B^{+}_{1,2} \, \mathcal{U}^{(2)}(t) \gamma_0^{(2)}\big)\,\,\widetilde{}\,\, (\tau, p\,;q)=$$
$$
=\langle p \rangle \langle q \rangle \sum_{m,n \in \mathbb{Z}^3} \delta \Big( \tau-|q|^2+|p|^2 - 2 \langle n, m \rangle \Big) (\gamma_0^{(2)})\,\,\widehat{}\,\,(p-m,p-n;q,p-n-m).$$
Let us fix $\kappa \gg 1$, and let us take $\bar{p}:=(\kappa,0,0)$ as in the proof of Proposition \ref{alpha1unbounded1}.
By the proof of Lemma \ref{alpha1unbounded1}, it follows that we can find a sequence $(c_{m,n}) \in \ell^2_{m,n}$, with all $c_{m,n} \geq 0$ such that:
\begin{equation}
\label{logkappa}
\sum_{m, n \in \mathbb{Z}^3} \frac{\delta\left(-2 \langle n, m \rangle\right) \langle \bar{p} \rangle}{\langle m-\bar{p} \rangle \langle n-\bar{p} \rangle \langle \bar{p}-n-m \rangle} c_{m,n} \gtrsim \sqrt{\ln \kappa} \|c_{m,n}\|_{\ell^2_{m,n}}.
\end{equation}
We now find $\gamma^{(2)}_0$ such that, for all $m,n \in \mathbb{Z}^3$:
$$\langle \bar{p}-m \rangle \langle \bar{p}-n \rangle \langle \bar{p}-n-m \rangle (\gamma_0^{(2)})\,\,\widehat{}\,\,(\bar{p}-m,\bar{p}-n;0,\bar{p}-n-m)=c_{m,n}$$
and such that
$$(\gamma_0^{(2)})\,\,\widehat{}\,\,(p_1,p_2;q_1,q_2)=0,$$
for all $(p_1,p_2,q_1,q_2) \in (\mathbb{Z}^3)^4$ which are not of the form
$(\bar{p}-m,\bar{p}-n;0,\bar{p}-n-m)$ for some $m,n \in \mathbb{Z}^3$.
We note that, in this way, $\gamma_0^{(2)}$ is uniquely determined, and it satisfies:
\begin{equation}
\label{S21gamma}
\|S^{(2,1)}\gamma_0^{(2)}\|_{L^2(\Lambda^2 \times \Lambda^2)}=\|c_{m,n}\|_{\ell^2_{m,n}} < \infty.
\end{equation}
Here, we note that the four-dimensional sum defining the above norm of $\gamma_0^{(2)}$ is reduced to a two-dimensional sum since we only have two degrees of freedom $m$ and $n$.
Now, we observe:
$$
\|\big(S^{(1,1)}B^{+}_{1,2} \gamma_0^{(2)}\big)\,\,\widetilde{}\,\, (\tau,p\,;q)\|_{\ell^2_{\tau,p\,;q}}=$$
$$=\Big\| \langle p \rangle \langle q \rangle \sum_{m,n \in \mathbb{Z}^3} \delta \big( (\tau-|q|^2+|p|^2) - 2 \langle n, m \rangle \big) (\gamma_0^{(2)})\,\,\widehat{}\,\,(p-m,p-n;q,p-n-m)\Big\|_{\ell^2_{\tau,p\,;q}}$$
Since $(\gamma_0^{(2)})\,\,\widehat{}\,\, \geq 0,$ it follows that the above expression is:
$$ \geq \Big\| \langle p \rangle \langle q \rangle \sum_{m,n \in \mathbb{Z}^3} \delta \big(- 2 \langle n, m \rangle \big) (\gamma_0^{(2)})\,\,\widehat{}\,\,(p-m,p-n;q,p-n-m)\Big\|_{\ell^2_{p\,;q}}.
$$
Namely, we just look at the contribution to the sum in $\tau$ where $\tau=|q|^2-|p|^2$ and we use properties of the Kronecker delta function.
Furthermore, we can only consider the contribution of the sum in $p$ and $q$ where $p=\bar{p}$ and $q=0$ to deduce that we can bound this quantity from below by:
$$ \geq \sum_{m,n \in \mathbb{Z}^3} \delta \big(- 2 \langle n, m \rangle \big) \langle \bar{p} \rangle
(\gamma_0^{(2)})\,\,\widehat{}\,\,(\bar{p}-m,\bar{p}-n;0,\bar{p}-n-m)
$$
which, by construction, equals:
$$= \sum_{m,n \in \mathbb{Z}^3} \frac{\delta \big(- 2 \langle n, m \rangle \big) \langle \bar{p} \rangle}{\langle \bar{p}-m \rangle \langle \bar{p}-n \rangle \langle \bar{p}-n-m \rangle} \, c_{m,n}
$$
By \eqref{logkappa} and \eqref{S21gamma}, we can bound this quantity from below by:
$$\gtrsim \sqrt{\ln \kappa} \, \|c_{m,n}\|_{\ell^2_{m,n}}= \sqrt{\ln \kappa} \, \|S^{(2,1)}\gamma_0^{(2)}\|_{L^2(\Lambda^2 \times \Lambda^2)}.$$
In other words, we have proven that for the $\gamma_0^{(2)}$ with $\|S^{(2,1)} \gamma_0^{(2)}\|_{L^2(\lambda^2 \times \lambda^2)} < \infty$ constructed as above, the following estimate holds:
$$\|S^{(1,1)}B^{+}_{1,2} \, \mathcal{U}^{(2)}(t)\gamma_0^{(2)}\|_{L^2([0,2\pi] \times \Lambda \times \Lambda)} \gtrsim \sqrt{\ln \kappa} \, \|S^{(2,1)} \gamma_0^{(2)}\|_{L^2(\Lambda^2 \times \Lambda^2)}.$$
We note that the implied constant is independent of $\kappa$.
Hence, \eqref{alpha1unbounded} cannot hold for a uniform constant $C_1>0$. The argument for $B^{-}_{j,k+1}$ is proved in a similar way. We omit the details.
\end{proof}
We recall that $B_{j,k+1}=B^{+}_{j,k+1}-B^{-}_{j,k+1}$. Let us note that Proposition \ref{alpha1unbounded2} does not immediately imply
Proposition \ref{Alpha1unbounded} since, in principle, there could be cancellation in the difference. This is in contrast to the proof we gave above for boundedness when $\alpha>1$ when it was sufficient to consider either $B^{+}_{j,k+1}$ or $B^{-}_{j,k+1}$ and argue by symmetry. We now show that a refinement of the proof of Proposition \ref{alpha1unbounded2} will indeed allow us to prove Proposition \ref{Alpha1unbounded}. The key point is the fact that we can choose initial data $\gamma^{(k+1)}_0$ with the property that there exists frequencies $(\tau, \vec{p}_k\,;\vec{p}'_k) \in \mathbb{Z} \times (\mathbb{Z}^3)^k \times (\mathbb{Z}^3)^k$ at which the contribution from $B^{+}_{j,k+1}$ is (logarithmically) large, but at which the contribution from $B^{-}_{j,k+1}$ consists of only a finite number of bounded terms. Hence, for the initial data $\gamma^{(k+1)}_0$, there is no significant cancellation coming from the operators $B^{+}_{j,k+1}$ and $B^{-}_{j,k+1}$ at this frequency.
\begin{proof} (of Proposition \ref{Alpha1unbounded})
As before, we will fix $j=k=1$, and we will choose an appropriate $\gamma_0^{(2)}$ for which we get logarithmic growth in the implied constant.
Let us first note that:
$$\big(S^{(1,1)} B_{1,2} \, \mathcal{U}^{(2)}(t) \gamma_0^{(2)}\big)\,\,\widetilde{}\,\, (\tau, p\,;q)=$$
$$
=\langle p \rangle \langle q \rangle \sum_{m,n \in \mathbb{Z}^3} \Big[\delta \Big( (\tau-|q|^2+|p|^2) - 2 \langle n, m \rangle \Big) (\gamma_0^{(2)})\,\,\widehat{}\,\,(p-m,p-n;q,p-n-m)$$
\begin{equation}
\label{eq:bigsum}
-\delta((\tau+|p|^2-|q|^2)-2\langle n, m \rangle)
(\gamma_0^{(2)})\,\,\widehat{}\,\,(p,q-m-n;q-n,q-m) \Big]
\end{equation}
As before, we take $\bar{p}:=(\kappa,0,0)$. Now, we also take $\overline{q}:=(0,0,0)$.
By the proof of Lemma \ref{alpha1unbounded1}, it follows that there exists $c \in \ell^2_m$, with $c_m \geq 0$ such that:
\begin{equation}
\label{eq:cm}
\sum_{m \in \mathbb{Z}^3} \frac{\delta(-2 \langle \overline{p},m\rangle) \langle \bar{p} \rangle}{\langle \overline{p}-m \rangle \langle -m \rangle}\, c_m \gtrsim \sqrt{\ln \kappa}\, (\sum_{m \in \mathbb{Z}^2}c_m^2)^{\frac{1}{2}}
\end{equation}
We now take $\gamma^{(2)}_0$ such that, for all $m \in \mathbb{Z}^3$:
$$\langle \bar{p}-m \rangle \langle -m \rangle (\gamma_0^{(2)})\,\,\widehat{}\,\,(\bar{p}-m,0;0,-m)=c_m$$
and such that
$$(\gamma_0^{(2)})\,\,\widehat{}\,\,(p_1,p_2;q_1,q_2)=0,$$
for all $(p_1,p_2,q_1,q_2) \in (\mathbb{Z}^3)^4$ which are not of the form
$(\bar{p}-m,0;0,-m)$ for some $m \in \mathbb{Z}^3$.
Then, $\|S^{(2,1)} \gamma^{(2)}_0\|_{L^2(\Lambda^2 \times \Lambda^2)} <\infty$.
Let us take $\bar{\tau}:=|\bar{p}|^2-|\bar{q}|^2$ and let us estimate: 
$$\big|\big(S^{(1,1)} B_{1,2} \, \mathcal{U}^{(2)}(t) \gamma_0^{(2)}\big)\,\,\widetilde{}\,\, (\bar{\tau}, \bar{p}\,;\bar{q})\big|$$ 
from below.
We first note that the non-zero terms in the contribution to $(\ref{eq:bigsum})$ coming from $B^{+}_{1,2}$ (i.e. the first sum) occur when $n=\bar{p}$, hence, the first sum is only in $m$. Moreover, it equals:
$$\sum_{m \in \mathbb{Z}^3} \frac{\delta(-2 \langle \bar{p},m \rangle) \langle \bar{p} \rangle}{\langle \bar{p}-m \rangle \langle -m \rangle} \langle \bar{p}-m \rangle \langle -m \rangle (\gamma_0^{(2)})\,\,\widehat{}\,\,(\bar{p}-m,0;0,-m)=\sum_{m \in \mathbb{Z}^3} \frac{\delta(-2 \langle \overline{p},m\rangle) \langle \bar{p} \rangle}{\langle \overline{p}-m \rangle \langle -m \rangle}\, c_m$$
$$\gtrsim \sqrt{\ln \kappa} (\sum_{m \in \mathbb{Z}^3} c_m^2)^{\frac{1}{2}} = \sqrt{\ln \kappa} \|S^{(2,1)}\gamma^{(2)}_0\|_{L^2(\Lambda^2 \times \Lambda^2)}.$$
Here, we used $(\ref{eq:cm})$ and the construction of $\gamma^{(2)}_0$.
Furthermore, when considering: $$\big(S^{(1,1)} B_{1,2} \, \mathcal{U}^{(2)}(t) \gamma_0^{(2)}\big)\,\,\widetilde{}\,\, (\bar{\tau}, \bar{p}\,;\bar{q}),$$ we note that the only non-zero contribution in the second term in $(\ref{eq:bigsum})$ occurs when $\bar{q}-m-n=\bar{q}-n=0$ and $\bar{p}=\bar{p}-\bar{q}+m$. In particular, one only obtains the contribution when $m=n=(0,0,0)$. The corresponding term equals:
$$-\delta(\bar{\tau}+|\bar{p}|^2) \langle \bar{p} \rangle (\gamma_0^{(2)})\,\,\widehat{}\,\,(\bar{p},0;0,0).$$
The latter expression is bounded in absolute value by $\|S^{(2,1)} \gamma^{(2)}_0\|_{L^2(\Lambda \times \Lambda)}$.
From the previous discussion, we may conclude that, for $\bar{\tau}, \bar{p}, \bar{q}$ as above:
$$\Big|\big(S^{(1,1)} B_{1,2} \, \mathcal{U}^{(2)}(t) \gamma_0^{(2)}\big)\,\,\widetilde{}\,\, (\bar{\tau}, \bar{p}\,;\bar{q})\Big|
\gtrsim \sqrt{\ln \kappa} \|S^{(2,1)} \gamma^{(2)}_0\|_{L^2(\Lambda^2 \times \Lambda^2)}.$$
Consequently, we can take the $\ell^2_{\tau, p\,;q}$ norm to deduce that:
$$\|S^{(1,1)} B_{1,2} \,\mathcal{U}^{(2)}(t) \gamma^{(2)}_0\|_{L^2([0,2\pi] \times \Lambda \times \Lambda)} \gtrsim \sqrt{\ln \kappa} \|S^{(2,1)} \gamma^{(2)}_0\|_{L^2(\Lambda^2 \times \Lambda^2)}.$$
The Proposition now follows.
\end{proof}
It is possible to prove that Proposition \ref{Spacetime Bound} still holds in the case $\alpha=1$ when one is considering factorized elements of $\mathcal{H}^1$ (We recall the definition of $\mathcal{H}^{\alpha}$ from \eqref{mathcalHalpha}). This result is proved in Proposition \ref{factorized} below. Such an estimate, however, is not applicable in the proof of our main result since the property of factorization is not preserved under the Duhamel iteration given in the proof of Theorem \ref{Theorem 1} in the next section.
\section{Proof of the main result}
\label{MainResult}
We are now able to give the proof of our main result. Let us note the following fact:
\begin{proposition}
\label{Proposition 2} Suppose that $\alpha>1$ and that $\Gamma(t)=(\gamma^k(t))$ is a solution to $(\ref{eq:GPhierarchy})$ in the class $\mathcal{A}$. Then, there exists a continuous, positive function $\rho: \mathbb{R} \rightarrow \mathbb{R}$ such that, for all $\tau \in \mathbb{R}$,
$\Gamma(\tau)=\vec{0}$ implies that $\Gamma(s)=\vec{0}$ for all $s \in [\tau-\rho(\tau),\tau+\rho(\tau)]$.
\end{proposition}
We notice how Proposition \ref{Proposition 2} implies the result:
\begin{proof}(of Theorem \ref{Theorem 1} assuming Proposition \ref{Proposition 2}):
Let $\Gamma(t)=(\gamma^{(k)}(t))$ be a solution to $(\ref{eq:homogeneousGPhierarchy})$.
Let $S:=\{t \in \mathbb{R}; \Gamma(s)=\vec{0}, \mbox{for all}\, s \in [0,t]\}$.
Since $0 \in S$, $S$ is non-empty.
We want to argue that $sup\, S=+\infty$, since then $\Gamma(t)=\vec{0}$ for all $t \geq 0$.
In order to prove this, we argue by contradiction. Namely, we suppose that $t:=sup\, S<+\infty$. We find a sequence $(s_n)$ in $S$ such that $s_n \rightarrow t$. Since the function $\rho$ constructed in Proposition \ref{Proposition 2} is positive and continuous, it follows that there exists $n$ sufficiently large such that $s_n+\rho(s_n)>t$. Namely, we use the fact that $s_n + \rho(s_n) \rightarrow t + \rho(t)$ for some $\rho(t)>0$. By Proposition \ref{Proposition 2}, it follows that $\Gamma(s)=0$ for all $s \in [s_n,s_n+\rho(s_n)]$, which contradicts the choice of $t$.
By an analogous argument, it follows that $\Gamma(t)=\vec{0}$ for all $t \leq 0$.
\end{proof}
We now prove Proposition \ref{Proposition 2}
\begin{proof}
Let us define: $B^{(k+1)}:=\sum_{j=1}^{k}B_{j,k+1}(\gamma^{(k+1)})$.
Hence, we are looking at $\Gamma(t)=(\gamma^{(k)}(t))$ which solves, for all $k \in \mathbb{N}$:
\begin{equation}
\label{eq:tauGPhierarchy}
\begin{cases}
i \partial_t \gamma^{(k)} + (\Delta_{\vec{x}_k}-\Delta_{\vec{x}_k'})\gamma^{(k)}=\sum_{j=1}^{k}B_{j,k+1}(\gamma^{(k+1)})=B^{(k+1)}(\gamma^{(k+1)})\\
\gamma^{(k)}(\tau)=0.
\end{cases}
\end{equation}
We first show that $\gamma^{(1)}(t)=0$ for all $t \in [\tau,\tau+\rho(\tau)]$, for some $\rho(\tau)$ which will be determined. As in \cite{KM}, given $n \in \mathbb{N}$ and $\underline{t}_{n+1}=(t_1,\ldots,t_{n+1})$, we let:
$$J(\underline{t}_{n+1}):=\mathcal{U}^{(1)}(t_1-t_2)B^{(2)} \, \mathcal{U}^{(2)}(t_2-t_3) \cdots
\mathcal{U}^{(n)}(t_n-t_{n+1})B^{(n+1)}(\gamma^{(n+1)})(t_{n+1}),$$
and we let $M$ denote the set of maps $\mu:\{2,\ldots,n+1\} \rightarrow \{1,\ldots,n\}$ such that $\mu(2)=1$ and $\mu(j)<j$ for all $j$. Then, we can write:
$$J(\underline{t}_{n+1})=\sum_{\mu \in M} J(\underline{t}_{n+1};\mu)$$
where:
$$J(\underline{t}_{n+1};\mu):=\mathcal{U}^{(1)}(t_1-t_2)B_{1,2} \, \mathcal{U}^{(2)}(t_2-t_3)B_{\mu(3),3}\cdots \, \mathcal{U}^{(n)}(t_n-t_{n+1})B_{\mu(n+1),n+1}(\gamma^{(n+1)})(t_{n+1}).$$
By a repeated application of Duhamel's principle, one can write, for some $c \in \mathbb{C}$ with $|c|=1$:
$$\gamma^{(1)}(t_1)=c \int_{\tau}^{t_1} \int_{\tau}^{t_2} \cdots \int_{\tau}^{t_n} J(\underline{t}_{n+1}) dt_{n+1} \cdots dt_3 dt_2=$$
\begin{equation}
\label{eq:Duhamel}
= c \sum_{\mu \in M} \int_{\tau}^{t_1} \int_{\tau}^{t_2} \cdots \int_{\tau}^{t_n} J(\underline{t}_{n+1};\mu) dt_{n+1} \cdots dt_3 dt_2,
\end{equation}
whenever $t_1 \geq \tau$.
We recall that the \emph{boardgame argument} in \cite{KM} shows that every $\mu \in M$ can be encoded in terms of an $n \times n$ matrix.
This matrix can be reduced to a special, upper echelon matrix, which by the above identification corresponds to a map $\mu_s\in M$. This transformation is done by a series of acceptable moves. This gives us an equivalence of elements of $M$. If $\mu$ and $\mu_s$ are equivalent, we denote this as $\mu \sim \mu_s$. We omit the details of this construction, but we remark that, for every special upper echelon matrix $\mu_s$, there exists $D=D(\mu_s) \subseteq [\tau,t_1]^n$ such that:
$$\sum_{\mu \sim \mu_s} \int_{\tau}^{t_1} \int_{\tau}^{t_2} \cdots \int_{\tau}^{t_n} J(\underline{t}_{n+1};\mu) dt_{n+1} \cdots dt_3 dt_2 = \int_{D} J(\underline{t}_{n+1};\mu_s)dt_{n+1} \cdots dt_3 dt_2.$$
We also remark that the construction is written out in the case $\tau=0$, but the same argument works for general $\tau$. Finally, it is shown that the number of $n \times n$ special upper echelon matrices is at most $4^n$. Strictly speaking, all of these arguments are given on the spatial domain of $\mathbb{R}^3$, but they carry over to the case of $\mathbb{T}^3$ without any modifications.
From the preceding, it follows that we can write
\begin{equation}
\label{eq:gamma1}
\gamma^{(1)}(t_1) = \sum_{\mu_S} \int_{D}J(\underline{t}_{n+1};\mu_s)dt_{n+1} \cdots dt_2
\end{equation}
Here $\mu_S$ ranges over the set of all $n \times n$ special upper echelon matrices, and hence there are at most $4^n$ summands. Let us fix one such $\mu_s$.
One then obtains:
$$\left\|S^{(1,\alpha)}\int_{D}J(\underline{t}_{n+1};\mu_s)dt_{n+1} \cdots dt_2\right\|_{L^2(\Lambda \times \Lambda)}=$$
$$=\left\|S^{(1,\alpha)}\int_{D}\mathcal{U}^{(1)}(t_1-t_2)B_{1,2} \, \mathcal{U}^{(2)}(t_2-t_3)
B_{\mu_s(3),3} \cdots B_{\mu_s(n+1),n+1}(\gamma^{(n+1)})(t_{n+1}) dt_{n+1} \cdots dt_2\right\|_{L^2(\Lambda \times \Lambda)}.$$
For fixed $t_2 \in [\tau,t_1]$, we let $D_{t_2}:=\{(t_3,\ldots,t_{n+1}) \in [\tau,t_1]^{n-1};(t_2,t_3,\ldots,t_{n+1}) \in D\}$.
Hence, since $S^{(1,\alpha)}$ and $\mathcal{U}^{(1)}(t_1-t_2)$ commute, it follows that the previous expression equals:
$$\left\|\int_{\tau}^{t_1} \mathcal{U}^{(1)}(t_1-t_2) \big(\int_{D_{t_2}} S^{(1,\alpha)}B_{1,2} \, \mathcal{U}^{(2)}(t_2-t_3)B_{\mu_s(3),3} \cdots B_{\mu_s(n+1),n+1}(\gamma^{(n+1)})(t_{n+1})dt_{n+1} \cdots dt_3 \big)dt_2\right\|_{L^2(\Lambda \times \Lambda)}$$
By Minkowski's inequality, this expression is:
$$\leq \int_{\tau}^{t_1} \left\|\mathcal{U}^{(1)}(t_1-t_2) \big(\int_{D_{t_2}} S^{(1,\alpha)}B_{1,2} \, \mathcal{U}^{(2)}(t_2-t_3)B_{\mu_s(3),3} \cdots B_{\mu_s(n+1),n+1}(\gamma^{(n+1)})(t_{n+1})dt_{n+1} \cdots dt_3 \big)\right\|_{L^2(\Lambda \times \Lambda)}dt_2$$
which by unitarity of $\mathcal{U}^{(1)}(t_2-t_3)$ on $L^2(\Lambda \times \Lambda)$ equals:
$$\int_{\tau}^{t_1} \left\|\int_{D_{t_2}} S^{(1,\alpha)}B_{1,2} \, \mathcal{U}^{(2)}(t_2-t_3)B_{\mu_s(3),3} \cdots B_{\mu_s(n+1),n+1}(\gamma^{(n+1)})(t_{n+1})dt_{n+1} \cdots dt_3\right\|_{L^2(\Lambda \times \Lambda)}dt_2$$
Applying Minkowski's inequality again, this expression is:
$$\leq \int_{\tau}^{t_1} \big(\int_{D_{t_2}} \|S^{(1,\alpha)}B_{1,2} \, \mathcal{U}^{(2)}(t_2-t_3)B_{\mu_s(3),3} \cdots B_{\mu_s(n+1),n+1}(\gamma^{(n+1)})(t_{n+1})\|_{L^2(\Lambda \times \Lambda)}dt_{n+1} \cdots dt_3 \big)dt_2$$
which by the fact that $D_{t_2} \subseteq [\tau,t_1]^{n-1}$ and by Fubini's theorem is:
$$\leq \int_{[\tau,t_1]^n} \|S^{(1,\alpha)}B_{1,2} \, \mathcal{U}^{(2)}(t_2-t_3)B_{\mu_s(3),3} \cdots B_{\mu_s(n+1),n+1}(\gamma^{(n+1)})(t_{n+1})\|_{L^2(\Lambda \times \Lambda)}dt_2 dt_3 \cdots dt_{n+1}$$
By the Cauchy-Schwarz inequality in $t_2$, this is:
$$\leq (t_1-\tau)^{\frac{1}{2}} \int_{[\tau,t_1]^{n-1}} \|S^{(1,\alpha)}B_{1,2} \, \mathcal{U}^{(2)}(t_2-t_3)B_{\mu_s(3),3} \cdots B_{\mu_s(n+1),n+1}(\gamma^{(n+1)})(t_{n+1})\|_{L^2(\mathbb[\tau,t_1] \times\Lambda \times \Lambda)}dt_3 \cdots dt_{n+1}$$
Since $\mathcal{U}^{(2)}(t_2-t_3)=\mathcal{U}^{(2)}(t_2)\mathcal{U}^{(2)}(-t_3)$, we can use Proposition \ref{Spacetime Bound} to deduce that this is:
$$\leq C_1(t_1-\tau)^{\frac{1}{2}} \int_{[\tau,t_1]^{n-1}} \|S^{(2,\alpha)}\mathcal{U}^{(2)}(-t_3)B_{\mu_s(3),3} \cdots B_{\mu_s(n+1),n+1}(\gamma^{(n+1)})(t_{n+1})\|_{L^2(\Lambda^2 \times \Lambda^2)}dt_3 \cdots dt_{n+1}$$
Since $S^{(2,\alpha)}$ and $\mathcal{U}^{(2)}(-t_3)$ commute, and since $\mathcal{U}^{(2)}(-t_3)$ is unitary on $L^2(\Lambda^2 \times \Lambda^2)$, this expression is:
$$=C_1(t_1-\tau)^{\frac{1}{2}} \int_{[\tau,t_1]^{n-1}} \|S^{(2,\alpha)}B_{\mu_s(3),3} \, \mathcal{U}^{(3)}(t_3-t_4)\cdots B_{\mu_s(n+1),n+1}(\gamma^{(n+1)})(t_{n+1})\|_{L^2(\Lambda^2 \times \Lambda^2)}dt_3 \cdots dt_{n+1}$$
We iterate this procedure $n-2$ more times to deduce that this is:
\begin{equation}
\label{eq:Duhamelterm1}
\leq (C_1 (t_1-\tau)^{\frac{1}{2}})^{n-1} \int_{\tau}^{t_1} \|S^{(n,\alpha)}B_{\mu_s(n+1),n+1}(\gamma^{(n+1)})(t_{n+1})\|_{L^2(\Lambda^n \times \Lambda^n)}dt_{n+1}.
\end{equation}
We now define the function $\rho=\rho(t)$ as:
\begin{equation}
\label{eq:rhodefinition}
\rho(t):=min\, \Big\{\sigma(t),\frac{1}{64 C_1^2 f^2(t)} \Big\}.
\end{equation}
where we recall that $C_1$ is the constant from Proposition \ref{Spacetime Bound} and $\sigma, f$ are as in $(\ref{eq:aprioribound})$.
Since $\sigma,f$ are continuous and positive functions, so is $\rho$.
Moreover, we choose $t_1$ such that:
\begin{equation}
\label{eq:t1choice}
t_1 \in [\tau,\tau+\rho(\tau)].
\end{equation}
Consequently, by the construction of $\rho$, and by the assumption $(\ref{eq:aprioribound})$, it follows that the expression in $(\ref{eq:Duhamelterm1})$ is:
\begin{equation}
\label{eq:Duhamelterm2}
\leq (C_1 (\rho(\tau))^{\frac{1}{2}})^{n-1} f^{n+1}(\tau) =
[C_1 f(\tau) (\rho(\tau))^{\frac{1}{2}}]^{n-1} f^2(\tau).
\end{equation}
We use $(\ref{eq:gamma1})$, $(\ref{eq:Duhamelterm2})$ and the triangle inequality to deduce that:
\begin{equation}
\label{eq:Duhamelterm3}
\|S^{(1,\alpha)}\gamma^{(1)}(t_1)\|_{L^2(\Lambda \times \Lambda)} \leq 4^n [C_1 f(\tau) (\rho(\tau))^{\frac{1}{2}}]^{n-1} f^2(\tau)=
\end{equation}
$$4[4 C_1 f(\tau) (\rho(\tau))^{\frac{1}{2}}]^{n-1} f^2(\tau).$$
By $(\ref{eq:rhodefinition})$, it follows that $4 C_1 f(\tau) (\rho(t))^{\frac{1}{2}} \leq \frac{1}{2}$,
hence by $(\ref{eq:Duhamelterm3})$, one has:
$$\|S^{(1,\alpha)}\gamma^{(1)}(t_1)\|_{L^2(\Lambda \times \Lambda)} \leq 4 \cdot \Big(\frac{1}{2}\Big)^{n-1} \cdot f^2(\tau).$$
We let $n \rightarrow \infty$ to deduce that $\gamma^{(1)}(t_1)=0$ for all $t_1 \in [\tau,\tau+\rho(\tau)]$.
We observe that an analogous argument shows that for all $k \in \mathbb{N}$, one has $\gamma^{(k)}(t_1)=0$ for all $t_1 \in [\tau,\tau+\rho(\tau)]$. The same argument also shows that for all $k \in \mathbb{N}$, one has $\gamma^{(k)}(t_1)=0$ for all $t_1 \in [\tau-\rho(\tau),\tau]$.
\end{proof}
\section{Properties of the factorized solution and the proof of Theorem \ref{Theorem 2}}
\label{Factorized}
In this section, we study the factorized solutions to $(\ref{eq:GPhierarchy})$ in more detail and we present the proof of Theorem \ref{Theorem 2}.
Let $I=[t_0-T,t_0+T]$. Let us consider $\phi$ such that:
\begin{equation}
\label{eq:NLS1}
\begin{cases}
i\partial_t \phi + \Delta \phi = |\phi|^2 \phi,\,\mbox{on}\,\Lambda \times I\\
\phi |_{t=t_0}=\Phi.
\end{cases}
\end{equation}
We are also assuming that the length $|I|$ of the interval $I$ satisfies:
\begin{equation}
\label{eq:Ileq1}
|I| \leq 1.
\end{equation}
The factorized solution $\Gamma(t)=(\gamma^{(k)}(t))$ is given by: $\gamma^{(k)}:=|\phi \rangle \langle \phi|^{\otimes k}(t)$.
Hence:
\begin{equation}
\label{eq:factorized}
\gamma^{(k)}(t,\vec{x_k};\vec{x'_k})=\prod_{j=1}^{k} \phi(t,x_j) \overline{\phi(t,x_j')}.
\end{equation}
We remark that the proof of the analogue of Theorem \ref{Theorem 2} in the non-periodic setting in the work of Klainerman and Machedon \cite{KM} is based on the use of Strichartz estimates on $\mathbb{R}^3$. Since we are working on $\mathbb{T}^3$, we cannot use all of the Strichartz estimates
which one needs to use on $\mathbb{R}^3$. One can apply the periodic Strichartz estimates as in \cite{B}, but in doing so, one ends up losing a certain number of derivatives which, in turn, require additional regularity on the initial data. Instead of taking this approach, we use variants of the function spaces $U$ and $V$, defined in Section \ref{NotationHarmonicAnalysis}. As was noted above, these spaces are typically used in critical problems, but we will see that they will be useful in the sub-critical setting as well. Namely, we recall that the cubic NLS on $\mathbb{T}^3$ is energy sub-critical. The advantage in using these function spaces is that they do not require us to assume additional regularity on the initial data. In fact, we only need to assume that the initial data lies in $H^1(\Lambda)$, the natural space associated to the energy.
On the other hand, the function spaces that we use will have the feature that one has to look at the time of local existence which is not immediately related to a conserved quantity of the equation.

In the proof of Theorem \ref{Theorem 2}, we will need to use the following local-in-time bound:
\begin{proposition}
\label{Proposition 4}
\begin{enumerate}
\item Given $\alpha \geq 1$ and $\Phi \in H^{\alpha}(\Lambda)$ and $t_0 \in \mathbb{R}$, there exists $T=T(t_0)>0$ and a unique function $\psi \in C_u([t_0-T(t_0),t_0+T(t_0)];H^{\alpha}(\Lambda)) \cap X^{\alpha}([t_0-T(t_0),t_0+T(t_0)])$ which solves:
\begin{equation}
\label{eq:PHIInitialValue}
\begin{cases}
i\partial_t \psi + \Delta \psi = |\psi|^2\psi,\,\mbox{on}\,\, \Lambda \times [t_0-T(t_0),t_0+T(t_0)]\\
\psi|_{t=t_0}=\Phi.
\end{cases}
\end{equation}
\item Furthermore, the function $T: \mathbb{R} \rightarrow \mathbb{R}$ can be chosen to be continuous.
\item Finally, there exists a constant $C_3>0$, independent of $t_0$, such that:
\begin{equation}
\label{eq:lwp}
\|\psi\|_{X^{\alpha}([t_0-T(t_0),t_0+T(t_0)])} \leq C_3 \|\Phi\|_{H^{\alpha}_x}
\end{equation}
and
\begin{equation}
\label{eq:lwpa}
\|\psi\|_{L^{\infty}([t_0-T(t_0),t_0+T(t_0)];H^{\alpha}_x)} \leq C_3 \|\Phi\|_{H^{\alpha}_x}.
\end{equation}
\end{enumerate}
\end{proposition}
Except for the step concerning the continuity of $T$, the proof of Proposition \ref{Proposition 4} is very similar to the proof of Theorem 1.1 in \cite{HTT}, with some minor modifications since we are considering the cubic nonlinearity instead of the quintic one. For completeness, we present the proof of Proposition \ref{Proposition 4} in Appendix A.
We note that, by uniqueness, on the appropriate time interval, one has that $\psi=\phi$, where $\phi$ is the solution to $(\ref{eq:PHIInitialValue})$ given in the work of Bourgain \cite{B}.
To be more precise, let us without loss of generality consider the case when $t_0=0$ and we look at non-negative times. We recall that, in \cite{B}, it is obtained that $F(u)=e^{it\Delta}\Phi-i\int_0^t e^{i(t-\tau)\Delta} |u|^2u(\tau) d\tau$ is a contraction in a space $B^{\alpha}([0,T])$, with norm given by:
\begin{equation}
\label{eq:triplenorm}
|||u|||:=\sup_{K,N} (K+1)^{\frac{1}{2}}(N+1)^{\alpha} \big(\int_{|\xi| \sim N, |\tau-|\xi|^2| \sim K} |\widehat{u}(\xi,\tau)|^2 d\xi d\tau \big)^{\frac{1}{2}}.
\end{equation}
In the corresponding class, it is shown that the local solutions are unique (we remark that the standard $X^{s,b}$ class is insufficient for uniqueness of solutions, see remark $(ii)$ at the end of section 5 of \cite{B}). The point is that we can now consider $F$ to be a contraction on: $$B^{\alpha}([0,T]) \cap C_u([0,T];H^{\alpha}(\Lambda)) \cap X^{\alpha}([0,T]).$$
(by possibly making $T$ smaller; later we put all the small time intervals together).
By the uniqueness part of Proposition \ref{Proposition 4} and by the uniqueness part of the result in \cite{B}, the fact that $\phi=\psi$ now follows. We can hence view Proposition \ref{Proposition 4} as a result which tells us additional local properties of the solutions to the NLS.
The main point is that we can now define:
\begin{equation}
\label{eq:sigmadefinition}
\sigma(t):=\min \{T(t),1\}.
\end{equation}
By construction, $\sigma$ is a positive continuous function. We can rewrite the result of the proposition as follows:
\begin{equation}
\label{eq:lwpbound}
\|\phi\|_{X^{\alpha}([t-\sigma(t),t+\sigma(t)])} \leq C_3 \|\phi(t)\|_{H^{\alpha}_x(\Lambda)}.
\end{equation}
Having summarized these facts, we can prove
Theorem \ref{Theorem 2}.
\begin{proof}(of Theorem \ref{Theorem 2}) Let $t_0 \in \mathbb{R}$ be given.
We write $I=[t_0-\sigma(t_0),t_0+\sigma(t_0)]$. We note that $I$ has length at most 2 by $(\ref{eq:sigmadefinition})$.
Let us without loss of generality consider the case when $j=1$ and when we are considering the expression $B_{1,k+1}^{+}S^{(k+1,\alpha)}\gamma^{(k+1)}$. The estimate for the other indices $j$, as well as for the terms coming from $B_{j,k+1}^{-}S^{(k+1,\alpha)}\gamma^{(k+1)}$ follow in a similar manner.
We note that by $(\ref{eq:factorized})$,
$$S^{(k,\alpha)}B_{1,k+1}^{+}\gamma^{(k+1)}(\vec{x_k};\vec{x'_k})=
[D^{\alpha}_x(|\phi|^2 \phi)](x_1) \prod_{j=2}^{k+1} (D^{\alpha}_x \phi)(x_j)
\prod_{j=1}^{k+1} (D^{\alpha}_x \bar{\phi})(x_j').$$
Consequently:
$$\int_{I} \|S^{(k,\alpha)}B_{1,k+1}^{+}\gamma^{(k+1)}(t)\|_{L^2(\Lambda^{k} \times \Lambda^{k})} dt= \|D^{\alpha}_x(|\phi|^2 \phi)\|_{L^1_t L^2_x(I \times \Lambda)} \|D^{\alpha}_x \phi\|_{L^{\infty}_t L^2_x(I \times \Lambda)}^{2k+1}$$
By $(\ref{eq:Ileq1})$ and by applying the Cauchy-Schwarz inequality in $t$, and the fact that the length of the interval $I$ is of length $O(1)$, this expression is bounded by:
\begin{equation}
\label{eq:Skalpha}
\lesssim \|D^{\alpha}_x(|\phi|^2 \phi)\|_{L^2_t L^2_x(I \times \Lambda)}
\|\phi\|_{L^{\infty}(I;H^{\alpha}_x)}^{2k+1}.
\end{equation}
By $(\ref{eq:lwpa})$ from Proposition \ref{Proposition 4}, we obtain:
\begin{equation}
\label{eq:lwp1}
\|\phi\|_{L^{\infty}(I;H^{\alpha}_x)}^{2k+1} \leq (C_3 \|\phi(t_0)\|_{H^{\alpha}_x})^{2k+1}.
\end{equation}

We need to estimate $\|D^{\alpha}_x(|\phi|^2\phi)\|_{L^2(I \times \Lambda)}.$
More generally, let us estimate the expression: $$\|D^{\alpha}_x(\phi_1 \phi_2 \phi_3)\|_{L^2(I \times \Lambda)}.$$
We dyadically decompose the factors, i.e. for $N_1 \geq N_2 \geq N_3 \geq 1$, we consider:
\begin{equation}
\label{eq:dyadicpieces}
\|D^{\alpha}_x(P_{N_1}\phi_1 P_{N_2}\phi_2 P_{N_3}\phi_3)\|_{L^2(I \times \Lambda)}.
\end{equation}
One has to consider several cases. Our argument is similar to the proof of Proposition 4.1 in \cite{HTT}.
\\
\\
\textbf{Case 1: $N_1 \gg N_2$:}
\\
\\
In this case, we estimate the expression $(\ref{eq:dyadicpieces})$ by duality. We fix $u \in L^2(I \times \Lambda)$ such that $\|u\|_{L^2(I \times \Lambda)}=1$.
We know that the Fourier transform in space of $P_{N_1}\phi_1 P_{N_2}\phi_2 P_{N_3}\phi_3$ has support in $|\xi| \sim N_1$. Hence:
$$\int_{I} \int_{\Lambda} u \, D^{\alpha}_x (P_{N_1}\phi_1 P_{N_2}\phi_2 P_{N_3} \phi_3)\, dx dt=$$
$$\int_{I} \int_{\Lambda} P_{N_0}u \, D^{\alpha}_x (P_{N_1}\phi_1 P_{N_2}\phi_2 P_{N_3} \phi_3)\, dx dt=$$
for some $N_0 \sim N_1$.
It follows that the contribution to $J:=|\int_{I} \int_{\Lambda} u D^{\alpha}_x (|\phi|^2\phi) dx dt|$ coming from this case is:
$$\lesssim \sum_{N_0,\ldots,N_3; N_0 \sim N_1, N_1 \gg N_2 \geq N_3}
|\int_{I} \int_{\Lambda} P_{N_0}u \, D^{\alpha}_x (P_{N_1}\phi_1 P_{N_2}\phi_2 P_{N_3} \phi_3)\, dx dt| $$
Which by applying the Cauchy-Schwarz inequality on $I \times \Lambda$ is:
$$\leq \sum_{N_0,\ldots,N_3; N_0 \sim N_1, N_1 \gg N_2 \geq N_3}
\|P_{N_0}u\|_{L^2(I \times \Lambda)} \|D^{\alpha}_x (P_{N_1}\phi_1 P_{N_2}\phi_2 P_{N_3} \phi_3)\|_{L^2(I \times \Lambda)}$$
By dyadic localization, we can bound this by:
$$\lesssim \sum_{N_0,\ldots,N_3; N_0 \sim N_1, N_1 \gg N_2 \geq N_3}
\|P_{N_0}u\|_{L^2(I \times \Lambda)} N_1^{\alpha} \|(P_{N_1}\phi_1 P_{N_2}\phi_2 P_{N_3} \phi_3)\|_{L^2(I \times \Lambda)}$$
By using $(\ref{eq:Star})$ and the dyadic localization, this expression is:
$$\lesssim \sum_{N_0,\ldots,N_3; N_0 \sim N_1, N_1 \gg N_2 \geq N_3}
\|P_{N_0}u\|_{L^2(I \times \Lambda)} N_1^{\alpha} N_2 N_3\, \max \, \{\frac{N_3}{N_1},\frac{1}{N_2}\}^{\delta} \prod_{j=1}^{3} \|P_{N_j}\phi_j\|_{Y^0(I)}$$
$$\lesssim \sum_{N_0,\ldots,N_3; N_0 \sim N_1, N_1 \gg N_2 \geq N_3}
\|P_{N_0}u\|_{L^2(I \times \Lambda)}\, \max \, \{\frac{N_3}{N_1},\frac{1}{N_2}\}^{\delta} \|P_{N_1}\phi_1\|_{Y^{\alpha}(I)} \|P_{N_2}\phi_2\|_{Y^1(I)} \|P_{N_3}\phi_3\|_{Y^1(I)}$$
which since $N_2 \geq N_3$ is:
$$\lesssim \sum_{N_0,\ldots,N_3; N_0 \sim N_1, N_1 \gg N_2 \geq N_3}
\|P_{N_0}u\|_{L^2(I \times \Lambda)}\, \max \, \{\big(\frac{N_3}{N_1}\big)^{\frac{1}{2}} \big(\frac{N_2}{N_1}\big)^{\frac{1}{2}},\frac{1}{N_2^{\frac{1}{2}}}\frac{1}{N_3^{\frac{1}{2}}}\}^{\delta} \cdot$$
$$\|P_{N_1}\phi_1\|_{Y^{\alpha}(I)} \|P_{N_2}\phi_2\|_{Y^1(I)} \|P_{N_3}\phi_3\|_{Y^1(I)}$$
$$\lesssim \sum_{N_0,\ldots,N_3; N_0 \sim N_1, N_1 \gg N_2 \geq N_3}
\Big(\big(\frac{N_3}{N_1}\big)^{\frac{\delta}{2}} \big(\frac{N_2}{N_1}\big)^{\frac{\delta}{2}}+
\frac{1}{N_2^{\frac{\delta}{2}}} \frac{1}{N_3^{\frac{\delta}{2}}}\Big) \cdot$$
$$\|P_{N_0}u\|_{L^2(I \times \Lambda)}\, \|P_{N_1}\phi_1\|_{Y^{\alpha}(I)}
\|P_{N_2}\phi_2\|_{Y^1(I)} \|P_{N_3}\phi_3\|_{Y^1(I)}$$
We apply the Cauchy-Schwarz inequality in $N_2$ and $N_3$ to deduce that this expression is:
$$\lesssim \sum_{N_0,N_1; N_0 \sim N_1} \|P_{N_0}u\|_{L^2(I \times \Lambda)}
\|P_{N_1}\phi_1\|_{Y^{\alpha}(I)} \|\phi_2\|_{Y^1(I)} \|\phi_3\|_{Y^1(I)}.$$
Finally, we apply the Cauchy-Schwarz inequality in $N_0 \sim N_1$ to deduce that this sum is:
$$\lesssim \|u\|_{L^2(I \times \Lambda)} \|\phi_1\|_{Y^{\alpha}(I)} \|\phi_2\|_{Y^1(I)} \|\phi_3\|_{Y^1(I)}.$$
Consequently, the contribution to $J$ from Case 1 is:
\begin{equation}
\label{eq:Case1bound}
\lesssim \|\phi_1\|_{Y^{\alpha}(I)} \|\phi_2\|_{Y^1(I)} \|\phi_3\|_{Y^1(I)}.
\end{equation}
\\
\\
\textbf{Case 2: $N_1 \sim N_2$:}
\\
\\
In this case, we again use dyadic localization and $(\ref{eq:Star})$ to deduce that:
$$\sum_{N_1,N_2,N_3;N_1 \sim N_2 \geq N_3} \|D^{\alpha}_x (P_{N_1}\phi_1 P_{N_2}\phi_2 P_{N_3}\phi_3)\|_{L^2(I \times \Lambda)}$$
$$\lesssim \sum_{N_1,N_2,N_3;N_1 \sim N_2 \geq N_3} N_1^{\alpha} N_2 N_3
\,\max\{\frac{N_3}{N_1},\frac{1}{N_2}\}^{\delta} \prod_{j=1}^{3} \|P_{N_j}\phi_j\|_{Y^0(I)}$$
$$\lesssim \sum_{N_1,N_2,N_3;N_1 \sim N_2 \geq N_3} \,\max\{\frac{N_3}{N_1},\frac{1}{N_2}\}^{\delta} \|P_{N_1}\phi_1\|_{Y^{\alpha}(I)} \|P_{N_2}\phi_2\|_{Y^1(I)} \|P_{N_3}\phi_3\|_{Y^1(I)}$$
Since $N_1 \sim N_2$, we note that $\frac{1}{N_2} \lesssim \frac{N_3}{N_1}$, hence the above sum is:
$$\lesssim \sum_{N_1,N_2,N_3;N_1 \sim N_2 \geq N_3} 
\big(\frac{N_3}{N_1}\big)^{\delta} \|P_{N_1}\phi_1\|_{Y^{\alpha}(I)} \|P_{N_2}\phi_2\|_{Y^1(I)} \|P_{N_3}\phi_3\|_{Y^1(I)}$$
By the Cauchy-Schwarz inequality in $N_3$, this expression is:
$$\lesssim \sum_{N_1,N_2;N_1 \sim N_2} \|P_{N_1}\phi_1\|_{Y^{\alpha}(I)} \|P_{N_2}\phi_2\|_{Y^1(I)} \|\phi_3\|_{Y^1(I)}$$
Furthermore, by the Cauchy-Schwarz inequality in $N_1 \sim N_2$, this is:
\begin{equation}
\label{eq:Case2bound}
\lesssim \|\phi_1\|_{Y^{\alpha}(I)}\|\phi_2\|_{Y^1(I)}\|\phi_3\|_{Y^1(I)}.
\end{equation}
\begin{remark}
\label{Complex_Conjugates}
We note that in the above proof, we could replace any of the $\phi_j$ with $\overline{\phi_j}$ in the expression $(\ref{eq:dyadicpieces})$, and we would still obtain the upper bound given by $(\ref{eq:Case2bound})$. The reason for this is that, in $(\ref{eq:Star})$, we can replace any of the $u_j$ on the left-hand side of the estimate by $\overline{u}_j$ without changing the value of the expression.
\end{remark}

From $(\ref{eq:Case1bound})$, $(\ref{eq:Case2bound})$, and Remark \ref{Complex_Conjugates}, it follows that:

\begin{equation}
\label{eq:trilinearbound}
\|D^{\alpha}_x(|\phi|^2\phi)\|_{L^2(I \times \Lambda)} \lesssim \|\phi\|_{Y^{\alpha}(I)} \|\phi\|_{Y^1(I)}^2 \lesssim \|\phi\|_{X^{\alpha}(I)}^3.
\end{equation}
In the last step, we used $(\ref{eq:spaceinclusion})$ and the fact that $\alpha \geq 1$.
By $(\ref{eq:lwpbound})$, the above expression is:
\begin{equation}
\label{eq:lwp2}
\leq (C_3 \|\phi(t_0)\|_{H^{\alpha}_x})^3,
\end{equation}
where $C_3>0$ is the constant from Proposition \ref{Proposition 4}.
By $(\ref{eq:lwp1})$ and $(\ref{eq:lwp2})$, it follows that we can take $f(t):=C \|\phi(t)\|_{H^{\alpha}_x}^2$, with appropriate $C>0$ and for $\sigma$ as defined in $(\ref{eq:sigmadefinition})$. Then $f$ is a positive and continuous function by Proposition \ref{Proposition 4} and the spacetime bound $(\ref{eq:aprioribound})$ holds for the factorized solution $\Gamma(t)=(|\phi \rangle \langle \phi|^{\otimes k}(t))$.
\end{proof}
By $(\ref{eq:trilinearbound})$, we note the following bound:
\begin{corollary}
\label{productestimate}
With notation as above:
$$\|D^{\alpha}_x(|\phi|^2\phi)\|_{L^2(I \times \Lambda)} \lesssim \|\phi\|_{Y^{\alpha}(I)} \|\phi\|_{Y^1(I)}^2 \lesssim \|\phi\|_{X^{\alpha}(I)}^3.$$
\end{corollary}
If one considers factorized solutions, we can prove that the bound in Proposition \ref{Spacetime Bound} holds when $\alpha=1$. We note that this case is easier since we have more structure than for general $k$-particle densities. It is not obvious how to apply this result to the uniqueness statement for the Gross-Pitaevskii hierarchy since the property of factorization is not preserved under the iterated Duhamel expansion $(\ref{eq:Duhamel})$.
\begin{proposition}
\label{factorized}
Suppose that $\alpha \geq 1$ and that $\phi_0 \in H^{\alpha}(\Lambda)$. Suppose that $k \in \mathbb{N}$ and consider 
$\gamma_0^{(k+1)}:=|\phi_0 \rangle \langle \phi_0|^{\otimes (k+1)}$. Let $I$ be a time interval of length $O(1)$. Then, there exists $C>0$, depending only on $\alpha$, such that for all $j \in \{1,\ldots,k+1\}$, the following bound holds:
\begin{equation}
\label{factorizedbound}
\|S^{(k,\alpha)} B_{j,k+1} \, \mathcal{U}^{(k+1)}(t) \gamma_0^{(k+1)}\|_{L^2(I \times \Lambda^k \times \Lambda^k)} \leq C \|S^{(k+1,\alpha)} \gamma_0^{(k+1)} \|_{L^2(\Lambda^{k+1} \times \Lambda^{k+1})}.
\end{equation}
\end{proposition}
\begin{proof}
As in previous arguments, it suffices to prove the claim for $B^{+}_{j,k+1}$ instead of $B_{j,k+1}$. In addition, one can assume, without loss of generality that $j=1$. As was noted in the proof of Theorem \ref{Theorem 2}:
\begin{equation}
\label{factorized1}
\|S^{(k,\alpha)} B^{+}_{1,k+1} \, \mathcal{U}^{(k+1)}(t) \gamma_0^{(k+1)}\|_{L^2(I \times \Lambda^k \times \Lambda^k)} \leq C \|\phi_0\|_{H^{\alpha}}^{2k} \|\phi_0\|_{H^1}^2.
\end{equation}
On the other hand, we also know that:
\begin{equation}
\label{factorized2}
\|S^{(k+1,\alpha)} \gamma_0^{(k+1)}\|_{L^2(\Lambda^{k+1} \times \Lambda^{k+1})}=\|\phi_0\|_{H^{\alpha}}^{2k+2}.
\end{equation}
Since $\alpha \geq 1$, \eqref{factorized1} and \eqref{factorized2} imply the claim.
\end{proof}
\section{Appendix A: Proof of Proposition \ref{Proposition 4}}
\label{AppendixA}
In this Appendix, we prove the local-in-time result given in Proposition \ref{Proposition 4}. For simplicity of notation, let us assume without loss of generality that $t_0=0$. We just need to keep in mind that the implied constants depend on conserved quantities for the NLS (i.e. they are functions of the $H^1$ norm of the solution).

Let us write the Duhamel term as:
\begin{equation}
\label{eq:DuhameltermI}
\mathcal{D}(f)(t):=\int_0^t e^{i(t-\tau)\Delta}f(\tau) d\tau.
\end{equation}
A key tool in our proof will be the following:
\begin{lemma}
\label{trilinearlemma}
For $\alpha \geq 1$, and for $I \subseteq \mathbb{R}$, a time interval of length $O(1)$, one has:
\begin{equation}
\label{eq:trilinearlemmabound}
\|\mathcal{D}(\prod_{k=1}^{3} \widetilde{u}_k)\|_{X^{\alpha}(I)} \lesssim \sum_{j=1}^{3}
\|u_j\|_{X^{\alpha}(I)} \prod_{k=1;k \neq j}^{3} \|u_k\|_{X^1(I)},
\end{equation}
where $\widetilde{u}_k$ denotes either $u_k$ or $\bar{u}_k$.
The implied constant does not depend on the length of the interval.
\end{lemma}
We note that there is no power of the length of the time interval in the above estimate. The quintilinear version of Lemma \ref{trilinearlemma} was proved as Proposition 4.1 in \cite{HTT}. The proof of Lemma \ref{trilinearlemma} proceeds in a similar way, and is based on a duality argument and the use of the good decay factor of $\max \Big\{\frac{N_3}{N_1}, \frac{1}{N_2}\Big\}^{\delta}$ obtained in $(\ref{eq:Star})$. We will omit the details and refer the reader to \cite{HTT}.

We now prove Proposition \ref{Proposition 4}.
\begin{proof}(of Proposition \ref{Proposition 4})
It suffices to consider non-zero solutions $u$, since the claim holds in the case of the zero solution.
By time translation, let us consider the case $t_0=0$.
\\
\\
\textbf{1: Small data:}
\\
\\
Since $\alpha \geq 1$, it follows from Lemma \ref{trilinearlemma} that:
$$\|\mathcal{D}(|u|^2u-|v|^2v)\|_{X^{\alpha}([0,T])} \leq C (\|u\|_{X^{\alpha}([0,T])}^2+\|v\|_{X^{\alpha}([0,T])}^2)\|u-v\|_{X^{\alpha}([0,T])},$$
whenever $u,v \in X^{\alpha}([0,T])$.
Given $\epsilon,\delta>0$, we consider:
$$B_{\epsilon}:=\{f \in H^{\alpha}(\Lambda); \|f\|_{H^{\alpha}} \leq \epsilon \}.$$
$$D_{\delta}:=\{u \in X^{\alpha}([0,T]) \cap C_u([0,T];H^{\alpha}(\Lambda)); \|u\|_{X^{\alpha}([0,T])} \leq \delta\}.$$
We note that $D_{\delta}$ is complete with respect to the metric in $X^{\alpha}([0,T])$.
Let us denote:
$$L(f):=e^{it\Delta}f,NL(u):=-i\mathcal{D}(|u|^2u).$$
For fixed $f \in B_{\epsilon}$, we want to solve:
\begin{equation}
\label{eq:fixedpoint1}
u=F(u):=L(f)+NL(u).
\end{equation}
by using the contraction mapping principle in $D_{\delta}$, with $f \in B_{\epsilon}$.
By $(\ref{eq:linearestimate})$ and Lemma \ref{trilinearlemma}, it follows that, for $u \in D_{\delta}$:
$$\|L(f)+NL(u)\|_{X^{\alpha}([0,T])} \leq \|e^{it\Delta}f\|_{X^{\alpha}([0,T])}+\|\mathcal{D}(|u|^2u)\|_{X^{\alpha}([0,T])}
\leq \|f\|_{H^{\alpha}(\Lambda)} + c\|u\|_{X^{\alpha}([0,T])}^3$$
\begin{equation}
\label{eq:F1}
\leq \epsilon + c\delta^3 \leq \delta
\end{equation}
if we take:
$$\epsilon:=\frac{\delta}{2}, \delta:=\frac{1}{(8c)^{\frac{1}{2}}}.$$
We also need to show that $F$ maps $D_{\delta}$ into $C_u([0,T];H^{\alpha}(\Lambda))$.
Let us first note that, since $f \in H^{\alpha}(\lambda)$:
\begin{equation}
\label{eq:continuity1}
L(f)=e^{it\Delta}f \in C_u([0,T];H^{\alpha}(\Lambda)).
\end{equation}
We now show that $NL(u) \in C_u([0,T];H^{\alpha}(\Lambda))$.
For $\kappa>0$, one has:
$$NL(u)(t+\kappa)-NL(u)(t)= \int_{0}^{t+\kappa} e^{i(t+\kappa-\tau)\Delta} (|u|^2u)(\tau) d\tau- \int_{0}^{t} e^{i(t-\tau)\Delta}(|u|^2u)(\tau) d\tau=$$
$$= \int_{0}^{\kappa} e^{i(t+\kappa-\tau)\Delta} (|u|^2u)(\tau) d\tau
+ \int_{0}^{t} e^{i(t-\tau)\Delta} \big((|u|^2u)(\tau+\kappa)-(|u|^2u)(\tau)\big) d\tau$$
Now, by using the Cauchy-Schwarz inequality in $t$ and the unitarity of $e^{it\Delta}$ on $H^{\alpha}(\Lambda)$, we note that for $u \in D_{\delta}$:
$$\|\int_{0}^{\kappa} e^{i(t+\kappa-\tau)\Delta} (|u|^2u)(\tau) d\tau\|_{L^{\infty}([0,T];H^{\alpha}(\Lambda))} \leq \kappa^{\frac{1}{2}}
\|D^{\alpha}(|u|^2u)\|_{L^2([0,T];L^2(\Lambda))}$$
which by $(\ref{eq:trilinearbound})$ is:
\begin{equation}
\label{eq:continuity2}
\lesssim \kappa^{\frac{1}{2}} \|u\|_{X^{\alpha}([0,T])}^3 \leq \kappa^{\frac{1}{2}} \delta^3.
\end{equation}
Furthermore,
$$\|\int_{0}^{t} e^{i(t-\tau)\Delta} \big((|u|^2u)(\tau+\kappa)-(|u|^2u)(\tau)\big) d\tau\|_{L^{\infty}([0,T];H^{\alpha}(\Lambda))}
\lesssim$$
$$\|\int_{0}^{t} e^{i(t-\tau)\Delta} \big((|u|^2u)(\tau+\kappa)-(|u|^2u)(\tau)\big) d\tau\|_{X^{\alpha}([0,T])}$$
which by Lemma \ref{trilinearlemma} is:
\begin{equation}
\label{eq:continuity3}
\lesssim \|u\|_{X^{\alpha}([0,T])}^2 \|u(\cdot+\kappa)-u\|_{X^{\alpha}([0,T])}.
\end{equation}
Here $\cdot + \kappa$ denotes translation in time. We now show that:
\begin{equation}
\label{eq:continuity4}
\|u(\cdot+\kappa)-u\|_{X^{\alpha}([0,T])} \rightarrow 0,\,\mbox{as}\,\kappa \rightarrow 0.
\end{equation}
Let us fix $v \in Y^{-\alpha}([0,T])$ with $\|v\|_{Y^{-\alpha}([0,T])} \leq 1$.
We note that:
$$\big|\int_{0}^{T} \int_{\Lambda} (u(t+\kappa,x)-u(t,x)) \overline{v(t,x)} dx dt \big|$$
which by using the Cauchy-Schwarz inequality, H\"{o}lder's inequality, and $(\ref{eq:LinftyY})$ is:
$$\lesssim \|u(\cdot+\kappa)-u\|_{L^{\infty}([0,T];H^{\alpha}(\Lambda))}
\|v\|_{L^1([0,T];H^{-\alpha}(\Lambda))}
\leq T \|u(\cdot+\kappa)-u\|_{L^{\infty}([0,T];H^{\alpha}(\Lambda))}
\|v\|_{L^{\infty}([0,T];H^{-\alpha}(\Lambda))}$$
$$\lesssim T \|u(\cdot+\kappa)-u\|_{L^{\infty}([0,T];H^{\alpha}(\Lambda))}
\|v\|_{Y^{-\alpha}([0,T])}$$
$$\leq T \|u(\cdot+\kappa)-u\|_{L^{\infty}([0,T];H^{\alpha}(\Lambda))} \rightarrow 0,\,\mbox{as}\,\kappa \rightarrow 0.$$
since, by assumption $u \in C_u([0,T];H^{-\alpha}(\Lambda))$.
$(\ref{eq:continuity4})$ now follows by $(\ref{eq:duality})$.
\footnote{Strictly speaking, here we are also using the fact that $u$ can be written as the restriction to $[0,T]$ of a function in $C_u(J;H^{\alpha}(\Lambda))$ for some interval $J$ which properly contains $[0,T]$, for example by letting it be equal to the constant at the appropriate endpoints outside of $[0,T]$.}
Combining $(\ref{eq:continuity2})$, $(\ref{eq:continuity3})$, and $(\ref{eq:continuity4})$, it follows that $\|NL(u)(t+\kappa)-NL(u)(t)\|_{L^{\infty}([0,T];H^{\alpha}(\Lambda)} \rightarrow 0$ as $\kappa \rightarrow 0+$. An analogous proof shows that $\|NL(u)(t+\kappa)-NL(u)(t)\|_{L^{\infty}([0,T];H^{\alpha}(\Lambda)} \rightarrow 0$ as $\kappa \rightarrow 0-$. We use this fact together with $(\ref{eq:continuity1})$ to deduce that:
\begin{equation}
\label{eq:F2}
F:D_{\delta} \rightarrow C_u([0,T];H^{\alpha}(\Lambda)).
\end{equation}
From $(\ref{eq:F1})$ and $(\ref{eq:F2})$, it follows that $F$ maps $D_{\delta}$ into itself.
Furthermore:
$$\|NL(u)-NL(v)\|_{X^{\alpha}([0,T])} \leq c (\|u\|_{X^{\alpha}([0,T])}+\|v\|_{X^{\alpha}([0,T])})^2\|u-v\|_{X^{\alpha}([0,T])} \leq c 4\delta^2 \|u-v\|_{X^{\alpha}([0,T])}$$
$$\leq 4c \cdot \frac{1}{8c} \|u-v\|_{X^{\alpha}([0,T])}
\leq \frac{1}{2}\|u-v\|_{X^{\alpha}([0,T])}.$$
Hence, we obtain a fixed point for $F(u)$ in $D_{\delta}$.
\\
\\
\textbf{2: Large data:}
\\
\\
Let $r>0$ and $N \geq 1$ be given. We suppose that $0<\epsilon \leq r_0 \leq r$ and $0 <\delta \leq R$. Here, $r_0>0$ is a fixed, small constant, and $\epsilon, \delta$ are as before.
We consider the sets:
$$B_{\epsilon,r}:=\{f \in H^{\alpha}(\Lambda); \|f_{>N}\|_{H^{\alpha}(\Lambda)} \leq \epsilon, \|f\|_{H^{\alpha}(\Lambda)} \leq r\}$$
and
$$D_{\delta,R,T}:=\{u \in X^{\alpha}([0,T]) \cap C_u([0,T];H^{\alpha}(\Lambda)); \|u_{>N}\|_{X^{\alpha}([0,T])} \leq \delta, \|u\|_{X^{\alpha}([0,T])} \leq R\}.$$
Let us consider $f \in B_{\epsilon,r}$. Then:
We note that $D_{\delta,R,T}$ is a closed subset of a Banach space and hence is also a Banach space.
$$\|[L(f)+NL(u)]_{>N}\|_{X^{\alpha}([0,T])}
\leq \|[L(f)]_{>N}\|_{X^{\alpha}([0,T])} + \|[NL(u)]_{>N}\|_{X^{\alpha}([0,T])}$$
\begin{equation}
\label{eq:highfreqpart}
\leq \epsilon + \|[NL(u)]_{>N}\|_{X^{\alpha}([0,T])}.
\end{equation}
Here, we used $(\ref{eq:linearestimate})$.

We write $|u|^2u=|u_{\leq N}+u_{>N}|^2(u_{\leq N}+u_{>N})$, and so:
$NL(u)=NL_1(u_{\leq N},u_{>N})+NL_2(u_{\leq N},u_{>N})$, where $NL_1$ is at least quadratic in $u_{>N}$, and $NL_2$ is quadratic in $u_{\leq N}$.
By using Lemma \ref{trilinearlemma}, it follows that:
\begin{equation}
\label{eq:NL1}
\|NL_1(u_{\leq N},u_{>N})\|_{X^{\alpha}([0,T])} \leq c \delta^2 R.
\end{equation}
In estimating $NL_2$, we use $(\ref{eq:Duhamelbound})$ and we use the Fractional Leibniz rule to consider separately when the $\alpha$ derivatives fall on the high frequencies and on the low frequencies. From H\"{o}lder's inequality, we deduce that:
$$\|NL_2(u_{\leq N},u_{>N})\|_{X^{\alpha}([0,T])} \leq c_1 \|u\|_{L^{\infty}([0,T];H^{\alpha}(\Lambda))} \|u_{\leq N}\|_{L^2([0,T];L^{\infty}(\Lambda))}^2$$
$$+c_1 N^{\alpha} \|u\|_{L^{\infty}([0,T];L^6(\Lambda))} \|u_{\leq N}\|_{L^2([0,T];L^6(\Lambda))}^2.$$
Let us note that by the Cauchy-Schwarz inequality, we obtain:
\begin{equation}
\label{eq:Sobolevembedding1}
\|g_M\|_{L^{\infty}(\Lambda)} \lesssim M^{\frac{3}{2}}\|g_M\|_{L^2(\Lambda)} \sim M^{\frac{1}{2}} \|g_M\|_{H^1(\Lambda)}.
\end{equation}
Moreover, by using the fact that $\|g_M\|_{L^6(\Lambda)} \leq \|g_M\|_{L^2(\Lambda)}^{\frac{1}{3}} \|g_M\|_{L^{\infty}(\Lambda)}^{\frac{2}{3}}$, it follows that:
\begin{equation}
\label{eq:Sobolevembedding2}
\|g_M\|_{L^6(\Lambda)} \lesssim M \|g_M\|_{L^2(\Lambda)} \sim \|g_M\|_{H^1(\Lambda)}.
\end{equation}
We use the Cauchy-Schwarz inequality in time and $(\ref{eq:Sobolevembedding1})$ to deduce that:
$$\|u_{\leq N}\|_{L^2([0,T];L^{\infty}(\Lambda))} \leq T^{\frac{1}{2}} \|u_{\leq N}\|_{L^{\infty}([0,T];L^{\infty}(\Lambda))} \lesssim T^{\frac{1}{2}} \sum_{M \in 2^{\mathbb{N}}; M \leq N} \|u_{M}\|_{L^{\infty}([0,T];L^{\infty}(\Lambda))}$$
\begin{equation}
\label{eq:lowfrequency2}
\lesssim T^{\frac{1}{2}} \sum_{M \in 2^{\mathbb{N}}; M \leq N} M^{\frac{1}{2}} \|u_M\|_{L^{\infty}([0,T];H^1(\Lambda))} \lesssim T^{\frac{1}{2}}N^{\frac{1}{2}}\|u_{\leq N}\|_{X^1([0,T])}.
\end{equation}
For the last inequality, we use the fact that $\sum_{M \in 2^{\mathbb{N}}; M \leq N} M^{\frac{1}{2}} \lesssim N^{\frac{1}{2}}.$
Similarly, we can use the Cauchy-Schwarz inequality in time and $(\ref{eq:Sobolevembedding2})$ to deduce that:
$$\|u_{\leq N}\|_{L^2([0,T];L^6(\Lambda))} \leq T^{\frac{1}{2}} \|u_{\leq N}\|_{L^{\infty}([0,T];L^6(\Lambda))}
\leq T^{\frac{1}{2}} \sum_{M \in 2^{\mathbb{N}}; M \leq N} \|u_{M}\|_{L^{\infty}([0,T];L^6(\Lambda))}
$$
$$\lesssim T^{\frac{1}{2}} \sum_{M \in 2^{\mathbb{N}}; M \leq N} \|u_M\|_{L^{\infty}([0,T];H^1(\Lambda))} \lesssim T^{\frac{1}{2}} \sum_{M \in 2^{\mathbb{N}}; M \leq N} \|u\|_{L^{\infty}([0,T];H^1(\Lambda))}$$
\begin{equation}
\label{eq:lowfrequency1}
\lesssim
T^{\frac{1}{2}} N^{0+} \|u\|_{L^{\infty}([0,T];H^1(\Lambda))} \lesssim T^{\frac{1}{2}}N^{0+} \|u\|_{X^1([0,T])}.
\end{equation}
Here, we have used the fact that the number of dyadic integers $M$ such that $M \leq N$ is $O(log N)=O(N^{0+})$.

On the other hand, by similar arguments:

By using $(\ref{eq:lowfrequency1})$ and $(\ref{eq:lowfrequency2})$, it follows that:
\begin{equation}
\label{eq:NL2}
\|NL_2(u_{\leq N},u_{>N})\|_{X^{\alpha}([0,T])} \leq c_2 N^{\alpha+}TR^3.
\end{equation}
We argue analogously as in the proof of $(\ref{eq:NL1})$ to deduce that:
\begin{equation}
\label{eq:NL1difference}
\|NL_1(u_{\leq N},u_{>N})_N-NL_1(v_{\leq N},v_{>N})\|_{X^{\alpha}([0,T])}
\leq c_3 \delta R \|u-v\|_{X^{\alpha}([0,T])}.
\end{equation}
Here, it is crucial to use the fact that:
$u_{\geq N}-v_{\geq N}=(u-v)_{\geq N}$ and $u_{>N}-v_{>N}=(u-v)_{>N}$.
Similarly, we use an analogous argument as for $(\ref{eq:NL2})$ to deduce that:
\begin{equation}
\label{eq:NL1difference}
\|NL_2(u_{\leq N},u_{>N})_N-NL_2(v_{\leq N},v_{>N})\|_{X^{\alpha}([0,T])}
\leq c_4 N^{\alpha+}TR^2 \|u-v\|_{X^{\alpha}([0,T])}.
\end{equation}
Let us now take $C$ to be the maximum of $c,c_j$. For fixed $f \in B_{\epsilon,r}$, we consider the map $F(u):=L(f)+NL(u)$.
For fixed $r,N>0$, we take the following choice of parameters:
\begin{equation}
\label{eq:parameterchoice}
R:=2r, \delta:=\frac{1}{A_1R}, \epsilon:=\frac{\delta}{2}, T:=\frac{\delta}{A_2N^{\alpha+}R^3}.
\end{equation}
The constants $A_1,A_2>0$ are chosen such that:
\begin{equation}
\label{eq:Ajchoice1}
\frac{C}{A_1}+\frac{C}{A_2} \leq \frac{1}{2}
\end{equation}
\begin{equation}
\label{eq:Ajchoice2}
\frac{C}{4 A_1^2 r_0^2}+\frac{C}{4 A_1 A_2 r_0^2} \leq \frac{1}{2}
\end{equation}
\begin{equation}
\label{eq:Ajchoice3}
\frac{C}{A_1}+\frac{C}{4 A_1 A_2 r_0^2} \leq \frac{1}{2}.
\end{equation}
Here, $C$ and $r_0$ are the positive constants defined above.

By using the previous choice of parameters, it follows that, for $u,v \in B_{\delta,R,T}$, one has:
\begin{enumerate}
\item (The high frequency part of $F(u)$ satisfies the appropriate bound)
$$\|[F(u)]_{>N}\|_{X^{\alpha}([0,T])} \leq \|[L(f)]_{>N}\|_{X^{\alpha}([0,T])}+
\|[NL(u)\|_{X^{\alpha}([0,T])} \leq \epsilon + C\delta^2R + CN^{\alpha+}TR^3 \leq \delta.$$
Here, we used $(\ref{eq:highfreqpart})$, as well as $(\ref{eq:NL1})$ and $(\ref{eq:NL2})$.
\item ($F(u)$ satisfies the appropriate bound)
$$\|F(u)\|_{X^{\alpha}([0,T])} \leq \|L(f)\|_{X^{\alpha}([0,T])} + \|NL(u)\|_{X^{\alpha}([0,T])} \leq r + C\delta^2R + CN^{\alpha+}TR^3 \leq R.$$
\item Arguing as in the small data case, we show that $F$ maps $B_{\delta,R,T}$ into $C_u([0,T];H^{\alpha}(\Lambda))$.
\item ($F$ is a contraction on $B_{\delta,R,T}$)
$$\|F(u)-F(v)\|_{X^{\alpha}([0,T])} = \|NL(u)-NL(v)\|_{X^{\alpha}([0,T])}
\leq (C \delta R + C N^{\alpha+}TR^2)\|u-v\|_{X^{\alpha}([0,T])} \leq $$
$$\leq \frac{1}{2}\|u-v\|_{X^{\alpha}([0,T])}.$$
\end{enumerate}
Therefore, $F:D_{\delta,R,T} \rightarrow D_{\delta,R,T}$ is a contraction, and hence it has a unique fixed point in $D_{\delta,R,T}$, since $D_{\delta,R,T}$ is a Banach space.
\\
\\
\textbf{3: Uniqueness in $X^{\alpha}([0,T]) \cap C_u([0,T];H^{\alpha}(\Lambda))$}
\\
\\
Suppose that $u,v \in X^{\alpha}([0,T]) \cap C_u([0,T];H^{\alpha}(\Lambda))$ solve:
\begin{equation}
\label{eq:NLS1}
\begin{cases}
iu_t + \Delta u = |u|^2u,\,\mbox{on}\, \Lambda \times [0,T]\\
u |_{t=0}=\phi_0 \in H^{\alpha}(\Lambda).
\end{cases}
\end{equation}
\begin{equation}
\label{eq:NLS2}
\begin{cases}
iv_t + \Delta v = |v|^2v,\,\mbox{on}\, \Lambda \times [0,T]\\
u |_{t=0}=\phi_0 \in H^{\alpha}(\Lambda).
\end{cases}
\end{equation}
We let $r \geq r_0>0$ be given. We then define $\epsilon, \delta, R$ as in $(\ref{eq:parameterchoice})$. Furthermore we choose $N \geq 1$ sufficiently large such that $\phi_0 \in B_{\epsilon,r}$ and $u,v \in D_{\delta,R}$.
We then make $T$ possibly smaller such that $T \leq \frac{\delta}{A_2N^{\alpha+}R^3}$. By Part 2, it follows that $u=v$ on the time interval $[0,T]$. Uniqueness now follows.
\\
\\
\textbf{4: Construction of $T$ as a function of time:}
\\
\\
We note from earlier that there exists $\tilde{T}_1: \mathbb{R} \rightarrow \mathbb{R}$, which is positive such that for all $t \in \mathbb{R}$, one has:
\begin{equation}
\label{eq:Ttilde1}
\|u\|_{X^{\alpha}([t-\tilde{T}_1(t),t+\tilde{T}_1(t)])} \leq 2\|u(t)\|_{H^{\alpha}(\Lambda)}.
\end{equation}
By construction, $\tilde{T}$ depends on the frequency parameter $N$ given in Step 2, which we cannot control by the energy. We also note that $\tilde{T}_1$ is not necessarily continuous.

We first observe that we can modify $\tilde{T}_1$ so that, locally in $t$, it is uniformly bounded away from zero and such that an inequality similar to $(\ref{eq:Ttilde1})$ holds. More precisely, we show that there exists $\tilde{T}_2 : \mathbb{R} \rightarrow \mathbb{R}$ such that, for all $t \in \mathbb{R}$:
\begin{equation}
\label{eq:Ttilde2}
\|u\|_{X^{\alpha}([t-\tilde{T}_2(t),t+\tilde{T}_2(t)])} \leq 3 \|u(t)\|_{H^{\alpha}(\Lambda)}.
\end{equation}
\begin{equation}
\label{eq:deltar}
\mbox{There exist $\delta(t), r(t)>0$ such that $\tilde{T}_2 > r(t)$ on $[t-\delta(t),t+\delta(t)]$}.
\end{equation}
We construct $\tilde{T}_2$ from $\tilde{T}_1$.
If $\tilde{T}_1$ has no zero limit points, i.e. there does not exist $t_{\infty} \in \mathbb{R}$ and a sequence $t_n \rightarrow t_{\infty}$ in $\mathbb{R}$ such that $\tilde{T}_1(t_n) \rightarrow 0$, we can just take $\tilde{T}_2:=\tilde{T}_1$. If, on the other hand, there exists such a $t_{\infty}$, we note by the continuity of $t \mapsto \|u(t)\|_{H^{\alpha}(\Lambda)}$ that there exists $\tau_{\infty}>0$ such that:
\begin{equation}
\label{eq:Sobolevcontinuity}
\|u(t)\|_{H^{\alpha}(\Lambda)} \geq \frac{2}{3}\|u(t_{\infty})\|_{H^{\alpha}(\Lambda)}, \mbox{for}\,t \in [t_{\infty}-\tau_{\infty},t_{\infty}+\tau_{\infty}].
\end{equation}
We can choose $\tau_{\infty}$ so that:
\begin{equation}
\label{eq:tau0cond}
\tau_{\infty}<\tilde{T}_1(t_{\infty}).
\end{equation}
By construction of $\tilde{T}_1$, we note that, for $t \in [t_{\infty}-\tau_{\infty},t_{\infty}+\tau_{\infty}]$, one has:
$$[t-(\tilde{T}_1(t_{\infty})-\tau_{\infty}),t+(\tilde{T}_1(t_{\infty})-\tau_{\infty})] \subseteq [t_{\infty}-\tilde{T}_1(t_{\infty}),t_{\infty}+\tilde{T}_1(t_{\infty})],$$ and hence:
$$\|u\|_{X^{\alpha}([t-(\tilde{T}_1(t_{\infty})-\tau_{\infty}),t+(\tilde{T}_1(t_{\infty})-\tau_{\infty})])}
\leq \|u\|_{X^{\alpha}([t_{\infty}-\tilde{T}_1(t_{\infty}),t_{\infty}+\tilde{T}_1(t_{\infty})])}$$
$$\leq 2 \|u(t_{\infty})\|_{H^{\alpha}(\Lambda)} \leq 3 \|u(t)\|_{H^{\alpha}(\Lambda)}.$$
Thus, for $t \in [t_{\infty}-\tau_{\infty},t_{\infty}+\tau_{\infty}]$, we can take:
$$\tilde{T}_2(t):=\max \{\tilde{T}_1(t),\tilde{T}_1(t_{\infty})-\tau_{\infty}\}.$$
By construction, $\tilde{T}_2$ does not have a zero limit at $t_{\infty}$. If $\tilde{T}_2$ constructed this way has no limit points, we are done. Otherwise, we repeat the procedure. We note that the function $\tilde{T}_2$ we obtain in the end might take the value $+\infty$, hence we have to replace $\tilde{T}_2$ by $\min \{\tilde{T}_2,1\}$. Such a function $\tilde{T}_2: \mathbb{R} \rightarrow \mathbb{R}$ will be positive and will satisfy $(\ref{eq:Ttilde2})$ and $(\ref{eq:deltar})$. 
The function $\tilde{T}_2$ is not necessarily continuous.

We now construct the function $T$. Let us construct $T$ on $[0,+\infty)$. An analogous argument works on $(-\infty,0]$. Let's start at time $t=0$. By continuity of $t \mapsto \|u(t)\|_{H^{\alpha}(\Lambda)}$, there exists $\kappa_0 \in (0,\frac{1}{2}\tilde{T}_2(0)]$ such that, for all $t \in [0,\kappa_0]$:
$$\|u(t)\|_{H^{\alpha}(\Lambda)} \geq \frac{1}{2} \|u(0)\|_{H^{\alpha}(\Lambda)}.$$
Consequently, for $t \in [0,\kappa_0]$:
$$\|u\|_{X^{\alpha}([t-(\tilde{T}_2(0)-t),t+(\tilde{T}_2(0)-t)])} \leq \|u\|_{X^{\alpha}([-\tilde{T}_2(0),\tilde{T}_2(0)]} \leq 3 \|u(0)\|_{H^{\alpha}(\Lambda)} \leq 6 \|u(t)\|_{H^{\alpha}(\Lambda)}.$$
We want:
\begin{equation}
\label{eq:T0}
0<T(t) \leq \tilde{T}_2(0) - t,\, \mbox{for}\,t \in [0,\kappa_0].
\end{equation}
We continue the construction inductively. Namely, given $\kappa_n$, for some $n \geq 0$, we construct a maximal $\kappa_{n+1}$ which satisfies:
\begin{equation}
\label{eq:kappaI}
\kappa_{n+1} \leq \kappa_n + \tilde{T}_2(\kappa_n).
\end{equation}
and
\begin{equation}
\label{eq:kappaII}
\|u(t)\|_{H^{\alpha}(\Lambda)} \geq \frac{1}{2}\|u(\kappa_n)\|_{H^{\alpha}(\Lambda)}, \mbox{for}\, t \in [\kappa_n,\kappa_{n+1}].
\end{equation}
An analogous argument as above shows that, for $t \in [\kappa_n,\kappa_{n+1}]$, one has:
$$\|u\|_{X^{\alpha}([t-(\tilde{T}_2(\kappa_n)-(t-\kappa_n)),t+(\tilde{T}_2(\kappa_n)-(t-\kappa_n))]} \leq 6 \|u(t)\|_{H^{\alpha}(\Lambda)}.$$
Hence, we want:
\begin{equation}
\label{eq:Tn}
0<T(t) \leq \tilde{T}_2(\kappa_n) - (t-\kappa_n),\, \mbox{for}\,t \in [\kappa_n,\kappa_{n+1}].
\end{equation}
Let $\kappa_{\infty}:=\lim \kappa_n$. We can construct $T:[0,\kappa_{\infty}) \rightarrow \mathbb{R}$ which is positive and continuous that satisfies $(\ref{eq:T0})$ and $(\ref{eq:Tn})$. Namely, we just note that the upper bounds are piecewise linear positive functions of slope $-1$. We want to argue that $\kappa_{\infty}=+\infty$.

We argue by contradiction. In other words, we suppose that $\kappa_n \rightarrow \kappa_{\infty} \in \mathbb{R}$. We know by construction that $\tilde{T}_2$ is bounded away from zero near $\kappa_{\infty}$. Hence, for sufficiently large $n$, one has:
$$\kappa_{n+1} < \kappa_n + \frac{1}{2}\tilde{T}_2(\kappa_{n+1}).$$
Thus, by $(\ref{eq:kappaI})$ and $(\ref{eq:kappaII})$, it follows that, for $n$ sufficiently large, one has:
$$\|u(\kappa_{n+1})\|_{H^{\alpha}(\Lambda)}=\frac{1}{2}\|u(\kappa_n)\|_{H^{\alpha}(\Lambda)}.$$
Namely, we note that equality cannot hold in $(\ref{eq:kappaI})$.
Consequently, $\|u(\kappa_{\infty})\|_{H^{\alpha}(\Lambda)}=0$, hence $u(\kappa_{\infty})=0$. It follows by uniqueness of solutions to the nonlinear Schr\"{o}dinger equation that $u=0$. This is a contradiction. Hence, $T$ can indeed be defined on all of $[0,+\infty)$.
We finally note that $(\ref{eq:lwpa})$ follows from $(\ref{eq:lwp})$.
\end{proof}

\end{document}